\documentclass[11pt]{amsart}
\usepackage{a4wide,amsfonts,graphicx,
amssymb,stmaryrd,amscd,amsmath,latexsym,amsbsy,amsthm}
\usepackage[all,knot,poly]{xy}
\usepackage[all,knot,poly]{xy}
\usepackage{color}
\usepackage{hyperref}
\usepackage{ifpdf}
\ifpdf %if using pdfLaTeX in PDF mode
  \usepackage{graphicx}
  \DeclareGraphicsExtensions{.pdf,.png,.mps}
  \usepackage{pgf}
  \usepackage{tikz}
  \usetikzlibrary{matrix,arrows,decorations.pathmorphing}
   \usepackage{tikz-cd}
\else %if using LaTeX or pdfLaTeX in DVI mode
  \usepackage{graphicx}
  \DeclareGraphicsExtensions{.eps,.bmp}
  \usepackage{pgf}
  \usepackage{tikz}
  \usetikzlibrary{matrix,arrows,decorations.pathmorphing}
  \usepackage{tikz-cd}
  \usepackage{pstricks}%variant: \usepackage{pst-all}
\fi
\usepackage{epic,bez123}
\usepackage{floatflt}% package for floatingfigure environment
\usepackage{wrapfig}% package for wrapfigure environment
\usepackage{hyperref}
\hypersetup{
  colorlinks   = true,    % Colours links instead of ugly boxes
  urlcolor     = blue,    % Colour for external hyperlinks
  linkcolor    = blue,    % Colour of internal links
  citecolor    = red      % Colour of citations
}

\newtheorem{thm}{Theorem}[section]
\newtheorem{theorem}[thm]{Theorem}

\newtheorem{cor}[thm]{Corollary}
\newtheorem{lem}[thm]{Lemma}
\newtheorem{prop}[thm]{Proposition}
\newtheorem{defn}[thm]{Definition}

\newtheorem{rem}[thm]{Remark}

\newcommand{\mb}{\mathbf}

\newcommand{\mc}{\mathcal}
\newcommand{\mf}{\mathfrak}

\newcommand{\de}{\textup{fdeg}}

\newcommand{\Lb}{\mathbf{L}}

\newcommand{\ep}{\epsilon}

\begin{document}
\title[Spectral correspondences]
{Spectral correspondences for affine Hecke algebras}
\author{Eric Opdam}
\address{Korteweg de Vries Institute for Mathematics\\
University of Amsterdam\\
P.O. Box 94248\\
1090 GE Amsterdam\\ 
The Netherlands\\
email: e.m.opdam@uva.nl}
\date{\today}
\thanks{This research was supported  by 
ERC-advanced grant no. 268105. It is a pleasure to thank 
Joseph Bernstein, Dan Ciubotaru and Mark Reeder for useful discussions and 
comments.}
\begin{abstract}We introduce the notion of spectral transfer morphisms 
between normalized affine Hecke algebras, and show that such morphisms 
induce 
spectral measure preserving correspondences on the level of the tempered 
spectra of the affine Hecke algebras involved. We define a partial ordering 
on the set of isomorphism classes of normalized affine Hecke algebras, which 
plays an important role for the Langlands parameters of Lusztig's unipotent 
representations.  
\end{abstract}
\maketitle
\tableofcontents
\section{Introduction}\label{sec:intro}
In this paper we introduce the notion of ``spectral transfer morphisms" between 
normalized affine Hecke algebras. 
We prove that such a spectral transfer morphism $\phi:\mc{H}_1\leadsto\mc{H}_2$ induces 
a finite morphism $\phi_Z$ from the affine variety defined by the center $\mc{Z}_1$ of $\mc{H}_1$ 
to the affine variety defined by the center $\mc{Z}_2$ of $\mc{H}_2$.
Using this map we will construct a correspondence between the tempered spectra of the affine 
Hecke 
algebras $\mc{H}_1$ and $\mc{H}_2$ which respects the connected components in the tempered 
spectra  of $\mc{H}_1$ and $\mc{H}_2$ and which respects the spectral measures 
up to rational constants and finite maps. By this we mean the following: If 
$\mf{S}_1$ is a connected component of the tempered spectrum of $\mc{H}_1$, then 
there exists a connected component $\mf{S}_2$ of the tempered spectrum of $\mc{H}_2$, 
a connected space $\mf{S}_{12}$ with 
finite surjective continuous maps $p_1:\mf{S}_{12}\to\mf{S}_1$ and $p_2:\mf{S}_{12}\to\mf{S}_2$, 
and a positive measure $\nu_{12}$ on $\mf{S}_{12}$ whose push forward to $\mf{S}_i$
yields the spectral measure of $\mc{H}_i$ restricted to $\mf{S}_i$, up to a rational constant 
factor depending on $\mf{S}_1$ and $\mf{S}_2$.

In \cite{Opd4} we apply this notion to unipotent affine Hecke algebras, i.e. to the normalized affine 
Hecke algebras associated to the unipotent types of inner forms of a given absolutely simple quasi split 
linear algebraic group $G$ defined over a non-archimedian local field, all normalized in 
an appropriate way. We will prove in that paper 
that all unipotent normalized affine Hecke algebras admit an 
essentially unique spectral transfer morphism to the Iwahori Hecke algebra of $G$, and that 
these maps are exactly the same as Lusztig's geometric-arithmetic correspondences \cite{Lu4,Lu6}. 
One could interpret this remarkable fact as an amplification of Mark Reeder's 
work on formal degrees of unipotent discrete series representations and L-packets for split 
exceptional groups \cite{Re} and small rank classical groups \cite{Re0}.

Thus Lusztig's geometric-arithmetic correspondences can be recovered 
from the perspective of harmonic analysis alone, and this is in itself a 
motivation for studying this notion of spectral transfer morphisms. 
As an application we will give a proof in \cite{Opd4} of a conjecture of K. Hiraga, A. Ichino 
and T. Ikeda \cite[Conjecture 1.4]{HII} on the formal degree of a discrete series representation $\pi$ 
of $G(k)$,  where $G$ is a connected reductive group over a local field $k$, for unipotent 
discrete series representations.
This conjecture expresses the formal degree in terms of the adjoint gamma 
factor of the conjectural Langlands parameters of $\pi$. 

It is well known \cite{Mo1,Mo2,Lu4,Lu6} that, if $k$ is non-archimedean,  
the unipotent Bernstein components \cite{Ber} of the category of smooth representations of $G(k)$ 
are Morita equivalent to the module category of an affine Hecke algebra, in general with 
unequal parameters. 
These Morita equivalences 
are realized by a type in the sense of \cite{BK1}, and in general, such Morita equivalences are 
known to respect the notion of temperedness and are Plancherel measure 
preserving \cite{BHK}. For more general Bernstein components one still expects variations 
of such descriptions to hold (see e.g. \cite{Yu}). An alternative apporoach to such results, suggested 
by J. Bernstein and established for classical groups by Heiermann \cite{Hei} , is to directly establish 
an isomorphism between the opposite endomorphism ring of a 
progenerator $\pi_\mc{O}$ of a Bernstein component $\mc{O}$ and an affine Hecke algebra (or a small 
variations thereof). 
In the latter situation there is not yet a complete understanding of the correspondences 
on the level of harmonic analysis.
In both approaches,  
the corresponding Bernstein variety is equipped with a rational function 
which has its origin in harmonic analysis: the $\mu$-function. By the work of Heiermann 
\cite{Hei1,Hei2}, it is known that the central support of the set of tempered representation in a 
Bernstein component can be completely described in terms of this $\mu$-function in analogy 
to the results of \cite{Opd1,Opd3} for affine Hecke algebras.

In recent years, from various perspectives, the representation 
theory and harmonic analysis of unequal parameter affine Hecke 
algebras has seen progress \cite{HOH,Opd1,OpdSol2,Sol2,CK,CKK,Lu6,COT,CiuOpd1,CiuOpd2}.  
The spectral transfer morphisms introduced in the present paper are defined in entirely 
terms of the $\mu$-functions of the affine Hecke algebras, using the properties of its poles \cite{Opd3}. 
It is therefore likely that spectral transfer morphism can also be defined for more general Bernstein 
components than the unipotent ones, without reference to an affine Hecke algebra.
In any case, in view of the natural role of (unequal parameter) affine Hecke algebras in the representation 
theory of reductive $p$-adic groups sketched above, it is a natural quest to understand how their irreducible 
spectra are combined into stable packets. 
Using the notion of spectral transfer morphisms we will show in this paper (together with 
\cite{Opd4} and \cite{FO}) that if such partitioning 
exists for tempered unipotent representations, then it is essentially uniquely determined by certain 
natural algebraic relations between 
the $\mu$-functions of the unipotent affine Hecke algebras, and that this yields  
the same partitioning as provided by Lusztig \cite{Lu4,Lu6}.

We will introduce a natural \emph{partial ordering} between the spectral isogeny classes of normalized affine Hecke algebras 
with respect to spectral transfer maps. For the normalized affine Hecke algebras which appear in the study of 
Lusztig's unipotent representations, we will in fact see that the spectral isogeny classes in this sense are the same as 
the isomorphism classes. This ordering gives a special role to the Iwahori Hecke algebra $\mc{H}^{IM}(G)$ of 
an unramified quasi-split semisimple group $G$ (defined over a nonarchimedean local field): 
Among the isomorphism classes of the unipotent affine Hecke 
algebras of inner forms of $G$, (the isomorphism class of) \emph{$\mc{H}^{IM}(G)$ is the least element.}
In view of the remarks above, one may expect that this ordering plays a role more generally for 
components of the Bernstein variety of $G$.

The aim of the present paper is the study of the Plancherel measure preserving correspondences associated 
to a spectral transfer morphism, using our knowledge \cite{Opd1,OpdSol2,Opd3} of the properties 
of the $\mu$-function and its relations to the Plancherel measure. 
\section{Affine Hecke algebras}\label{sub:affhecke}
We shall mainly follow the conventions of \cite{OpdSol2}. 
However, we will restrict ourselves in the present paper to (generic) affine Hecke algebras whose 
parameters are unequal but which are all integral powers of a fixed invertible indeterminate $v$.
Therefore certain notions and definitions of \cite{OpdSol2} require some small adaptations
(e.g. the notion of the ``spectral diagram" of the Hecke algebra). 
\subsection{Basic notions and conventions}
\subsubsection{Root data and affine Weyl groups}\label{sub:rootdata}
Recall the notion of a based, semisimple root datum
$\mc{R}=(X,R_0,Y,R_0^\vee,F_0)$, consisting of a semisimple root datum
with a given basis $F_0\subset R_0$ of simple roots.
Let $W_0=W(R_0)$ be the Weyl
group of the root system $R_0\subset X$, acting on the lattice $X$, and
let $W=W(\mc{R}):=X\rtimes W_0$ denote the associated extended affine Weyl group.
We equip the real vector space $V^*:=X\otimes \mathbb{R}$ with a Euclidian
structure which is $W_0$-invariant.
\subsubsection{Affine root system and the affine Coxeter group $W^a$}
\label{par:affroot}
The set $R=R_0^\vee+\mathbb{Z}\subset Y+\mathbb{Z}$ is an affine root
system whose associated affine Coxeter group $W^a=W(R)\approx Q(R_0)\rtimes W_0$
acts naturally on $X$. Then $W^a$ is a normal subgroup of $W$.
The basis $F_0$ of simple roots of $R_0$ determines a corresponding set
of simple affine roots denoted by $F\subset R$.
If $S\subset W^a$ is the corresponding set of affine simple reflections
then $(W^a,S)$ is an affine Coxeter group with Coxeter generators $S$.
The Coxeter group $(W^a,S)$ is determined by the Coxeter graph
$X(F)$ of $(W^a,F)$. Since we are assuming that $(W^a,S)$ is an
affine Coxeter group (i.e. all connected components of $X(F)$ are
irreducible affine Coxeter graphs) it is clear by the classification
of irreducible affine Coxeter graphs that $X(F)$ determines a unique
untwisted affine Dynkin diagram $D(F)$ (whose nodes are identified
with the set $F$).
\subsubsection{$W$ as extended Coxeter group}\label{sub:extcox}
Let $C\subset V^*$ be the fundamental alcove
$C=\{\lambda\in V^*\mid \forall\, a\in F: a(\lambda)>0\}$.
The closure $\overline{C}$ of $C$ is a fundamental domain for the
action of $W^a$ on $V^*$. Then $S$ is the set of isometric
affine hyperplane reflections in the faces of codimension $1$ of
$C$. We introduce the finite subgroup
$\Omega_X=\{\omega\in W\mid \omega(C)=C\}\subset W$.
Obviously $\Omega_X$ acts on $D(F)$ by diagram automorphisms.
Hence the natural action of $\Omega_X$ on $S$ extends to an
action of $\Omega_X$ on $W^a$ by automorphisms. It is clear that
we have $W=W^a\rtimes \Omega_X$. It follows easily that
$\Omega_X\approx W/W^a$ is isomorphic to the finitely generated
abelian group $\Omega_X=X/Q(R_0)$, where $Q(R_0)$ denotes the root lattice
of the root system $R_0$.
\subsubsection{Length function} The canonical length function $l$
on the Coxeter group $(W^a,S)$ can be extended uniquely to a length
function $l$ on the extended affine Coxeter group $W$ such that
$\Omega_X$ is the set of elements of length zero. Thus $W$ is generated
be its distinguished set $S$ of simple generators (of length $1$)
and its subgroup group $\Omega_X$ of elements of length zero.
\subsubsection{Generic affine Hecke algebra}\label{sub:genaffhecke}
We introduce the ring
$\Lambda=\mathbb{C}[v(s)^{\pm 1};s\in S]$ of Laurent polynomials
in invertible commuting indeterminates $v(s)$ (with $s\in S$) subject to the
relations $v(s)=v(s^\prime)$ iff $s$ and $s^\prime$ are conjugate
in $W$. The maximal spectrum of $\Lambda$ is a complex
algebraic torus denoted by $\mc{Q}_c$ in \cite{OpdSol2}
(we reserve the notation $\mc{Q}$ for the subset
$\mc{Q}\subset\mc{Q}_c$ of positive (or infinitesimally
real) points of $\mc{Q}_c$).

Let $\Lb=\mathbb{C}[v^{\pm 1}]$ denote be the ring of regular
functions on $\mathbb{C}^\times$, and let
$m:\mathbb{C}^\times\to\mc{Q}_c$ be a cocharacter. Then $m$ is
given by a collection of integers $m_S(s)\in\mathbb{Z}$ ($s\in S$)
defined on the set of $W$-conjugacy classes meeting $S$
by $m^*(v(s))=v^{m_S(s)}$ for all $s\in S$. Thus the collection
of such cocharacters $m$ is in natural bijection with the set of
$W$-invariant functions $m_R$ on $R$ with values in $\mathbb{Z}$,
where $m$ and $m_R$ correspond iff $m_R(a)=m(s_a)$ for all $a\in F$.
We consider $\Lb$ as a $\Lambda$ algebra via the
homomorphism $m^*:\Lambda\to\Lb$.

In \cite{OpdSol2} the generic affine Hecke algebra $\mc{H}_\Lambda(\mc{R})$
was defined (see below).
In the present paper the basic objects of study are Hecke
algebras with coefficient ring $\Lb$, obtained by
specialization from $\Lambda$ via a cocharacter $m$.
\begin{defn}\label{def:def}
Let $\mc{R}$ be a based affine root datum and
let $m:\mathbb{C}^\times\to\mc{Q}_c$ be a cocharacter as above.
We associated with these data the generic affine Hecke algebra
$\mc{H}(\mc{R},m)$ over $\mb{L}$ as follows.
Recall
that $\mc{H}_\Lambda(\mc{R})$ is the unique unital, associative,
$\Lambda$-algebra with distinguished $\Lambda$-basis $\{N_w\}_{w\in W}$
parametrized by $w\in W$, satisfying the relations
\begin{enumerate}
\item[(i)] $N_wN_{w^\prime}=N_{ww^\prime}$ for all $w,w^\prime\in W$
such that $l(ww^\prime)=l(w)+l(w^\prime)$.
\item[(ii)] $(N_s-v(s))(N_s+v(s)^{-1})=0$ for all $s\in S$.
\end{enumerate}
The generic affine Hecke algebra with parameter $m$
equals $\mc{H}(\mc{R},m):=\mc{H}_\Lambda(\mc{R})\otimes_{m^*} \Lb$

The set of distinguished generators $\{N_s\otimes 1\mid s\in S\}$
is considered an intrinsic part of the structure of $\mc{H}=\mc{H}(\mc{R},m)$,
as well as the set of generators $\{N_\omega\otimes 1\mid\omega\in\Omega_X\}$ of length $0$.
Finally we consider the subset $S_0\subset S$ as part of the structure of $\mc{H}$. 
\end{defn}
\begin{rem}
In particular the $\mb{L}$-basis
$\{N_w\otimes 1\mid w\in W\}$ is a distinguished bases of $\mc{H}$ which is an
integral part of its structure.
We will from now on write $N_w$ again instead of $N_w\otimes 1$.
\end{rem}
The following result is well known and its easy proof is left to the reader. 
\begin{prop}\label{prop:Riscan} The pair of data $(\mc{R},m)$ is canonically
determined by $\mc{H}$.
\end{prop}
We include the following simple result on root data, for later reference.
\begin{prop}\label{prop:Rmcan} 
Suppose that $\mc{R}=(X,R_0,Y,R_0^\vee)$ is a root datum, and 
that $R_0'\subset R_0$ is an irreducible component of $R_0$. 
Put $X':=X\cap \mathbb{R}R_0'$, and $Y':=Y\cap \mathbb{R}R_0'^\vee$. 
Assume that either $(i)$: $X'=P(R_0')$ (the weight lattice of $R_0'$), or $(ii)$:  
$2Y\cap R_0'^\vee\not=\emptyset$.
Then the root datum $\mc{R}':=(X',R_0',Y',(R_0')^\vee)$ is a direct 
summand of $\mc{R}$. In the case $(ii)$,  $\mc{R}'$ has type $\textup{C}_n^{(1)}$
(i.e. $R_0'$ is of type $\textup{B}_n$, and $X'=Q(R_0')$).
\end{prop}
\begin{proof}
If $X'=P(R_0')$ then $X'$ is the complement of the $W_0$-invariant sub-lattice 
$K:=\cap_{\alpha\in R_0'}\textup{Ker}(\alpha^\vee)\subset X$. 
If $2Y\cap R_0'^\vee\not=\emptyset$ then the same argument applies to $\mc{R}_1:=(X,R_1,Y,R_1^\vee)$ 
(where $R_1$ is the set of inmultipliable roots of the non-reduced root system 
$R_{nr}$ defined by $R_{nr}^\vee=R_0^\vee\cup (Y\cap \frac{1}{2}R_0^\vee)$)
since $X'=Q(R_0')=P(R_1')$, showing that $\mc{R}'_1:=(X',R_1',Y',R_1'^\vee)$
is a direct summand of $\mc{R}_1$. But then $\mc{R}'$ is also a direct summand of $\mc{R}$. 
Moreover $2Y'\cap R_0'^\vee\not=\emptyset$. 
This implies that the irreducible root datum $\mc{R}'$ has type 
$\textup{C}_n^{(1)}$.
\end{proof}
\subsubsection{$\mc{H}(\mc{R},m)$ as an extended affine Hecke algebra}
As is well known, we can view $\mc{H}(\mc{R},m)$ as an extended affine Hecke algebra
in the following way. Define $\mc{R}^a=(R_0,Q(R_0),R_0^\vee,P(R_0^\vee),F_0)$
and let $m^a$ be the composition of $m$ with the natural map
$\mc{Q}_c\to\mc{Q}^a_c$. Then we have an algebra isomorphism
$\mc{H}(\mc{R},m)=\mc{H}(\mc{R}^a,m^a)\rtimes\Omega_X$.
\subsubsection{Homomorphisms of affine Hecke algebras} We introduce some
useful classes of $\Lb$-algebra homomorphisms between generic affine Hecke algebras.
A \emph{strict} homomorphism $\lambda:\mc{H}\to\mc{H}^\prime$ between two
generic extended affine Hecke algebras is an algebra homomorphism given by
a homomorphism 
$\lambda^W:(W,S,S_0,\Omega_X)\to(W^\prime,S^\prime,S^\prime_0,\Omega^\prime_{X^\prime})$
between the underlying (extended affine) Coxeter groups (of course, in order to extend to 
the Hecke algebra level, the Hecke parameters of $\mc{H}$ and $\mc{H}'$ should match
accordingly).
An \emph{essentially strict} homomorphism is like a strict homomorphism except that
$\lambda(N_s)=\epsilon(s) N_{\lambda^W(s)}$ for some signatures $\epsilon(s)\in\{\pm 1\}$,
and $\lambda(N_\omega)=\delta(\omega)N_{\lambda^W(\omega)}$ for some character
$\delta\in\Omega^*_X$.
An \emph{admissible} homomorphism is a homomorphism of algebras respecting the 
standard bases of $\mc{H}$ and $\mc{H}'$ in the sense that for all $w\in W$: 
$\lambda(N_w)=\epsilon(w) N_{\lambda^W(w)}$, 
where $\epsilon$ is a linear character of $W$ and $\lambda^W: W\to W^\prime$ a group
homomorphism.
If $\mc{H}=\mc{H}(\mc{R},m)$ we speak of strict, essentially strict
and admissible automorphisms of $\mc{H}$ (even though 
an ``admissible automorphism" of $\mc{H}$ may not 
respect all intrinsic structures of $\mc{H}$, in which case it is strictly speaking  
an isomorphism between different Hecke algebras).
\subsubsection{Notational conventions}
We use the following notational conventions for change of coefficients.
If $\mc{A}$ is an $\Lb$-algebra and
$\mb{M}$ is a unital, commutative $\Lb$-algebra, then
$\mc{A}_\mb{M}:=\mc{A}\otimes_{\Lb}\mb{M}$ denotes the extension of
scalars from $\Lb$ to $\mb{M}$. Given $\mb{v}\in\mathbb{C}^\times$
we use the notation $\mc{A}_\mb{v}$ as a shorthand for the
$\mathbb{C}$-algebra $\mc{A}_{\mathbb{C}_\mb{v}}$, where
$\mathbb{C}_\mb{v}$ denotes the residue field of $\Lb$ at
$\mb{v}$.

Now let $\mc{A}$ be a commutative $\mb{L}$-algebra,
and let $X=\textup{Spec}(\mc{A})$ be its spectrum
viewed as an affine scheme over $\textup{Spec}(\mb{L})$.
If $\mb{v}\in\textup{Spec}(\mb{L})$ we denote by
$X_\mb{v}=\operatorname{Spec}(\mc{A}_\mb{v})$
its fiber at $\mb{v}$. We denote by $X(\mb{L})$ the
set of $\Lb$-valued points of $X$.
\begin{rem}
We will often use the indeterminate $q=v^2$ (since various naturally
given functions turn out to be functions of $v^2$). If we specialize $q$ at
$\mb{q}>0$ we will tacitly assume that $\mb{v}$ is
the \emph{positive} square root of $\mb{q}>0$.
\end{rem}
\subsection{The Bernstein basis and the center}
Let $\mc{H}=\mc{H}(\mc{R},m)$ be a generic affine Hecke algebra, and
let $X^+\subset X$ denote the cone of dominant elements in the lattice $X$.
As is well known, the map $X^+\to\mc{H}^\times$ defined by
$x\to N_{t_x}:=\theta_x$ is a monomorphism of monoids, and the
commutativity of the monoid $X^+$ implies that this
can be uniquely extended to a group monomorphism
$X\ni x\to\theta_x\in\mc{H}^\times$.
We denote by
$\mc{A}\subset\mc{H}$ the commutative $\Lb$-subalgebra of
$\mc{H}$ generated by the elements $\theta_x$ with $x\in X$.
Observe that $\mc{A}\approx \Lb[X]$.
Let $\mc{H}_0=\mc{H}(R_0,m)$ be the Hecke
subalgebra (of finite rank over the algebra $\Lb$) corresponding to
the Coxeter subgroup $(W_0,S_0)\subset(W^a,S)$.
We have the following well known result due to Bernstein-Zelevinski
(unpublished) and Lusztig (\cite{Lus2}):
\begin{thm}\label{thm:ber}
Let $\mc{H}=\mc{H}(\mc{R},m)$ be a generic affine Hecke algebra.
The multiplication map defines an isomorphism of
$\mc{A}-\mc{H}_{0}$-modules
$\mc{A}\otimes\mc{H}_{0}\to\mc{H}$ and an isomorphism
of $\mc{H}_{0}-\mc{A}$-modules
$\mc{H}_{0}\otimes\mc{A}\to\mc{H}$.
The algebra structure on $\mc{H}$ is determined
by the the following cross relation (with $x\in X$, $\alpha\in F_0$,
$s=r_{\alpha^\vee}$, and $s^\prime\in S$ is a simple reflection
such that $s^\prime\sim_W r_{\alpha^\vee+1}$):
\begin{equation}\label{eq:ber}
\theta_x N_s-N_s \theta_{s(x)}=\left((v^{m_S(s)}-v^{-m_S(s)})+
(v^{m_S(s^\prime)}-v^{-m_S(s^\prime)})\theta_{-\alpha}\right)
\frac{\theta_x-\theta_{s(x)}}{1-\theta_{-2\alpha}}
\end{equation}
(Note that if $s^\prime\not\sim_W s$ then $\alpha^\vee\in 2Y$,
which implies  $x-s(x)\in 2\mathbb{Z}\alpha$ for all $x\in X$. This guarantees
that the right hand side of (\ref{eq:ber}) is always an element of $\mc{A}$).
\end{thm}
\begin{cor} The center $\mc{Z}$ of $\mc{H}$ is the algebra
$\mc{Z}=\mc{A}^{W_0}$. For any $\mb{v}\in\mc{Q}_c$ the center of
$\mc{H}_\mb{v}$ is equal to the subalgebra
$\mc{Z}_\mb{v}=\mc{A}^{W_0}_\mb{v}\subset\mc{H}_\mb{v}$.
\end{cor}
In particular $\mc{H}$ is a finite type algebra over its center
$\mc{Z}\approx \Lb[X]^{W_0}$, and similarly $\mc{H}_\mb{v}$
is a finite type algebra over its center $\mc{Z}_\mb{v}$.
\subsection{Arithmetic and spectral diagrams}\label{par:arithspec}
\subsubsection{The arithmetic diagram}
Let $\mc{H}=\mc{H}(\mc{R},m)$ be a generic affine Hecke algebra
for a \emph{semisimple} root datum $\mc{R}$.
Then the affine Hecke algebra $\mc{H}(\mc{R},m)$ is determined by
the ``arithmetic diagram'' $\Sigma_a=\Sigma_a(\mc{R},m)$ derived from the data
$(\mc{R},m)$ as follows. The arithmetic diagram consists of the
affine Dynkin diagram $(R_0^\vee)^{(1)}$ of the based affine root system
$(R,F)$, with one marked special vertex in each
connected component (corresponding to the simple affine roots
which are not in $F_0^\vee$), the action of $\Omega_X=X/Q(R_0)$
on this diagram, and a labelling of the vertices by the values
of the integer valued function $m_R$. Clearly $\Sigma_a$
determines the tuple of data $(W,S,S_0,\Omega_X,m_S)$ and therefore
it determines $\mc{H}$.
\subsubsection{Standard and semi-standard pairs $(\mc{R},m)$}
\begin{defn} We call the data $(\mc{R},m)$ \emph{standard} if (i)
the root datum $\mc{R}$ is semisimple, and (ii)
$\forall s\in S$: $m_S(s)\not=0$.
If (ii) is replaced by (ii') $\forall s\in S$: either $m_S(s)\not=0$ or
$m_S(s^\prime)\not=0$ (using the notation of Theorem \ref{thm:ber}),
then we call $(\mc{R},m)$ semi-standard.
\end{defn}
We call an arithmetic diagram $\Sigma_a(\mc{R},m)$ of (semi-) standard data $(\mc{R},m)$
a (semi-) standard arithmetic diagram.
\begin{lem}\label{lem:stand}
Let $(\mc{R},m)$ be semi-standard and let $\mc{H}=\mc{H}(\mc{R},m)$.
There exists an admissible isomorphism $\mc{H}\to\mc{H}(\mc{R}^\prime,m^\prime)$
with $(\mc{R}^\prime,m^\prime)$ standard. Hence up to an admissible isomorphism,
$\mc{H}$ is represented by a standard arithmetic diagram.
\end{lem}
\begin{proof}
Suppose that $m_S(s)=0$. Then $m_S(s^\prime)\not=0$ by assumption, and
either $s$ or $s^\prime$ belongs to a simple root $\alpha\in F_0$ such that
$\alpha^\vee\in 2Y$. By Proposition
\ref{prop:Rmcan} we see that $s$ and $s^\prime$ belong to a component $C$ of
$\Sigma_a$ of type $\textup{C}_n^{(1)}$, and $\Omega_X$ fixes all the vertices of $C$
point wise. We denote the corresponding direct summand of $\mc{R}$ by
$\mc{R}^C=(R_0^C,X^C,(R_0^C)^\vee,Y^C,F_0^C)$ (hence $R_0^C$ has type $\textup{B}_n$
and $X$ is its root lattice).

By applying an admissible isomorphism to $\mc{H}$ we may assume that
$m_R(\alpha_0^C)=0$ (where $\alpha_0^C$ is the unique simple affine root of
$\mc{R}^C$ which is not in $F^C_0$).
Applying another admissible isomorphism we may replace
$(\mc{R}^C,m^C)$ by standard data $(\mc{R}^C_1,m^C_1)$ with
$\mc{R}^C_1=(R_1^C,X^C,(R_1^C)^\vee,Y^C,F_1^C)$
(where $R_1$ is of type $\textup{C}_n$
such that its long roots are twice the short roots of $R_0$),
and $m_{1,R}^C(\alpha^\vee)=m^C_R(2\alpha^\vee)$
for all long roots of $R_1$. We repeat this procedure for all components
of type $\textup{C}_n^{(1)}$ if necessary.
\end{proof}
\subsubsection{The spectral diagram}
For the purpose of this paper it is
also convenient to describe $\mc{H}$ in a ``dual'' way, by means of
the ``spectral diagram'' $\Sigma_s(\mc{R},m)$.
A similar (but different) notion of the spectral diagram for a pair $(\mc{R},q)$
was defined in \cite[Definition 8.1]{OpdSol2}. We will use a slight variation
of this definition (see below) adapted to generic affine Hecke algebras
in the present context of Definition \ref{def:def}.
\begin{defn}
Let $n_m:R_0\to\{1,2\}$ be the $W_0$-invariant function on $R_0$
defined by $n_m(\alpha)=2$ iff $m_R(1-\alpha^\vee)\not=m_R(\alpha^\vee)$
(notice that this inequality in particular implies that $\alpha^\vee\in 2Y$).
\end{defn}
We introduce a set $R_m\subset X$ by
$R_m=\{n_m(\alpha)\alpha\mid \alpha\in R_0\}$.
Then $R_m$ is again a root system, and $W_0=W_0(R_m)$.
By construction we have
$Q(R_m)\subset Q(R_0)\subset X\subset P(R_m)\subset P(R_0)$.
Let ${\Omega^\vee_Y}$ be the abelian group
$\Omega^\vee_Y=Y/Q(R_m^\vee)$.
Let $W^\vee=Y\rtimes W_0$ be the affine Weyl
group associated with the dual $\mc{R}^\vee$
of the root datum $\mc{R}$.
Observe that
\begin{equation}\label{eq:ominv}
\mc{H}(\mc{R}_0,m)\subset \mc{H}(\mc{R},m)=(\mc{H}(\mc{R}^m,m))^{\Omega^\vee_Y}\subset 
\mc{H}(\mc{R}^m,m)
\end{equation}
with $\mc{R}_0=(Q(R_0),R_0,P(R_0^\vee),R_0^\vee,F_{0})$ and $\mc{R}^m=(P(R_m),R_0,Q(R_m^\vee),R_0^\vee,F_{m,0})$.
Both inclusions have finite index (here we use that $\mc{R}$ is semisimple).
We have
\begin{equation}\label{eq:dualWeyl}
W^\vee=W((\mc{R}^m)^\vee)\rtimes {\Omega^\vee_Y}
\end{equation}
where $W((\mc{R}^m)^\vee)=W(R_m^{(1)})$ is the affine Coxeter group associated
with the affine extension of the root system $R_m$.
Since $\mc{R}$ is semi-simple then $X_m:=P(R_m)$ is the largest sublattice of
$P(R_0)$ such that the function $m_R$ on $R=R_0^\vee\times\mathbb{Z}$ is
invariant for the natural action of $W(\mc{R}^m)=X_m\rtimes W_0$ on $R$.
Thus we see that $\mc{H}(\mc{R}^m,m)$ is the largest affine Hecke
algebra containing $\mc{H}(\mc{R},m)$ as a subalgebra of finite index.
Observe that both $\mc{R}_0$ and $\mc{R}^m$ are direct products of their 
irreducible components.

The diagram underlying the spectral diagram $\Sigma_s(\mc{R},m)$ is the diagram
of the affine extension $R_m^{(1)}$ of $(R_m,F_m)$. We identify its
vertices with the set $F_m^{(1)}$ of affine simple roots. We have a
natural action of ${\Omega^\vee_Y}$ on this diagram, and this information is included
in the spectral diagram.
The last piece of data of the spectral diagram $\Sigma$
is a $W^\vee$-invariant labeling of its nodes by integers.
We will define this labeling by restriction to $F_m^{(1)}$
of a $W^\vee$-invariant function $m_{R}^\vee$
on $R_m^{(1)}$ with integer values.

Recall that $m$ determines a $W$-invariant integer valued function
$m_R$ on $R$. From this function we construct two half-integral $W_0$-invariant
functions $m_{\pm}$ on $R_0$ as follows: for $\alpha\in R_0$ we define
\begin{equation}\label{eq:m}
\left\{
\begin{array}{ll}
   m_{+}(\alpha)=\frac{1}{2}(m_R({\alpha^\vee})+m_R(1-{\alpha^\vee}))\\
   m_{-}(\alpha)=\frac{1}{2}(m_R({\alpha^\vee})-m_R(1-{\alpha^\vee}))
\end{array}
  \right.
\end{equation}
The nodes of the Dynkin diagram of $R_m^{(1)}$ are labelled by
the values of a $W^\vee$-invariant function $m_{R}^\vee$ on
the affine root system $R_m^{(1)}$ defined as follows.
Let $a^\vee=n_m(\alpha) \alpha+k\in R_m^{(1)}$.
We introduce a $W^\vee$-invariant signature function
$\ep:R_m^{(1)}\to\{\pm\}$
by defining $\ep(a^\vee)1=(-1)^{k(n_m(\alpha)-1)}$
(one easily checks that this is indeed $W^\vee$-invariant).
Finally we define an integer valued function $m^\vee_R$ by:
\begin{equation}
m_{R}^\vee(a^\vee):=n_m(\alpha) m_{\ep(a^\vee)}(\alpha)
\end{equation}
\begin{defn}\label{defn:specdiagr}
The spectral diagram $\Sigma_s(\mc{R},m)$
associated with $(\mc{R},m)$ consists of the extended
Dynkin diagram of $(R_m^{(1)},F_m^{(1)})$,
the finite abelian group ${\Omega^\vee_Y}=Y/Q(R_m^\vee)$
of automorphisms of this diagram, and the function
$m_R^\vee$ defined above.
\end{defn}
The following is clear:
\begin{prop} If $\mc{R}$ is semisimple then the algebra $\mc{H}(\mc{R},m)$ is
completely determined by its spectral diagram $\Sigma_s=\Sigma_s(\mc{R},m)$. 
Note that  $(\mc{R},m)$
is semi-standard iff $\Sigma_s$ is semi-standard in the sense that for
all $a^\vee=n_m(\alpha) \alpha+k\in R_m^\vee$, either 
$m_-(\alpha)\not=0$ or $m_+(\alpha)\not=0$.
\end{prop}
\section{Normalized affine Hecke algebras and the $\mu$-function}\label{sec:plan}
In this section we introduce ``abstract" normalized generic affine Hecke algebras. 
By way of motivation, let us briefly discuss the role of such objects in the study of the Plancherel 
measure of  
(the group of rational points) of a reductive group $G$ defined over a nonarchimedean local field
$k$.
In this context one associates to an irreducible smooth representation 
$\rho$ of a compact open subgroup $\mathbb{K}\subset G$ the 
$\rho$-spherical Hecke algebra $\mc{H}_\rho:=\mc{H}(G,\rho)$, with unit element $e_\rho$. 
According to \cite{BHK},  
this complex algebra $\mc{H}_\rho$ has the structure of a normalized type I Hilbert algebra  
defined by a trace functional $\tau^1$ and a certain $*$-operation, such that  $\tau^1(e_\rho)=1$.
For our purposes it is more appropriate (see below) to use the normalization $\tau^d=d\tau^1$ 
of the trace 
functional of $\mc{H}_\rho$, with $d=\frac{\textup{deg}(\rho)}{\textup{vol}(\mathbb{K})}$.
With this normalization, \cite[Theorem 3.3]{BHK} can be formulated as follows: 
Let $C^*_r(\mc{H}_\rho)$ be the $C^*$-algebra closure of $\mc{H}_\rho$. 
The pair $(\rho,\mathbb{K})$ 
defines a closed subset  $\hat{G}_r(\rho)\subset\hat{G}_r$ of the tempered irreducible 
spectrum $\hat{G}_r(\rho)$ of $G$, and there exists a natural 
homeomorphism $m_{(\rho,\mathbb{K})}:\hat{G}(\rho)\to \widehat{C^*_r(\mc{H}_\rho)}$
which is Plancherel measure preserving. 

We will apply these ideas in the context of unipotent types $(\rho,\mathbb{K})$ for $G$ 
(see \cite{Opd4}), 
and in this situation the algebras $\mc{H}_\rho$ will be extended affine Hecke algebras 
specialized at $\mb{v}$, where $\mb{q}=\mb{v}^2$ is the cardinality of the residue field of $k$. 
If we consider all unramified extensions of $k$ at once, the corresponding generic 
Hecke algebras over $\mb{L}$ come into play. With the normalization 
of Haar measures defined in \cite{DeRe}, the normalizing factors $d$ of the traces $\tau^d$ 
are elements of the quotient field $\mb{K}$ of $\mb{L}$.  The pair $(\mc{H}_\rho,\tau^d)$ 
will be referred to as a normalized Hecke algebra.

It is important to include the normalized trace $\tau^d$ as part of the 
definition of the normalized Hecke algebra $\mc{H}_{\rho,\mb{L}}$. It will turn out 
\cite{Opd4,FO}, rather surprisingly,  
that this additional structure suffices 
to treat the unipotent affine Hecke algebras of rank $0$ (the so-called cuspidal case) on an equal 
footing with the unipotent 
affine Hecke algebras of positive rank, and this is essential to the notion of  
spectral transfer morphisms between unipotent affine Hecke algebras. 

Let us now turn to the detailed discussion of these notions.
\subsection{The $\mu$-function and its automorphisms}
\subsubsection{The standard trace of an affine Hecke algebra}
We equip a generic affine Hecke algebra $\mc{H}=\mc{H}(\mc{R},m)$
with a $\mb{L}$-valued trace
$\tau^1:\mc{H}\to\mb{L}$ defined by
\begin{equation}
\tau^1(N_w)=\delta_{w,e}
\end{equation}
Then $\tau^1$ defines a family of $\mathbb{C}$-valued traces
$\mathbb{R}_{>1}\ni\mb{v}\to\tau^1_\mb{v}$
on the $\mathbb{R}_{>1}$-family of algebras of complex
algebras $\mc{H}_\mb{v}$.
%It is an elementary but fundamental fact that
%the traces $\tau^1_\mb{v}$ are positive with respect to the conjugate
%linear anti-involution $*$ on $\mc{H}_\mb{v}$ defined by $N_w^*=N_{w^{-1}}$.
%For all $\mb{v}\in\mathbb{R}_{>1}$
%the trace $\tau^1_\mb{v}$ and the $*$ operator
%define the structure of a type I Hilbert algebra on
%$\mc{H}_\mb{v}$ (see \cite{Opd1}). 
It will be of crucial importance
to normalize the family of traces $\tau^1_{\mb{v}}$ in a suitable way.
\subsubsection{Definition of normalized affine Hecke algebras}\label{subsubnormaffha}
Let $\mb{K}$ denote the field of fractions of $\Lb$, and
let $\mb{M}^\prime\subset\mb{K}^\times$ be the subgroup of
the multiplicative group $\mb{K}^\times$
generated by $(v-v^{-1})$, by $\mathbb{Q}^\times$,
and by the $q$-integers $[n]_q:=
\frac{v^n-v^{-n}}{v-v^{-1}}$. Observe that $d(\mb{v}^{-1})=
\pm d(\mb{v})$ for all $d\in\mb{M}'$.
We put
\begin{equation}
\mb{M}=\{d\in\mb{M}^\prime \mid d(\mb{v})>0\mathrm{\ if\ }\mb{v}>1\}
\end{equation}
We write $\mb{M}^\prime_{(n)}$ and $\mb{M}_{(n)}$ for the subsets
of elements with vanishing order $n$ at $\mb{v}=1$. Observe that
$\mb{M}=\cup_{n\in\mathbb{Z}}\mb{M}_{(n)}$.
\begin{defn}\label{def:normalize}
A normalized affine Hecke algebra is a pair $(\mc{H},\tau^d)$ where $\mc{H}$ is
a generic affine Hecke algebra and where $\tau^d$ is a $\mb{K}$-valued trace
on $\mc{H}$ of the form $\tau^d=d\tau^1$ with $d\in\mb{M}$ as above. 
\end{defn}
In a normalized affine Hecke algebra $(\mc{H},\tau^d)$ we have a family
of positive traces $\{\tau^d_\mb{v}=d(\mb{v})\tau^1_{\mb{v}}\}_{\mb{v}>1}$ on the
family of specializations $\{\mc{H}_\mb{v}\}_{\mb{v}>1}$. Before we go into the harmonic
analytic aspects of the spectral decomposition of these traces we discuss
the so-called $\mu$-function of a normalized affine Hecke algebra.
\subsubsection{The $\mu$ function of $(\mc{H},\tau^d)$}
The $\mu$ function captures the structure of a normalized
affine Hecke algebra and plays a key role in the harmonic analysis
related to the spectral decomposition of the traces $\tau^d_\mb{v}$.

Let $(\mc{H},\tau^d)$ be a normalized affine Hecke
algebra, and let $T$ be the diagonalizable group scheme
with character lattice $\mathbb{Z}\times X$, viewed as
diagonalizable group scheme over $\textup{Spec}(\Lb)$
via the $\mathbb{C}$-algebra homomorphism
$\Lb\to\mathbb{C}[\mathbb{Z}\times X]$ given by
$v^n\to (n,0)$.
We denote by $T(\Lb)$ its group of $\Lb$-points, and by
$T_\mb{v}$ the fibre at $\mb{v}\in\mathbb{C}^\times$.
The Weyl group $W_0$ acts on $T$ in the usual way. If 
$f$ is a function on $T$ and $w\in W_0$ then we 
denote by $f^w$ the function $f\circ w^{-1}$. 

The $\mu$-function of $(\mc{H},\tau^d)$ is a rational
function on $T$ which we define in terms of $d$ and
the spectral diagram of $\mc{H}$. We first assign the so-called 
Macdonald $c$-function \cite{Ma1} $c_{m,\alpha}$ to the roots $\alpha\in R_0$ by
\begin{equation}\label{eq:formc}
c_{m,\alpha}:=\frac
{(1+v^{-2m_{-}(\alpha)}\alpha^{-1})
 (1-v^{-2m_{+}(\alpha)}\alpha^{-1})}
{1-\alpha^{-2}}
\end{equation}
The $c$-function $c_m$ of $\mc{H}$ is defined by
$c_m:=\prod_{\alpha\in R_{0,+}}c_\alpha$. Finally we define
\begin{defn}\label{defn:mu}
The $\mu$-function of $(\mc{H},\tau^d)$ is defined by 
\begin{equation}
\mu=\mu_{\mc{R},m,d}=v^{-2m_W(w_0)}\frac{d}{c_mc_m^{w_0}}
\end{equation}
where $w_0\in W_0$ is the longest Weyl group element,
and where $m_W:W\to\mathbb{Z}$ is defined by
\begin{equation}
m_W(w)=\sum_{a\in R_+\cap w^{-1}(R_-)}m_R(a)
\end{equation}
\end{defn}
From Proposition \ref{prop:Riscan} it is clear that:
\begin{prop}\label{prop:mucan}
The $\mu$-function is canonically determined by $(\mc{H},\tau^d)$.
\end{prop}
Let us comment on the meaning of the spectral
diagram in terms of the $\mu$-function.
We are going to consider below the group of affine symmetries of
the rational function $\mu$ on $T$ (with $T$ viewed as $\mb{L}$-scheme).
This symmetry group is determined by its action on the group of
complex points $T_\mb{v}(\mathbb{C})$ of a generic fibre $T_\mb{v}$.
We fix a fibre at $\mb{v}>1$. If $\mc{R}$ is semisimple then it is
also easy to see that this group acts faithfully on the
compact form $T_u:=(T_\mb{v}(\mathbb{C}))_u$.
We lift the $\mu$-function to the vector space
$V=\mathbb{R}\otimes_\mathbb{Z} Y$ via the exponential map
$x\to\exp(2\pi i x)\in T_u$ of the compact torus $T_u$.
The zero set of $\mu$, considered as periodic function on $V$,
is a union of affine hyperplanes in $V$. Let us call these
hyperplanes \emph{the $\mu$-mirrors}.
\begin{prop}\label{prop:musym}
Assume that $(\mc{R},m)$ is semi-standard, with spectral diagram $\Sigma_s$.
\begin{enumerate}
\item[(i)] The group of affine symmetries of $\mu$ on $V$
generated by $Y$ and by the reflections in the $\mu$-mirrors is equal
to the group $W^\vee=W(R_m^{(1)})\rtimes \Omega^\vee_Y$ (the affine Weyl group
of the affine Dynkin diagram (equipped with $\Omega_Y^\vee$-action)
underlying $\Sigma_s$).
\item[(ii)] The group $\textup{Aut}_T(\mu)$ of affine symmetries of
$\mu$ on $T$ is a semidirect product of the form
$\textup{Aut}_T(\mu)=W_0\rtimes\textup{Out}_T(\mu)$ for some
subgroup $\textup{Out}_T(\mu)$.
\item[(iii)] $(\mc{R},m)$ can be replaced by a standard pair
$(\mc{R}_1,m_1)$ without changing $\mu$.
\item[(iv)] If $(\mc{R},m)$ is standard then
$\textup{Out}_T(\mu)=\Omega_X^*\rtimes\Omega_0^Y$, with
$\Omega_0^Y=\textup{Out}_Y(R_0^\vee)$, the group of diagram
automorphisms of $(R_0^\vee,F_0^\vee)$ which normalize the lattice $Y$,
and $\Omega_X^*=P(R_0^\vee)/Y\subset T$ (the dual of $\Omega_X$), acting
trivially on $W_0$.
\item[(v)] There is a canonical isomorphism between the group
$\textup{Aut}_{es}(\mc{H})$, the opposite of the group of essentially
strict automorphisms of $\mc{H}$, and $\textup{Out}_T(\mu)$.
\end{enumerate}
\end{prop}
\begin{proof}
For (i) we only
need to consider the case where there exists $\alpha\in R_0$ such that
$m_+(\alpha)=0$ (and thus $m_-(\alpha)\not=0$ by assumption). But then
$\alpha$ belongs to an indecomposable summand $\mc{R}^\prime$ of $\mc{R}$ of type
$\textup{C}_n^{(1)}$ (see the proof of Proposition \ref{prop:Rmcan}).
In particular, the translation lattice $Y^\prime$ of this summand
is equal to $Q((R_m^\prime)^\vee)=P((R_0^\prime)^\vee)$. This shows that
the reflections of $W(R_0)$ (acting on $V$) are contained in the group generated
by $Y$ and the reflections in the $\mu$-mirrors.

Assertion (ii) will follow from the proof of (iv) using (iii).

For (iii), consider the $\mu$-function of a semi-standard but not standard
pair of data $(\mc{R},m)$. Hence there exists a component
in $\mc{R}$ of type $\textup{C}_n^{(1)}$ such that $m_S$ has a zero at precisely one
of the extreme ends of $\Sigma_a$. We may (and will) assume without loss of generality
that $\mc{R}$ equals this component. The assumption implies that $m_-(\alpha)=m_+(\alpha)$
or that $m_-(\alpha)=-m_+(\alpha)$ for a root $\alpha\in R_0$ such that $\alpha^\vee\in 2Y$.
It is easy to see that the $\mu$ function is unchanged if we change $m_+$ or $m_-$ to
its negative on any of the $W_0$-orbits of roots in $R_0$. So after such a symmetry
transformation of $\mu$ we may assume that $m_+(\alpha)=m_-(\alpha)$. Now we can
simplify the denominator of $\mu$ to $(1-v^{-4m_+(\alpha)}\alpha^{-2})$, or
(which amounts to the
same thing) multiply all roots of $R_0$ in the orbit of $\alpha$ by $2$.
Hence $R_0$ is now changed to a new root system $R_1$ of type $\textup{C}_n$,
and the corresponding pair $(\mc{R}_1,m_1)$ is obviously standard.

Next we consider (ii) and (iv). By elementary covering theory we have
$\textup{Aut}_T(\mu)=\textup{Aut}_V(\mu,Y)/Y$, where
$\textup{Aut}_V(\mu,Y)$ denotes the group of affine linear
transformation of $V$ which fix $\mu$ and normalize $Y$.
Using that we are in the standard case, this last group is easily seen to
be equal to $(\Omega_X^*\rtimes W_0)\rtimes\Omega^Y_0$.
But $W_0$ acts trivially on $P(R_0^\vee)/Q(R_0^\vee)$, and in particular
on $\Omega_X^*=P(R_0^\vee)/Y$. Hence $W_0$ is a normal subgroup,
and we can rewrite the answer in the stated form.

Finally we prove (v). By Lemma \ref{lem:stand} and parts (ii),(iv)
we may assume that $(\mc{R},m)$ is standard. This implies easily that an
essentially strict automorphisms $\lambda$ acts on the distinguished
basis elements $N_s$ with $s\in S$ by $\lambda(N_s)=N_{l(s)}$ for
an element $l\in\Omega_0^X$, the group of diagram automorphisms of
$(R_0,F_0)$ stabilizing $X$ (in fact, these automorphisms are automatically
strict in the standard case). In addition we have that action $N_\omega\to\lambda(\omega)N_\omega$
of $\lambda\in\Omega_X^*$ on the elements of the form $N_\omega$ with $\omega\in\Omega_X$.
It is clear that $\Omega_X^*\rtimes\Omega_0^X$ exhausts $\textup{Aut}_{es}(\mc{H})$,
and that its action on $\mc{H}$ induces by duality a faithful right action on $T$
by elements of $\textup{Aut}_T(\mu)$. Now (iv) completes the proof
($\Omega_0^X$ and $\Omega_0^Y$ correspond, as well as the actions
of $\Omega_X^*$).
\end{proof}
\begin{prop}\label{prop:omega*}
Suppose that $(\mc{R},m)$ is standard and put
$\Omega^*_m=P(R_0^\vee)/Q(R_m^\vee)$.
The group $\Omega^*_m\rtimes\Omega^Y_0$
acts faithfully on the labelled Dynkin diagram underlying
$\Sigma_s$, and we have an exact sequence
\begin{equation}
1\to\Omega^\vee_Y\to\Omega^*_m\rtimes\Omega^Y_0\to\textup{Out}_T(\mu)\to 1
\end{equation}
In particular, $\Omega^*_m\rtimes\Omega^Y_0$ acts faithfully on
$\Sigma_s(\mc{R}^m,m)$ (since $\Omega^\vee_Y=1$ in this case).
\end{prop}
%\subsubsection{The trace in terms of the $\mu$-function}
%The underlying reason for the importance of the
%$\mu$-function for the harmonic analysis of $(\mc{H},\tau^d)$
%is the following formula for the trace $\tau$ of
%$(\mc{H},\tau^d)$ (we refer to \cite{Opd0} for further details):
%\begin{thm}[\cite{Opd0}, Theorem 3.7]
%For all $h\in\mc{H}$ we have
%\begin{gather}
%\begin{split}
%\tau(h)=v^{-2m_W(w_0)}
%\int_{t\in tT_u}
%{\Delta(t)}^{-1}E_t(h)
%\mu_{m,d}(v,t)dt\\
%\end{split}
%\end{gather}
%Here $T_u$ is the compact form of $T(\mathbb{C})$, with
%normalized Haar measure of $dt$ (extended holomorphically
%to $T(\mathbb{C})$), $t\in T(\mathbb{R})$ a real point
%satisfying $\log(\alpha(t))<<0$ for all $\alpha\in R_{0,+}$,
%$E_t$ is a certain holomorphic family of matrix coefficients of the
%minimal principal series representations, and $\Delta(t)$ is the
%Weyl denominator.
%\end{thm}
\subsection{Poles and zeroes of $\mu$-functions; residual cosets}\label{par:polezero}
The poles and zeroes of the $\mu$-function $\mu_1:=\mu_{\mc{R},m,1}$ have
certain remarkable properties which reflect the role of $\mu_1$ in 
the spectral decomposition of the positive trace $\tau^1$
(and can in fact be deduced from this, see \cite{Opd3}). The central notion
is that of a ``residual coset" \cite{HOH,Opd1,OpdSol2}. We review residual 
cosets in this subsection, paying special attention to their behavior as 
a family over $\operatorname{Spec}(\mb{L})$.    
\subsubsection{Cosets of tori}\label{par:cos}
Recall that the category of diagonalizable group schemes of finite
type over $\operatorname{Spec}(\mb{L})$ is anti-equivalent with the
category of finitely generated abelian groups via the anti-equivalence
$S\to X_S:=\textup{Hom}_{\textup{Spec}(\mb{L})}(S,\mathbb{G}_{m,\textup{Spec}(\mb{L})})$.
Below we will use the phrase ``torus over $\mb{L}$'' as shorthand for a
``connected diagonalizable group scheme of finite type over
$\operatorname{Spec}(\mb{L})$''. Let $\mb{A}$ be a commutative
unital $\mb{L}$-algebra.
A subtorus $i_S: S_\mb{A}\to T_\mb{A}$ of $T_\mb{A}$ over $\mb{A}$
corresponds to an epimorphism $i_S^*:X\to X_S$ of lattices. By a \emph{coset}
$L_\mb{A}\subset T_\mb{A}$ of $S_\mb{A}$ we mean a closed subscheme of $T_\mb{A}$
defined by a closed immersion $i:S_\mb{A}\to T_\mb{A}$
of $S_\mb{A}$ in $T_\mb{A}$ that is given by a group homomorphism
$i^*:X\to \mb{A}[X_S]^\times$ of
the form $i^*(x)=r(x)i_S^*(x)$, where $r:X\to\mb{A}^\times$ denotes an
$\mb{A}$-point $r\in T(\mb{A})$. In this situation we write
$L_\mb{A}=rS_\mb{A}$. We define
$\textup{codim}(L_\mb{A}):=\textup{rank}(X)-\textup{rank}(X_S)$.
\subsubsection{Residual cosets}\label{subsub:rescos}
Given a coset $L_\mb{A}\subset T_\mb{A}$ in $T_\mb{A}$ over $\mb{A}$ we
define $p_\ep(L_\mb{A})=\{\alpha\in R_0\mid
\ep\alpha|_{L_\mb{A}}=v^{-2m_\ep(\alpha)}.1_\mb{A}\}$
and
$z_\ep(L_\mb{A})=\{\alpha\in R_0\mid \ep\alpha|_{L_\mb{A}}=1_\mb{A}\}$
(with $\ep=+$ or $\ep=-$).
Then $L_\mb{A}\subset T_\mb{A}$ is called a residual coset over $\mb{A}$ if
\begin{equation}\label{eq:ineqgen}
\#p_+(L_\mb{A})+\#p_-(L_\mb{A})-\#z_+(L_\mb{A})-\#z_-(L_\mb{A})\geq\textup{codim}(L_\mb{A})
\end{equation}
We denote by $\mc{L}_\mb{A}=\mc{L}_\mb{A}(\mc{R},m)$ the set of residual cosets in
$T_\mb{A}$ over $\mb{A}$.
If $\mb{v}\in\mathbb{R}_+\subset\mathbb{C}^\times$ we use
the shorthand notation $\mc{L}_\mb{v}$ for the set of residual cosets in
$T_\mb{v}$ over $\mathbb{C}_\mb{v}$. We write $\mc{L}=\mc{L}(\mc{R},m)$ for the
set of residual cosets of $\mu$ (or of $\mu_1$; this set is independent of the value of $d$) 
over $\mb{L}$. Of special interest are the residual points:
\begin{defn}
Let $\textup{Res}(\mc{R},m)\subset\mc{L}(\mc{R},m)$ denote
the set of $\mb{L}$-valued points of $T$. 
This set is called the set of generic residual
points of $\mu$ (or of $\mu_1$).  
\end{defn}
%$L_\mb{v}\subset T_\mb{v}$ a coset we similarly define
%$p_\ep(L_\mb{v})=\{\alpha\in R_0\mid
%\mb{v}^{m_\ep(\alpha)}\alpha|_{L_\mb{v}}=1\}$
%and
%$z_\ep(L_\mb{v})=\{\alpha\in R_0\mid \ep\alpha|_{L_\mb{v}}=1\}$
%(with $\ep=+$ or $\ep=-$).
%Then $L_\mb{v}$ is called a residual  of $T_\mb{v}$ if
%\begin{equation}\label{eq:ineq}
%\#p_+(S_\mb{v})+\#p_-(S_\mb{v})-\#z_+(S_\mb{v})-\#z_-(S_\mb{v})\geq\textup{codim}(S_\mb{v})
%\end{equation}
%We denote by $\mc{L}_\mb{v}$ the set of residual subvarieties of
%$T_\mb{v}$.
%Finally, a closed connected subscheme $S$ of $T$ over $\textup{Spec}(\Lb)$
%is called a
%residual subscheme if its generic fiber
%$S_\mb{K}\subset T_\mb{K}$ is residual.
%We use the notation $p_\ep(S):=p_\ep(S_\mb{K})$
%and $z_\ep(S):=z_\ep(S_\mb{K})$.
%Observe that $W_0$ acts naturally on the set
%of residual subvarieties and subschemes.
%We denote by $\mc{L}$ the set of residual
%subschemes over $\textup{Spec}(\Lb)$ of $T$.
%We write $\mc{L}$ if we want to stress the dependence
%of the set $\mc{L}$ on the root datum $\mc{R}$ and on $m$ (and
%similarly for the fibers $\mc{L}_\mb{K}$ and $\mc{L}_\mb{v})$.
Let $L\subset T$ be a coset of a subtorus $T^L\subset T$ and suppose that
$L$ is residual. Let $q:T\to {}_LT:=T/T^L$ be the quotient map,
and put ${}_LX\subset X$ for the corresponding sublattice of $X$.
Then $R_L=R_0\cap {}_LX$ is a parabolic root subsystem.
Putting ${}_LY=Y/({}_LX)^\bot$
we obtain a root datum ${}_L\mc{R}:=(R_L,{}_LX,R_L^\vee,{}_LY)$.
If $m_{L,\pm}=m_\pm|_{R_L}$ then there exists a unique cocharacter
${}_Lm:\mathbb{C}^\times\to\mc{Q}_c({}_L\mc{R})$
such that $m_{L,\pm}$ is associated with a spectral diagram
$\Sigma_s({}_L\mc{R},{}_Lm)$ as described in \ref{par:arithspec}.
\begin{thm}\label{thm:genrescos}
\begin{enumerate}
\item[(i)]
If $L_\mb{v}\subset T_\mb{v}$
is a residual coset over $\mb{A}=\mathbb{C}_\mb{v}$
with $\mb{v}\in\mathbb{R}_+$ then the inequality
(\ref{eq:ineqgen}) is an equality. Similarly
for all residual cosets $L\in\mc{L}$
and $L_\mb{K}\in\mc{L}_\mb{K}$ the inequality
(\ref{eq:ineqgen}) is an equality.
\item[(ii)]
$L\in\mc{L}$ iff $L=q^{-1}(r)$ with $r\in {}_LT(\mb{L})$
a generic $({}_L\mc{R},{}_Lm)$-residual $\mb{L}$-valued point.
Analogous statements hold for $\mc{L}_\mb{K}$ and
for $\mc{L}_\mb{v}$, provided $\mb{v}\in\mathbb{R}_+$.
\item[(iii)]
The sets $\mc{L}_\mb{K}$, $\mc{L}_\mb{v}$ and
$\mc{L}$ are finite, $W_0$-invariant and consist of
cosets of subtori of
$T_\mb{K}$, $T_\mb{v}$ and $T$ respectively.
\item[(iv)]
The map $\mc{L}\to\mc{L}_\mb{K}$ given by $L\to L_\mb{K}$ is bijective.
\item[(v)]
If $\mb{v}\in\mathbb{R}_+$ and $\mb{v}\not=1$ then the map
$\mc{L}\to\mc{L}_\mb{v}$ given by $L\to L_\mb{v}$ is bijective.
\end{enumerate}
\end{thm}
\begin{proof}
The assertion (i) for $\mc{L}_\mb{v}$ (with $\mb{v}\in\mathbb{R}_+$)
is \cite[Theorem 6.1(B)]{Opd3}. Let $L\in\mc{L}$.
Since $\mathbb{R}_+\subset\mathbb{C}^\times$ is Zariski dense
there exists
$\mb{v}\in\mathbb{R}_+$ such that $p_\ep(L)=p_\ep(L_\mb{v})$
and $z_\ep(L)=z_\ep(L_\mb{v})$, proving the assertion of (i)
for the generic case as well. A similar argument
works for $\mb{A}=\mb{K}$.

To prove (ii) first observe that $p_\ep(L),z_\ep(L)\subset R_L$.
A glance at (\ref{eq:ineqgen}) makes it clear that $q(L)\subset {}_LT$
is an $\mb{L}$-valued point ${}_Lr$ of ${}_LT$ which is an
$({}_L\mc{R}, {}_Lm)$-residual point over $\mb{L}$. The converse is clear
from the definitions.
Analogous arguments can be given for $\mc{L}_\mb{K}$ and
for $\mc{L}_\mb{v}$. This proves (ii).

Now we are in the position to
reduce the remaining statements to the case of (generic) residual points
and use the results of \cite{OpdSol2}.

By
\cite[Theorem 2.44, Proposition 2.52,
Proposition 2.56, Proposition 2.63]{OpdSol2}
the assertions of (iii) and (iv) can be easily checked.

Finally observe that
\cite[Proposition 2.56]{OpdSol2} implies that
if $r$ is as above and $\alpha\in R_L$ then
$\alpha(r)=\zeta v^n$
with $\zeta$ a root of unity, and $n\in\mathbb{Z}$.
This implies that
$r(\mb{v})$ is residual for all
$\mb{v}\in\mathbb{R}_+\backslash\{1\}$, proving (v).
\end{proof}
\subsubsection{Regularization of $\mu$ along a residual coset}
Let $L\in\mc{L}(\mc{R},m)$ be a residual coset for $\mu=\mu_{\mc{R},m,d}$. We define
the regularization $\mu^L$ of $\mu$ along $L$ as follows:
\begin{equation}\label{eq:mureg}
\mu^L=d
{\frac{v^{-2m_W(w_0)}\prod_{\alpha\in R_0\backslash z_-(L)}(1+\alpha^{-1})
\prod_{\alpha\in R_0\backslash z_+(L)}(1-\alpha^{-1})}
{\prod_{\alpha\in R_0\backslash p_-(L)}(1+v^{-2m_{-}(\alpha)}\alpha^{-1})
 \prod_{\alpha\in R_0\backslash p_+(L)}(1-v^{-2m_{+}(\alpha)}\alpha^{-1})}}
\end{equation}
where $p_\ep(L)$ and $z_\ep(L)$ are as defined in \ref{par:polezero}.
\begin{prop}\label{prop:reg}
Let $L$ be a residual coset.
The function $\mu^L$ restricts to a nonzero rational
function $\mu^{(L)}=\mu^{(L)}_{\mc{R},m,d}$ on $L$.
For all $\mb{v}\in\mathbb{R}_+\backslash\{1\}$
the rational function $\mu^{(L)}$ on $L$ is regular at the
generic point of the fiber $L_\mb{v}$, and restricts to a nonzero
rational function on $L_\mb{v}$. The function $\mu^{(L)}_{\mc{R},m,1}$ has
vanishing order $\textup{codim}(L)$ at the generic point of the fiber
$L_\mb{v}$ for $\mb{v}=1$.
\end{prop}
\begin{proof} We need to check that $\mu^L$ maps to an element
of the function field $\mb{K}(L_\mb{K})$ of $L$ if we restrict
to $L$. This is clear by the definition of the sets $p_\ep(L)$
and $z_\ep(L)$. Let $\mb{v}\in\mathbb{R}_+$.
Let $r\in {}_LT$ be as in Theorem \ref{thm:genrescos}. As was mentioned
at the end of the proof of Theorem \ref{thm:genrescos}, for all $\alpha\in R_L$
we have that $\alpha(r)=\zeta v^n$ for some $n\in\mathbb{Z}$ and $\zeta$ a root of $1$.
Hence $p_\ep(L_\mb{v})= p_\ep(L)$ and $z_\ep(L_\mb{v})= z_\ep(L)$ if
$\mb{v}\not=1$ (recall that these sets are subsets of $R_L$, see the proof of
Theorem \ref{thm:genrescos}).
On the other hand it follows that $p_\ep(L_\mb{v})=z_\ep(L_\mb{v})$ if $\mb{v}=1$.
Using (\ref{eq:ineqgen}) for $\mb{A}=\mb{L}$ we are done.
\end{proof}
Observe that $\mu^{(\{r\})}_{\mc{R},m,d}\in\mb{K}^\times$ if $r\in\textup{Res}(\mc{R},m)$
(see Theorem \ref{thm:ratfdeg} for more precise statements).
We decompose $\mu^{(L)}_{\mc{R},m,d}$ as follows. Assume (by replacing
$L$ by a $W_0$-translate $w(L)$ if necessary) that $R_L=R_P$ for a standard
parabolic root subsystem $R_P\subset R_0$ defined by $P\subset F_0$.
Using the notation of Theorem \ref{thm:genrescos}(ii) we write
\begin{equation}\label{eq:musplit}
\mu^{(L)}_{\mc{R},m,d}=\mu^{(\{r\})}_{\mc{R}_P,m_P,d}
\frac
{v^{-2m_{W_0}(w^P)}}
{\prod_{\alpha\in R_{0,+}\backslash R_{P,+}}(c_{m,\alpha}c_{m,\alpha}^{w^P})}
\end{equation}
where $w^P=w_0w_P^{-1}\in W^P$ is the longest element.
\section{The spectral decomposition of a normalized affine Hecke algebra}
In this section we discuss the spectral decomposition of the trace of a 
normalized affine Hecke algebra $(\mc{H},\tau^d)$, by which we mean the 
family of spectral decompositions of the family of type I Hilbert algebras 
$(\mc{H}_\mb{v},\mb{v})$ over the base $\mathbb{R}_{>1}$ (viewed as subset 
of the set of real points of $\operatorname{Spec}(\mb{L})$). The goal is 
the explicit description of the family of spectral measures in terms of the regularizations 
$\mu^{(L)}_{\mc{R},m,d}$ of the $\mu$-function along the $W_0$-orbits 
of residual cosets $L\in\mc{L}(\mc{R},m)$. 
\subsection{The tempered spectrum of generic affine Hecke algebras}
Suppose we are given a normalized affine Hecke algebra
$(\mc{H},\tau^d)$ (with $\mc{H}=\mc{H}(\mc{R},m)$ and $d\in\mb{M}$)
Recall that $\tau^d$ defines a family of traces
$\mathbb{R}_{>1}\ni\mb{v}\to\tau^d_\mb{v}=d(\mb{v})\tau^1_\mb{v}$
on the family of $\mathbb{C}$-algebras $\mc{H}_\mb{v}$.
The trace $\tau^d_\mb{v}$ is
positive with respect to the conjugate linear anti-involution
$*$ on $\mc{H}_\mb{v}$ defined by $N_w^*=N_{w^{-1}}$, and in fact
defines the structure of a type I Hilbert algebra on
$\mc{H}_\mb{v}$ (see \cite{Opd1}).
\begin{defn}
We denote by $\mathfrak{S}_\mb{v}=\mathfrak{S}_\mb{v}(\mc{R},m)$ the
irreducible spectrum of the $C^*$-algebra completion
$\mc{C}_\mb{v}=\mc{C}_\mb{v}(\mc{R},m)$ of $\mc{H}_\mb{v}$.
In particular $\mathfrak{S}_\mb{v}$ is equal
to the support of the spectral measure of $\tau^d_\mb{v}$.
\end{defn}
It is known that the underlying set of $\mathfrak{S}_\mb{v}$
can also be characterized
as the set of equivalence classes of irreducible tempered representations
of $\mc{H}_\mb{v}$ (see \cite{DeOp1}), and for this reason we
refer to $\mathfrak{S}_\mb{v}$ as the tempered spectrum
of $\mc{H}_\mb{v}$.

It is an important fact for our purpose that
the tempered spectra $\mathfrak{S}_\mb{v}$ of the algebras
$\mc{H}_\mb{v}$ are the
fibers of the spectrum $\mathfrak{S}\to\mathbb{R}_{>1}$
of a trivial bundle $\mc{C}$ of $C^*$-algebras. We
will discuss this in some detail this section.
\subsubsection{Generic families of discrete series characters}
The irreducible discrete series characters of $\mc{H}_\mb{v}(\mc{R},m)$
(with $\mb{v}\in\mathbb{R}_{>1}$) arise in finitely many families depending
continuously on the parameters of the Hecke algebra (\cite{OpdSol2}, Theorem 5.7).
In our present setup we only consider a very special
instance of this general phenomenon, called ``scaling'' of the parameters
(i.e. varying $\mb{v}>1$ while keeping $m$ fixed). 

Recall that a family $\{f_\mb{v}\}_{\mb{v}>1}$ of complex linear 
functionals on $\{\mc{H}_\mb{v}\}_{\mb{v}>1}$  is called ``weakly continuous" 
(cf. \cite[definition 3.6]{OpdSol2}) if for all 
$w\in W$ the complex function $\mathbb{R}_{>1}\ni\mb{v}\to f_\mb{v}(N_w)$ is continuous.
\begin{thm}[\cite{OpdSol}, Theorem 1.7]\label{thm:def}
There exists a finite set $[\Delta]=[\Delta(\mc{R},m)]$ of weakly continuous
families $\mathbb{R}_{>1}\ni\mb{v}\to\delta_\mb{v}$ of irreducible
characters of $\mc{H}=\mc{H}(\mc{R},m)$ such that
\begin{enumerate}
\item For all $\mb{v}\in\mathbb{R}_{>1}$, the character $\delta_\mb{v}$ is
an irreducible discrete series character of $\mc{H}_\mb{v}$.
\item Given $\mb{v}\in\mathbb{R}_{>1}$ we denote by
$[\Delta_\mb{v}]$ the set of discrete series characters of
$\mc{H}_\mb{v}$. Then the evaluation map
$[\Delta]\ni\delta\to\delta_\mb{v}\in[\Delta_\mb{v}]$
is a bijection.
\end{enumerate}
\end{thm}
\subsubsection{Scaling of the parameters}
The continuous families of discrete series
characters described in the previous paragraph arise from underlying families
of ``scaling isomorphisms'' which can be defined on the level of
formal completions of affine Hecke algebras at a fixed central character.
This construction yields (see \cite[Theorem 1.7]{OpdSol}) a family of bijections
for $\ep\in (0,1]$
\begin{align}\label{eq:scale}
\tilde \sigma_\ep:\textup{Irr}(\mc{H}_\mb{v})&\to
\textup{Irr}(\mc{H}_{\mb{v}^\ep})\\
\nonumber (\pi ,V) &\to (\pi_\ep ,V)
\end{align}
preserving unitarity, temperedness, and the discrete series, and
having the property that for all $h\in\mc{H}_\mb{v}$ the family
$(0,1]\ni\ep\to\pi_\ep(h)$ is real analytic in $\ep$.
These results enable us to choose a set $\Delta$ of real analytic
families of irreducible representations $\mb{v}\to(\delta_\mb{v},V_\delta)$
for $\mc{H}$ such that the character map
$\delta\to\chi_\delta$ yields a bijection $\Delta\to[\Delta]$.
In other words, we can and will choose a complete set $\Delta$ of
real analytic representatives for the set of continuous families of
irreducible discrete series characters $[\Delta]$.
\begin{defn}
We call $\Delta=\Delta(\mc{R},m)$  a complete set of generic irreducible
discrete series characters. For all subsets $P\subset F_0$ we fix a complete set of
generic irreducible discrete series representations
$\Delta_P=\Delta(\mc{R_P},m_P)$
of the semisimple quotient algebra $\mc{H}_P:=\mc{H}(\mc{R}_P,m_P)$ of
the ``Levi-subalgebra'' $\mc{H}^P:=\mc{H}(\mc{R}^P,m^P)\subset \mc{H}(\mc{R},m)$
(see \cite[Section 3.5, Section 4.1]{DeOp1} for the precise definitions
of these subquotient algebras associated with $P$).
\end{defn}
\subsubsection{Scaling transformations on the tempered spectrum}
Let $\mc{S}_\mb{v}$ denote the Schwartz algebra
of $\mc{H}_\mb{v}$. As subspace of the $L^2$-completion
$L^2(\mc{H}_\mb{v})$ we have the following description of
$\mc{S}_\mb{v}$:
\begin{equation}\label{eq:schw}
\mc{S}_\mb{v}:=
\{s=\sum_{w\in W}c_wN_w\mid
\forall k\in\mathbb{N}\ \exists M_k\in\mathbb{R}_+\ \forall w\in W:(1+\mc{N}(w))^k|c_w|\leq M_k\}
\end{equation}
(where $\mc{N}(w)$ is a norm function on $W$, see \cite{DeOp1}).
This is a Fr\'echet completion of $\mc{H}_\mb{v}$ whose
irreducible spectrum, the space of tempered irreducible representation, is known to be
equal (see \cite[Corollary 4.4]{DeOp1}) to the support
$\mathfrak{S}_\mb{v}$ of the spectral measure
$\mu_{Pl,\mb{v}}$ of $\tau^d_\mb{v}$.
(We remark in passing that the (Jacobson) topologies of
$\mathfrak{S}_\mb{v}$ when viewed as the irreducible spectrum of $\mc{C}_\mb{v}$
or when viewed as the irreducible spectrum of the dense, spectrally closed subalgebra
$\mc{S}_\mb{v}\subset \mc{C}_\mb{v}$ are in fact equal
(see \cite[Theorem 3.3]{Lau}; the technical condition required in this theorem is
obviously satisfied in the present case). But the Schwartz algebras have the important
advantage above $\mc{C}_\mb{v}$ of admitting the explicit
description (\ref{eq:schw}). In particular we see that they are \emph{independent}
of $\mb{v}$ when viewed as Fr\'echet spaces (forgetting the algebra structure).
We sometimes write $\mc{S}(\mc{R})$ to denote this underlying Fr\'echet space.

The scaling maps $\tilde{\sigma_\ep}$ (with $\ep>0$)
of (\ref{eq:scale}) define homeomorphisms on the tempered spectrum:
\begin{thm}[\cite{Sol}, Lemma 5.19 and Theorem 5.21]\label{thm:cont}
There exists a family of isomorphisms of pre-$C^*$-algebras
\begin{equation}\label{eq:scaletemp}
\phi_{\ep,\mb{v}}:\mc{S}_{\mb{v}^\ep}\to\mc{S}_\mb{v}
\end{equation}
(with $\mb{v}>1$ and $\ep>0$) such that $\phi_{1,\mb{v}}=\textup{Id}$ and
such that the following holds:
\begin{enumerate}
\item[(i)] For all $s\in\mc{S}(\mc{R})$ fixed
the map $\mathbb{R}_+\times \mathbb{R}_{>1}\ni(\ep,\mb{v})\to
\phi_{\ep,\mb{v}}(s)\in\mc{S}_\mb{v}$ is continuous
if we equip $\mc{S}_\mb{v}$ with its canonical
pre-$C^*$-algebra topology (recall that
$\mc{S}(\mc{R})$ denotes the underlying Fr\'echet space
of $\mc{S}_{\mb{v}^\ep}$ and  as such is independent of
$\mb{v}^\ep$).
\item[(ii)] For all $\pi\in\mathfrak{S}_\mb{v}$ we have
$\tilde{\sigma}_{\ep,\mb{v}}(\pi)\simeq\pi\circ\phi_{\ep,\mb{v}}$
(with $\tilde{\sigma}_{\ep,\mb{v}}$ as in (\ref{eq:scale}), but extended to all
$\ep>0$ by defining $\tilde{\sigma}_{\ep,\mb{v}}:=
\tilde{\sigma}_{\ep^{-1},\mb{v}^\ep}^{-1}$
if $\ep\geq 1$).
\item[(iii)] The map $\tilde{\sigma_\ep}$ (with $\ep>0$)
induces a homeomorphism from $\mathfrak{S}_\mb{v}$ onto
$\mathfrak{S}_{\mb{v}^\ep}$.
\end{enumerate}
\end{thm}
\begin{proof}
In \cite[Theorem 5.21]{Sol} the isomorphisms of pre-$C^*$-algebras
$\phi_{\ep,\mb{v}}$ were constructed for $\ep\in(0,1]$ in such a way that
$\phi_{1,\mb{v}}$ is the identity, and
such that the continuity (i) holds (this is \cite[Lemma 5.19]{Sol}).
For $\ep\geq 1$ we define $\phi_{\ep,\mb{v}}:=\phi_{\ep^{-1},\mb{v}^\ep}^{-1}$.
It is easy to see that the required continuity holds in this case as well,
by the same arguments as used in \cite[Lemma 5.19]{Sol}.
According to \cite[Theorem 5.19(2)]{Sol}, for all triples $(P,\delta,t)$ with
$P\subset F_0$, $\delta\in\Delta_{P}$ and $t\in T^P_u$ one has
$\pi(P,\delta_\mb{v},t)\circ\phi_{\ep,\mb{v}}\simeq \pi_\ep(P,\delta_\mb{v},t)$.
Since these induced
characters are finite sums of irreducible tempered characters this implies
that $\pi_i\circ \phi_{\ep,\mb{v}}\simeq \tilde{\sigma_\ep}(\pi_i)$ for all irreducible summands
$\pi_i$ of $\pi(P,\delta_\mb{v},t)$. But all irreducible tempered representations are direct
summands of unitarily induced standard representations of the form $\pi(P,\delta_\mb{v},t)$
\cite[Theorem 3.22]{DeOp1}, hence we deduce the second assertion (ii).
Since the maps $\phi_\ep$ constructed by Solleveld are
isomorphisms of pre-$C^*$-algebras this implies that the maps $\tilde{\sigma_\ep}$
define homeomorphisms between the spectra of the $C^*$-algebras
$\mc{C}_\mb{v}$ and $\mc{C}_{\mb{v}^\ep}$, proving (iii).
\end{proof}
%\begin{defn}
%We define a topological space $\widehat{\mc{H}}^{temp}$ with
%underlying set
%\begin{equation}
%\widehat{\mc{H}}^{temp}:=\coprod_{\mb{v}>1}\widehat{\mc{H}_\mb{v}}^{temp}
%\end{equation}
%and topology
%\end{defn}
\subsubsection{The generic $C^*$-algebra}
It is natural at this point to construct the ``generic $C^*$-algebra $\mc{C}=\mc{C}(\mc{R},m)$'',
a continuous bundle of $C^*$-algebras over $\mathbb{R}_{>1}$ which is the analytic counterpart
of the generic Hecke algebra $\mc{H}=\mc{H}(\mc{R},m)$ (this construction was briefly discussed
in \cite[Section 5.2]{Sol}).

Consider the bundle of Banach spaces $\coprod_{\mb{v}>1}\mc{C}_\mb{v}$.
Recall that the underlying Fr\'echet space $\mc{S}(\mc{R})$ of the Schwartz algebras
$\mc{S}_\mb{v}$ is independent of $\mb{v}$ and is contained in
$\mc{C}_\mb{v}$ for all $\mb{v}>1$. Using this fact we construct a
linear subspace (isomorphic to $C_c(\mathbb{R}_{>1})\otimes\mc{S}(\mc{R})$)
of sections $\sigma_{f,s}$ of the above bundle by the linear map defined by
$f\otimes s\to\sigma_{f,s}$, where $\sigma_{f,s}$ is the section defined by
$\sigma_{f,s}(\mb{v}):=f(\mb{v})s\in\mc{C}_\mb{v}$.
By \cite[Proposition 5.6]{Sol} the map $\mb{v}\to\Vert \sigma_{f,s}(\mb{v})\Vert_o$ is
continuous. Hence by \cite[Section 10.3]{Dix} this collection of sections generates
a continuous bundle of $C^*$-algebras $\mc{C}$.
By construction $\mc{C}$
contains all compactly supported continuous maps
$\sigma:\mathbb{R}_{>1}\to \mc{S}(\mc{R})$.
\begin{thm}\label{thm:bundle}
The continuous bundle of $C^*$-algebras $\mc{C}$ is trivial.
In particular, for any $\mb{v}_0>1$ we have a canonical homeomorphism
\begin{align}
\mathbb{R}_{>1}\times\mathfrak{S}_{\mb{v}_0}
&\to\mathfrak{S}\\
\nonumber
(\mb{v},\pi)&\to \tilde{\sigma}_{\log_{\mb{v}_0}(\mb{v}),\mb{v}_0}(\pi)
\end{align}
\end{thm}
\begin{proof}
If we fix $\mb{v}_0>1$ then Theorem \ref{thm:cont}(i) implies that
for all $s\in \mc{S}(\mc{R})$ and $f\in C_c(\mathbb{R}_{>1})$
the function $\sigma^\prime_{f,s}$ defined by
$\mb{v}\to \sigma^\prime_{f,s}(\mb{v}):=f(\mb{v})\phi_{\log_{\mb{v_0}}(\mb{v}),\mb{v_0}}(s)$
defines a section
of $C_0(\mathbb{R}_{>1})\otimes\mc{C}_{\mb{v}_0}$. It is clear that
these sections are dense in $C_0(\mathbb{R}_{>1})\otimes\mc{C}_{\mb{v}_0}$,
and using Theorem \ref{thm:cont} we easily see that the linear map defined by
$\sigma_{f,s}\to\sigma^\prime_{f,s}$ induces a trivialization of
$\mc{C}$. The assertion on the spectrum of $\mc{C}$
follows immediately from this fact.
\end{proof}
\subsection{The generic Plancherel measure}
Given a normalized affine Hecke algebra $(\mc{H},\tau^d)$ we define
a family of positive traces $\{\tau^d_\mb{v}\}_{\mb{v}>1}$ on
$\mc{C}$ by $\tau^d_\mb{v}(h)=d(\mb{v})h_e(\mb{v})$
if $h=\sum_{w\in W}h_wN_w\in\mc{C}$.
\begin{defn} Let $\mb{v}>1$.
The spectral measure of the positive trace $\tau^d_\mb{v}$
(supported on the fibre $\mathfrak{S}_\mb{v}$
of $\mathfrak{S}$) is called the Plancherel measure
of $(\mc{H}_\mb{v},\tau^d_\mb{v})$ and is denoted by
$\nu_{Pl,\mb{v}}$. We denote the family of spectral measures
$\{\nu_{Pl,\mb{v}}\}_{\mb{v}>1}$ by $\nu_{Pl}$
and call this the generic Plancherel measure.
\end{defn}
\subsubsection{The central character of generic discrete series}
The Plancherel measure is given by density functions which are
rational functions with rational coefficients (this will be discussed in more
detail below). A crucial step is the following result, which is an easy 
consequence of \cite[Theorem 5.3]{OpdSol2} (but it should be remarked that, 
since we are only considering ``scaling of the parameters" here, a more elementary 
direct proof by reducing to the graded affine 
Hecke algebra is also possible in the present 
context):
\begin{thm}\label{thm:gencc}
The central character $gcc(\delta)$
of a generic irreducible discrete series representation
$\delta\in\Delta:=\Delta(\mc{R},m)$ is a $W_0$-orbit of
generic residual points $gcc(\delta)\in W_0\backslash\textup{Res}(\mc{R},m)$.
\end{thm}
\subsubsection{The generic formal degree}
Given $\delta\in\Delta$ it is known \cite[Theorem 5.12]{OpdSol2}
that the Plancherel measure
$\mathbb{R}_{>1}\ni\mb{v}\to\nu_{Pl,\mb{v}}(\{\delta_\mb{v}\})$
is a function on $\mathbb{R}_{>1}$ whose value at $\mb{v}\in\mathbb{R}_{>1}$
is the specialization at $\mb{v}$ of a rational function $\de(\delta)$
with rational coefficients. The rational function $\de(\delta)$ is called the
generic formal degree of $\delta$.
\begin{thm}\label{thm:ratfdeg}
Let $(\mc{H},\tau^d)$ be a normalized
affine Hecke algebra of rank $l$ and with $d\in\mb{M}_{(k)}$.
Let $\delta\in\Delta$ be a generic
discrete series character of $\mc{H}$, and let
$gcc(\delta)=W_0r$ with $r\in\textup{Res}(\mc{R},m)$.
\begin{enumerate}
\item[(i)]
For all $r\in\textup{Res}(\mc{R},m)$ the set
$\Delta_{W_0r}(\mc{R},m):=\{\delta\in\Delta(\mc{R},m)\mid gcc(\delta)=W_0r\}$
is a nonempty, finite set.
\item[(ii)]
There exists a nonzero rational constant $D_\delta\in\mathbb{Q}^\times$
such that for all $\mb{v}>1$
\begin{equation}
\de(\delta)(\mb{v})=D_\delta\mu^{(\{r\})}(\mb{v},r(\mb{v}))
\end{equation}
\item[(iii)]
One has $\de(\delta)\in\mb{M}_{(k+l)}$.
\end{enumerate}
\end{thm}
\begin{proof}
The assertion (i) is \cite[Theorem 5.7]{OpdSol2}.

Claim (ii) follows from
\cite[Proposition 3.27(v)]{Opd1} and \cite[Theorem 5.12]{OpdSol2}.

For (iii) we will first show that $\de(\delta)\in\mb{M}$, i.e. that
$\de(\delta)(\mb{v})=\de(\delta)(-\mb{v})=\pm\de(\delta)(\mb{v}^{-1})$.
The first equality follows from the
Euler-Poincar\'e formula \cite[Theorem 4.3]{OpdSol2}
for the generic formal degree since the corresponding
identity for
the generic formal degrees of irreducible representations of finite
type Hecke algebras is true by the explicit results in \cite[Section 13.5]{C}.
(Here we rely on the technique due to \cite{Geck} and discussed
in the proof of \cite[Lemma 4.1]{OpdSol2} to reduce the computation
of formal generic degrees of the cross product of a finite type
Hecke algebra by a finite group of diagram automorphisms to the
case without the cross product).

Let us now look at the second identity.
It is easy to see that there exists a unique ring automorphism
($\mathbb{C}$-linear, but obviously not $\Lb$-linear)
$\phi:\mc{H}\to\mc{H}$
defined on the generators by sending $v\to v^{-1}$, $\omega\to\omega$
(for all $\omega\in\Omega_X$),
and $N_s\to -N_s$ (for all $s\in S$). Given
$\mb{v}\in\mathbb{R}_+\subset\mathbb{C}^\times$ we obtain
an isomorphism
\begin{equation}
\phi:\mc{H}_\mb{v}\to\mc{H}_{\mb{v}^{-1}}
\end{equation}
This algebra isomorphism is an isomorphism of
Hilbert algebras as one easily checks.
Recall that the translations $t_x\in W$ have \emph{even} lengths
(well known, see e.g. \cite[(1.3)]{Opd0}). Hence if we identify
$\mc{Z}_\mb{v}$ with $\mathbb{C}[X]^{W_0}$
by identifying the Bernstein basis element
$\theta_x\in\mc{H}_\mb{v}$ with $x\in X$, then
the restriction
$\phi|_{\mc{Z}}:\mc{Z}_\mb{v}\to\mc{Z}_{\mb{v}^{-1}}$
of $\phi$ to the center yields the identity morphism.
Recall from \cite[Theorem 3.29, Lemma 3.31]{Opd1} that the
spectral measure $\nu_\mb{v}$ of the commutative
Hilbert algebra $\mc{Z}_\mb{v}$ is supported by $W_0$-orbits
of tempered residual cosets. By the above it follows that
(cf. \cite[Corollary 6.8]{Opd1}) the spectral measure $\nu$
satisfies the symmetry relation
\begin{equation}\label{eq:sym}
\nu_\mb{v}(\{W_0r\})=\nu_{\mb{v}^{-1}}(\{W_0r\})
\end{equation}
for all $W_0$-orbits $W_0r$ of residual points.
On the other hand, by \cite[Corollary 3.25]{Opd1} we have,
for all $\mb{v}\in\mathbb{R}_+\backslash\{1\}$, that
\begin{equation}\label{eq:zdeghdeg}
\nu_\mb{v}(\{gcc(\delta)(\mb{v})\})=
\kappa_\ep\mu^{(\{r\})}(\mb{v},r(\mb{v}))
\end{equation}
where the central character $gcc(\delta)(\mb{v})$ is
now viewed as a discrete series
character of $\mc{Z}_\mb{v}$, and where
$\kappa_\ep\in\mathbb{Q}^\times$ is a rational number depending
only on $\ep:=\textup{sign}(\log(\mb{v}))$.
In addition \cite[Lemma 3.31]{Opd1} and \cite[Theorem A.14(ii)]{Opd1}
(also see \cite{Opd3} for classification free proofs)
imply that
\begin{equation}\label{eq:symres}
gcc(\delta)(\mb{v})=gcc(\delta)(\mb{v}^{-1})
\end{equation}
We conclude by (\ref{eq:sym}), (\ref{eq:zdeghdeg}), (\ref{eq:symres})
that
$\mu^{(\{r\})}(\mb{v},r(\mb{v}))=\kappa\mu^{(\{r\})}(\mb{v}^{-1},r(\mb{v}^{-1}))$ for
some nonzero rational number $\kappa$. Combining with
(\ref{eq:mureg}) we see that $\kappa\in\{\pm 1\}$, and by (i) this
implies that
$\de(\delta)(\mb{v})=\pm\de(\delta)(\mb{v}^{-1})$ as required. 

Finally we need to check that the vanishing order of $\de(\delta)$ at
$\mb{v}=1$ equals $k+l$. Using (i) this follows directly from
Proposition \ref{prop:reg}.
\end{proof}
\subsubsection{The tempered central character map}\label{subsub:tempcc}
A residual coset $L$ (over $\mathbb{L}$) determines a parabolic subsystem 
$R_L:=\{\alpha\in R_0\mid \alpha|_L \text{ is constant}\}$.  
Let $T_L\subset T$ be the subtorus whose cocharacter lattice $Y_L$ 
equals $Y_L=Y\cap\mathbb{Q}R_L^\vee$. Let $T^L\subset T$ be the  
subtorus of which $L$ is a coset (i.e. $T^L$ is the subtorus whose 
cocharacter lattice is $Y^L:=Y\cap (R_L)^\perp$.
It is clear from the results stated in paragraph \ref{subsub:rescos} that we have 
$L=rT^L$ for some residual point $r\in T_L$ of the affine Hecke algebra 
$\mc{H}(\mc{R}_L,m_L)$, where $\mc{R}_L$ is the based root datum 
$\mc{R}_L:=(R_L,X_L,R_L^\vee,Y_L,F_L)$ where $X_L$ is the dual lattice 
to $Y_L$ (this is the quotient of the lattice $X$ by the sublattice $X^L$ 
of characters of $T^L$),  $m_L$ is defined as in \ref{subsub:rescos}, and 
$F_L$ the base of $R_L$ such that $R_{L,+}=R_L\cap R_{0,+}$.
We denote by $K_L$ the finite subgroup $T_L\cap T^L$ (this is the group 
of characters of the finite abelian group $X_L/(X\cap \mathbb{Q}R_L)$).

A crucial role in the Fourier transform of the Schwartz algebra completion 
$\mc{S}_\mb{v}$ of $\mc{H}_\mb{v}$ is played by the groupoid 
$\mc{W}_{\Xi,\mb{v}}$ of 
tempered standard induction data via the Fourier transform (see \cite{DeOp1}).
In the present context it is natural to consider $\mc{W}_{\Xi,\mb{v}}$ as the
fiber at $\mb{v}$ of a smooth, \'etale groupoid $\mc{W}_\Xi$ trivially fibered over
$\mathbb{R}_{>1}$ (equipped with the trivial groupoid structure). We will
first describe $\mc{W}_\Xi$ now.

Let $P\subset F_0$ be a subset and let $T^P\subset T$ be the subtorus over $\Lb$ 
with cocharacter lattice $Y^P:=Y\cap R_P^\vee$. 
There exists a unique real structure on $T^P$ given by the
$\mathbb{C}$-anti-linear
$\mathbb{R}$-algebra automorphism $\Lb[X^P]\to\Lb[X^P]$ defined by $v\to v$ and
$x\to -x$ for all $x\in X^P$. We denote by $T^P(\mathbb{R})\to\mathbb{R}^\times$
the corresponding set of $\mathbb{R}$-points of $T^P$
(thus $T^P(\mathbb{R})\approx \mathbb{R}^\times\times (S^1)^{(|F_0|-|P|)}$).
Let $T^{P,r}\subset T^P(\mathbb{R})$ be the subset of
points that lie above $\mathbb{R}_{>1}\subset\mathbb{R}^\times$. We define
$T_P^r\to\mathbb{R}_{>1}$ similarly. We equip $T^{P,r}\to \mathbb{R}_{>1}$ and
$T_P^r\to \mathbb{R}_{>1}$ with the analytic topology.

The set of objects $\Xi$ is a
disjoint union of components $\Xi_{(P,\delta)}$, where
$P$ runs over all subsets $P\subset F_0$ and 
$\delta\in\Delta_{P,m}$. The elements of the component
$\Xi_{(P,\delta)}$ are triples $(P,\delta,t)$ with
$t\in T^{P,r}$. We denote by $\Xi_\mb{v}$ the subset of objects
$(P,\delta,t)$ with $t\in T^{P,r}_\mb{v}$.
Thus $\Xi_{(P,\delta,\mb{v})}$ is a copy of the compact
torus $T^P_{\mb{v},u}$ for each $\mb{v}>1$.

The arrows of $\mc{W}_\Xi$ are defined as follows.
Let $K_P=T_P^r\cap T^{P,r}\to\mathbb{R}_{>1}$
(thus $K_{P,\mb{v}}$ is a finite abelian group for all $\mb{v}>1$).
Given $P,Q\subset F_0$, let $W(P,Q):=\{w\in W_0\mid w(P)=Q\}$.
Given objects $\xi=(P,\delta,t)\in\Xi_\mb{v}$ and
$\eta=(Q,\sigma,s)\in\Xi_\mb{v^\prime}$ we put
\begin{equation}
\textup{Hom}_{\mc{W}}(\xi,\eta)=\{(k,w)\in K_Q\times W(P,Q)
\mid w(P)=Q,\,(\delta^w)_{k^{-1}}\simeq \sigma,\,kw(t)=s\}
\end{equation}
The composition of arrows is defined by $(k,u)\circ(l,v)=(u(l)k,uv)$.
Observe that the arrows respect the fibers $\Xi_\mb{v}$ of the natural map
$\Xi\to\mathbb{R}_{>1}$. Hence if we equip $\mathbb{R}_{>1}$ with the trivial
groupoid structure then the above map extends in the obvious way to a homomorphism
$\mc{W}_{\Xi}\to\mathbb{R}_{>1}$ of groupoids.
We introduce ``scaling isomorphisms'' between the fibers of the groupoid $\mc{W}_\Xi$
\begin{equation}\label{eq:scalexi}
\phi_{\ep,\mb{v}}:\mc{W}_{\Xi,\mb{v}}\to\mc{W}_{\Xi,\mb{v}^\ep}
\end{equation}
for $\ep>0$ which are identical on the sets of morphisms (in the obvious sense,
viewing a morphism as an element of some $K_Q\times W(P,Q)$) and which
act on objects by $\phi_{\ep,\mb{v}}(P,\delta,t):=(P,\delta,t^\ep)$,
where $t^\ep\in T^{P,r}$ is defined by $v(t^\ep)=\mb{v}^\ep$ and for all
$x\in X^P$, $x(t^\ep)=x(t)$.
In particular the fibration $\mc{W}_\Xi\to\mathbb{R}_{>1}$ is trivial.

The restriction $\mc{F}_\mc{Z}$ of the Fourier transform of $\mc{S}_\mb{v}$
to the center
$\mc{Z}_{\mc{S},\mb{v}}$ of $\mc{S}_\mb{v}$ defines
an isomorphism \cite[Corollary 5.5]{DeOp1}
\begin{equation}\label{eq:FZ}
\mc{F}_\mc{Z}:\mc{Z}_{\mc{S},\mb{v}}\to C^\infty(\Xi_\mb{v})^{\mc{W}}
\end{equation}
between $\mc{Z}_{\mc{S},\mb{v}}$ and the ring of smooth, $\mc{W}$-invariant
functions on $\Xi_{\mb{v}}$. This isomorphism extends to
an isomorphism of $C^*$-algebras, and shows  that we may define
a continuous map
\begin{equation}\label{eq:cctempv}
cc^{temp}_\mb{v}:\mathfrak{S}_\mb{v}\to
\mc{W}\backslash\Xi_\mb{v}
\end{equation}
which sends a subrepresentation $\pi<\pi(\xi)$ of the standard tempered
representation $\pi(\xi)$ to the orbit $\mc{W}\xi\in\mc{W}\backslash\Xi_\mb{v}$
(``disjointness'', see \cite[Corollary 5.6]{DeOp1}).
This map is clearly compatible with the scaling maps (\ref{eq:scaletemp}) and
(\ref{eq:scalexi}). In particular these maps are the fibers of a continuous map
\begin{equation}\label{eq:cctempgen}
cc^{temp}:\mathfrak{S}\to
\mc{W}\backslash\Xi
\end{equation}
Another consequence of (\ref{eq:FZ}) is the fact that
the components $\mathfrak{S}_{(P,\delta,\mb{v})}$
of the tempered spectrum \cite[Corollary 5.8]{DeOp1} of
$\mc{H}_\mb{v}$
are parametrized by association classes of pairs
$(P,\delta)$ with $P\subset F_0$ a standard parabolic subset and
$\delta\in\Delta_{P}$. Thus the same holds for the spectrum
of $\mc{C}$:
\begin{equation}\label{eq:cccomp}
\mathfrak{S}=
\coprod_{(P,\delta)/\sim}\mathfrak{S}_{(P,\delta)}
\end{equation}
In addition, for each $\mb{v}>1$ there exist dense open subsets
$\mathfrak{S}_{(P,\delta,\mb{v})}^\prime\subset \mathfrak{S}_{(P,\delta,\mb{v})}$
and a dense open subsets
$\Xi^\prime_{(P,\delta,\mb{v})}\subset\Xi_{(P,\delta,\mb{v})}$ such that
\begin{equation}\label{eq:ccgen}
cc^{temp}_{\mb{v}}|_{\mathfrak{S}_{(P,\delta,\mb{v})}^\prime}:\mathfrak{S}_{(P,\delta,\mb{v})}^\prime\to
\mc{W}_{(P,\delta)}\backslash\Xi^\prime_{(P,\delta,\mb{v})}
\end{equation}
is a homeomorphism (see \cite[Theorem 4.39]{Opd1};
the map $[\pi]$ (loc. cit.) yields the inverse homeomorphism).
It is obvious from the definitions of \cite[paragraph 4.5.2]{Opd1}
that the sets $\mathfrak{S}_{(P,\delta,\mb{v})}^\prime$ and $\Xi^\prime_{\mb{v}}$
are compatible with the scaling isomorphisms
(\ref{eq:scaletemp}), (\ref{eq:scalexi}). Hence we obtain
open dense subsets $\mathfrak{S}_{(P,\delta)}^\prime\subset \mathfrak{S}_{(P,\delta)}$,
and open dense subsets $\Xi^\prime_{(P,\delta)}\subset\Xi_{(P,\delta)}$,
and a homeomorphism
\begin{equation}\label{eq:ccgengl}
cc^{temp}|_{\mathfrak{S}_{(P,\delta)}}:\mathfrak{S}_{(P,\delta)}^\prime\to
\mc{W}_{(P,\delta)}\backslash\Xi^\prime_{(P,\delta)}
\end{equation}
\subsubsection{The algebraic central character map}
The restriction $p_Z^{temp}$ of the algebraic central character map
\begin{equation}
p_Z:\textup{Irr}(\mc{H})\to W_0\backslash T(\mathbb{C})
\end{equation}
(recall that $T(\mathbb{C})=
\textup{Hom}(\mathbb{Z}\times X,\mathbb{C}^\times)\simeq \mathbb{C}^\times\times\textup{Hom}(X,\mathbb{C}^\times)$) 
to $\mathfrak{S}\subset\textup{Irr}(\mc{H})$
factors via $cc^{temp}$. Indeed, the map
\begin{align*}
p_{\Xi}:\mc{W}\backslash\Xi&\to W_0\backslash T(\mathbb{C})\\
\mc{W}\backslash\Xi_\mb{v}\ni\mc{W}(P,\delta,t)&\to W_0(gcc(\delta)(\mb{v})t)
\end{align*}
is well defined and it is obvious that we have a commuting diagram
as follows:
\begin{equation}\label{cd}
\begin{CD}
\mathfrak{S}@>cc^{temp}>>\mc{W}\backslash\Xi\\
@V{p_Z^{temp}}VV    @VV{p_\Xi}V\\
W_0\backslash T(\mathbb{C})@=W_0\backslash T(\mathbb{C})
\end{CD}
\end{equation}
Now we describe the image $S:=p_Z^{temp}(\mathfrak{S})$ of $p_Z^{temp}$,
and its components. If $L$ is a residual coset over $\mathbf{L}$
(see section \ref{par:polezero}) we can write it in the form $L=rT^P$ where
$r$ is a $(\mc{R}_L,m_L)$-residual point over $\mb{L}$ by
Theorem \ref{thm:genrescos}(ii). We define the corresponding
\emph{tempered residual coset} $L^{temp}$ to be the subset
\begin{equation}
L^{temp}:=\bigsqcup_{\mb{v}>1}r(\mb{v})T^P_{\mb{v},u}\subset T(\mathbb{C})
\end{equation}
of its $\mathbb{C}$-points. We equip $L^{temp}$ with the relative topology
with respect to the analytic topology on the complex variety $T(\mathbb{C})$.
The set of points of $L^{temp}$ that lie over $\mb{v}>1$ is denoted by
$L^{temp}_\mb{v}$. This is a tempered residual coset in the sense of \cite{Opd1}.
We remark that $L^{temp}$ is independent of the chosen decomposition $L=rT^P$.
\begin{thm}\label{thm:S}
The image $S=p_Z^{temp}(\mathfrak{S})$ of $p_Z^{temp}$ is equal the union
\begin{equation}
S:=\bigsqcup_{L \mathrm{\ residual\ over\ }\mb{L}} L^{temp}\subset
W_0\backslash T(\mathbb{C})
\end{equation}
For all $L\in\mc{L}$ the subset
$W_0L^{temp}\subset W_0\backslash T(\mathbb{C})\subset S$ is a component of $S$.
\end{thm}
\begin{proof}
Consider $S_{(P,\delta)}:=p_Z^{temp}(\mathfrak{S}_{(P,\delta)})$.
By the above commuting diagram we can rewrite this as
$S_{(P,\delta)}=p_\Xi(\mc{W}_{(P,\delta)}\backslash\Xi_{(P,\delta)})$. Using
the definition of $p_\Xi$ and of $\Xi_{(P,\delta)}$, Theorem \ref{thm:genrescos} and
Theorem \ref{thm:gencc} it is clear that
\begin{equation}
S_{(P,\delta)}=\bigsqcup_{\mb{v}>1}W_0(r(\mb{v})T^P_{\mb{v},u})\subset W_0\backslash T(\mathbb{C})
\end{equation}
where $r$ is an $(\mc{R}_P,m_P)$-residual point of $T_P$ such that $gcc(\delta)=W_Pr$.
By Theorem \ref{thm:genrescos}(ii) we see that $L_{(P,r)}=rT^P$ is a residual coset, and
clearly $S_{(P,\delta)}=W_0L^{temp}_{(P,r)}$. This proves that
$S=\cup_{(P,\delta)/\sim}S_{(P,\delta)}$. The subsets $S_{(P,\delta)}$ are clearly
connected. We will now show that these sets are in fact the components of $S$ (thus proving
the remaining assertions of the Theorem). It is enough
to show that if $L_1=r_1T^{P_1}$ and $L_2=r_2 T^{P_2}$ are residual cosets
then their tempered parts do not intersect each other unless possibly if
$L_1$ and $L_2$ are in the same $W_0$-orbit.
It is enough to check this for the fibers $L_{i,\mb{v}}$ over $\mb{v}>1$, and
this is the content of Lemma \ref{lem:app} below.
\end{proof}
\begin{lem}\label{lem:app}
Let $L_1,L_2\in\mc{L}_\mb{v}$ for some $\mb{v}>1$. Then
\begin{equation}
L^{temp}_1\cap L^{temp}_2\not=\emptyset\Longrightarrow
L_1=w(L_2) \mathrm{\ for\ some\ }w\in W_0
\end{equation}
\end{lem}
\begin{proof}
This is \cite[Theorem A.18]{Opd1}, but since this fact did not
play a role in \cite{Opd1} the proof was indicated only very briefly there.
We will fill in the details now.

Write $L_1=r_1T^{P_1}_\mb{v}$ and $L_2=r_2T^{P_2}_\mb{v}$ with $r_i\in T_{P_1,\mb{v}}$
residual points for $(\mc{R}_{P_i},m_{P_i})$, where $R_{P_i}\subset R_0$ are parabolic
subsets of roots. Write $c_i=|r_i|$. Suppose that $t\in L^{temp}_1\cap L^{temp}_2$.
Then clearly $|t|=c_1=c_2$. Unfortunately this real point is not necessarily the
center of a real residual affine subspace (in the sense of \cite{HO}, also see
\cite[Appendix A]{Opd1} and \cite{Opd3}) of the vector group $V:=\textup{Hom}(X,\mathbb{R}_+)$,
and we need to manipulate a bit (by shrinking $R_0$ in a suitable way)
to arrive at such a favorable situation. We identify the real vector space 
$\mathbb{R}\otimes_\mathbb{Z} Y$ with $V$ via the exponential homeomorphism
$\exp(\xi\otimes y)(x):=\exp(\xi(x(y)))$.

Consider the subset $R_{12}\subset R_0$ consisting of the
roots $\alpha\in R_0\cap R_m$ such that $\alpha(L_1\cap L_2)\subset\mathbb{R}_+$ together
with the roots $\beta\in R_0\backslash R_m$ such that $\pm\alpha(L_1\cap L_2)\subset\mathbb{R}_+$.
It is clear that $R_{12}\subset R_0$ is a subsystem of roots.
We form the affine root datum $\mc{R}_{12}=(X,R_{12},Y,R_{12}^\vee)$.
We equip $R_{12}$ with the restriction to $R_{12}$ of the half integral functions
$m_\pm$ on $R_0$.
Then this corresponds to a multiplicity function $m_{12}$ for $\mc{R}_{12}$.
We thus obtain a pair $(\mc{R}_{12},m_{12})$ consisting of an affine root datum
and a multiplicity function. The first important property of $(\mc{R}_{12},m_{12})$
is the fact that $L_1$ and $L_2$ are still residual cosets
with respect to $(R_{12},m_{12})$. Indeed, the pole sets and the zero sets
of roots for either $L_1$ or $L_2$ are clearly subsets of $R_{12}$ (see \ref{subsub:rescos}).

Put $s_i=tc_i^{-1}$ (recall that $t\in L^{temp}_1\cap L^{temp}_2$).
By definition of $R_{12}$ the unitary points $s_i$ satisfy $\alpha(s_i)=1$ for all
$\alpha\in R_{12,m}$.
Hence the points $s_i$ are fixed by $W(R_{12})$, and $L_i^\prime:=s_i^{-1}L_i=c_iT^{P_i}_\mb{v}$
is also a $(\mc{R}_{12},m_{12})$-residual coset.
We define $L_i^V:=L^\prime_i\cap V$ with $V:=\textup{Hom}(X,\mathbb{R}_+)$.

We have obtained two residual real affine subspaces $L_i^V$ of $V$ with respect to the
root system $R_{12}\subset R_0$ and a multiplicity function $m^h_{12}$ for a suitable $h$
(see \cite[Theorem A.7]{Opd1}).
Moreover, the center of $L_i^V$ equals $c_i$ and, as we have seen above,
these centers coincide $c_1=c_2$.
Now we are reduced to showing that the set of $W(R_{12})$-orbits of $(R_{12},m^h_{12})$
residual real affine subspaces is in bijection with the set of $W(R_{12})$-orbits of
the centers of residual real affine subspaces.

If the parameters function $m^h_{12}$ is constant on $R_{12}$ then this assertion
follows from the Bala-Carter theorem on the classification of nilpotent orbits of
$\mathfrak{g}(R_{12},\mathbb{C})$, see \cite[Remark 7.5]{Opd3}. Observe that this
problem for affine residual subspaces is completely reducible, and thus we may assume
that $R_{12}$ is irreducible without loss of generality.

If $R_{12}$ is simply laced we are done by the above
application of the Bala-Carter theorem. If $R_{12}$ is of type $\textup{B}_n$ or $\textup{C}_n$ then
this was proved by Slooten \cite[Theorem 1.5.3, Theorem 1.5.5]{SlootenThesis}. For
$R_{12}$ of type $\textup{G}_2$ it is a trivial verification, and for type $\textup{F}_4$ is a more involved
case by case verification which was done by Slooten (private communication, 2001).
\end{proof}
\begin{prop}\label{prop:S}
\begin{enumerate}
\item[(i)] The finite map $p_Z^{temp}:\mathfrak{S}\to S$
maps components to components. The components of $S$
are of the form $p_Z^{temp}(\mathfrak{S}_{(P,\delta)})=S_{(P,\delta)}$.
\item[(ii)] For each component $S_{(P,\delta)}$ of $S$ there exists a
residual coset $L\in\mc{L}$, unique up to the action of $W_0$ on $\mc{L}$,
such that $S_{(P,\delta)}=W_0\backslash W_0L^{temp}$.
\item[(iii)]
Let $L\in\mc{L}$ be as in (ii). The map $q_L:L^{temp}\to S_{(P,\delta)}$ given by
$q_L(t)=W_0t$ is an identification map whose fibers are the $N_{W_0}(L)$-orbits
in $L^{temp}$. Hence $S_{(P,\delta)}\approx N_{W_0}(L)\backslash L^{temp}$.
\item[(iv)]
The restriction of $p_Z^{temp}$ to $\mathfrak{S}_{(P,\delta)}^\prime$ is a finite
covering map onto an open, dense subset
$S^\prime_{(P,\delta)}\subset S_{(P,\delta)}$ of $S_{(P,\delta)}$.
\item[(v)] Let $L\in\mc{L}$ be as in (ii), and let
$(L^{temp})^\prime\subset L^{temp}$ be the inverse image of
$\mathfrak{S}_{(P,\delta)}^\prime\subset\mathfrak{S}_{(P,\delta)}$
with respect of the quotient map $q_L$ of (iii).
Then $N_{W_0}(L)$ acts freely on $(L^{temp})^\prime$. This gives
$S_{(P,\delta)}^\prime$ the structure of a smooth manifold.
\end{enumerate}
\end{prop}
\begin{proof}
The first assertion (i) was shown in the first part of the proof of Theorem \ref{thm:S}.
In this argument it was also explained that $S_{(P,\delta)}$ has the form
$S_{(P,\delta)}=W_0\backslash W_0L^{temp}$
for a residual coset $L$ of the form $L=L_{(P,r)}$, implying (ii).
Assertion (iii) is obvious, and (iv)
follows directly from \cite[Corollary 4.40]{Opd1} and Theorem \ref{thm:bundle}.
Finally (v) follows from \cite[Lemma 4.35, Lemma 4.36]{Opd1}.
\end{proof}
\subsubsection{The generic Plancherel measure $\nu_{Pl}$ and regularizations of the $\mu$-function}\label{par:specmeas}
\begin{lem}\label{lem:plan}
Let $L\in\mc{L}$. There exists a $\ep_L\in\{\pm 1\}$
such that the function $m_L:=\ep_L\mu^{(L)}|_{L^{temp}}$ is
a smooth, $N_{W_0}(L)$-invariant positive function on $L^{temp}$.
\end{lem}
Consider a component $\mathfrak{S}_{(P,\delta)}$ of $\mathfrak{S}$.
As we have seen in Theorem \ref{thm:S} its image
$p_Z^{temp}(\mathfrak{S}_{(P,\delta)})=S_{(P,\delta)}\subset S$ is a component
of $S$. From Corollary \ref{prop:S} we see that there exists a residual
coset $L$ such that $S_{(P,\delta)}=W_0\backslash W_0L^{temp}$.
Consider the smooth $\mathbb{R}_{>1}$-family of volume forms $\nu_L$ on $L^{temp}$
which has density function $m_L$ with respect to the family $d^L(t)$ of
normalized $T^L_{\mb{v},u}$-invariant volume forms on $L^{temp}$.
Let $\nu_L^\prime$ denote the restriction of $\nu_L$ to
$(L^{temp})^\prime$. By Proposition \ref{prop:S}(v) there exists a unique
smooth $\mathbb{R}_{>1}$-family of volume forms $\nu_S^\prime$ on $S_{(P,\delta)}^\prime$
such that $\nu_L^\prime=q_L^*(\nu_S^\prime)$. Using Proposition
\ref{prop:S}(iv) we define a smooth $\mathbb{R}_{>1}$-family of volume forms on
$\mathfrak{S}_{(P,\delta)}$
by $\nu_{\mathfrak{S}}^\prime:=(p^{temp}_Z)^*(\nu_S^\prime)$.
\begin{thm}\label{thm:plancherel}
Let $P\subset F_0$, $\delta\in\Delta_P$
and let $L\in\mc{L}$ such that $S_{(P,\delta)}=W_0\backslash W_0L^{temp}$.
There exists a constant $a_{(P,\delta)}\in\mathbb{Q}_+$ such that the
restriction of $\nu_{Pl}$ to $\mathfrak{S}_{(P,\delta)}$ is given by
$\nu_{Pl}^\prime=a_{(P,\delta)}i_*(\nu_{\mathfrak{S}}^\prime)$, where
$i$ denotes the open embedding
$i:\mathfrak{S}_{(P,\delta)}^\prime\to \mathfrak{S}_{(P,\delta)}$.
\end{thm}
\begin{proof}
This follows directly from \cite[Theorem 4.43]{Opd1} and
\cite[Theorem 5.12]{OpdSol2}.
\end{proof}
\section{The spectral transfer category}
In this section we define ``spectral transfer morphisms" between 
normalized affine Hecke algebras. Very briefly, these are finite morphisms 
between the spectra of the centers of the underlying generic affine Hecke algebras,  
which ``respect" the $\mu$-functions. We show that this gives rise to a category 
whose objects are normalized affine Hecke algebras. 
\subsection{Spectral transfer maps} 
\subsubsection{Definition of spectral transfer maps}
In view of Proposition \ref{prop:mucan} the following definition makes sense.
Let $\mc{H}_i=\mc{H}(\mc{R}_i,m_i)$ ($i=1,2$) be two
(possibly extended, not necessarily semisimple) affine Hecke algebras, 
with normalizing elements $d_i\in\mb{M}$.
Recall from 
paragraph \ref{subsub:tempcc} that given a residual coset $L$ of $T_2$ 
we can write $L=rT^L$ with $r\in T_L$ a residual point. 
Let $K_L:=T_L\cap T^L$, a $N_{W_{2,0}}(L)$-stable 
finite abelian group acting faithfully on $L$, where $N_{W_{2,0}}(L)$ is 
the stabilizer of $L$ in $W_{2,0}$. Then  $L\cap T_L=K_L r$. 
Let $K_L^n:=K_L\cap N_{W_{2,0}}(L)/Z_{W_{2,0}}(L)$ (with $Z_{W_{2,0}}(L)$ 
the pointwise stabilizer
of $L$),  and $\dot{K}^n_L\subset N_{W_{2,0}}(L)$ its  
inverse image in $N_{W_{2,0}}(L)$. 
Then $\dot{K}_L^n= W(R_L)\cap N_{W_{2,0}}(L)=N_{W(R_L)}(L)$, and  
$N_{W_{2,0}}(L)$ acts on $L_n:=L/K^n_L$. 
Clearly $\mu_{\mc{R},m,d}^{(L)}$ is $K_L^n$-invariant, and thus 
descends to a rational function $\mu_{\mc{R},m,d}^{(L_n)}$ on $L_n$.
\begin{defn}\label{defn:spectralTM} 
We first assume that $(\mc{R}_1,m_1)$ is (semi)-standard. 
By a \emph{(semi)-standard spectral transfer map} from $\mc{H}_1$ to $\mc{H}_2$  
we mean a morphism $\phi_T:T_1\to L_n:=L/K^n_L$ over $\mb{L}$, where $L\subset T_2$ 
denotes a residual coset, such that:
\begin{enumerate}
\item[(T1)] $\phi_T$ is finite.
\item[(T2)] $\phi_T(e)\in (T_L\cap L)/K^n_L$, and if we declare $\phi_T(e)$
to be the unit of the $T^L_n:=T^L/K^n_L$-torsor $L_n$ then $\phi_T$ is a homomorphism of 
algebraic tori over $\mb{L}$.
\item[(T3)] There exists an $a\in\mathbb{C}^\times$ such 
that $\phi_T^*(\mu_{\mc{R}_2,m_2,d_2}^{(L_n)})=a\mu_{\mc{R}_1,m_1,d_1}$.
\end{enumerate}
If $(\mc{R}_1,m_1)$ is not semi-standard we impose the
additional condition: 
\begin{enumerate}
\item[(T4)] The $\phi_T$-image of a $W_{1,0}$-orbit is
contained in a single $N_{W_{2,0}}(L)$-orbit.
\end{enumerate}
\end{defn}
\begin{prop}\label{prop:restrans}
A spectral transfer map $\phi_T$ 
has the following property: if $L_1\subset T_1$ is a residual coset with
respect to $(\mc{R}_1,m_1)$ then there exists a residual coset 
$L_2\subset L$ such that $\phi_T(L_1)=K_L^nL_2/K_L^n$. 
In this case we also have 
$\phi_T(L_1^{temp})=\overline{L_2^{temp}}:=K^n_LL_{2}^{temp}/K^n_{L}$. Conversely,
if $L_2\subset L$ is a residual coset then $\phi^{-1}_T(\overline{L_{2,n}})$
(with $\overline{L_{2,n}}:=K^n_LL_{2}/K^n_{L}$) is a finite union of residual cosets of $T_1$ (and analogously
for tempered forms of residual cosets of $T_2$).
\end{prop}
\begin{proof}
Clear by the fundamental property Theorem \ref{thm:genrescos}(i)
of residual cosets. The assertion about the tempered forms of the
residual cosets follows easily from (T1) and (T2).
\end{proof}
\begin{prop}\label{prop:ratcon}
The constant $a$ of (T3) is always rational.
\end{prop}
\begin{proof}
Let $\phi_T(T_1)=L_n$ with $L\subset T_2$ a residual coset.
By Theorem \ref{thm:genrescos}(ii),
equation (T3) is equivalent to
\begin{equation}\label{eq:restr}
\phi_T^*(\mu_{\mc{R}_2,m_2,d_2}^{(L_n)})=a\mu_{\mc{R}_1,m_1,d_1}
\end{equation}
Now consider formula (\ref{eq:musplit}) for $\mu_{\mc{R}_2,m_2,d_2}^{(L)}$.
By \cite[Theorem 3.27(v)]{Opd1} and Proposition \ref{prop:reg}
we have, if $d_2\in\mb{M}_k$ and $l=\textup{codim}(L)$, that
$\mu^{(\{r\})}_{\mc{R}_{P_2},m_{P_2},d_2}\in\mb{M}_{k+l}$.
Obviously (\ref{eq:restr}) implies that $d_1\in\mb{M}_{k+l}$ as well, and
thus after dividing by $d_1$ we can specialize the resulting identity
\begin{equation}\label{eq:restr2}
\phi_T^*(\mu_{\mc{R}_2,m_2,d_2/d_1}^{(L_n)})=a\mu_{\mc{R}_1,m_1,1}
\end{equation}
at the generic point of the fiber $T_{1,\mb{v}}$ of $T_1$ at $\mb{v}=1$.
But from the definition of $\mu_{\mc{R}_1,m_1,1}$ and from
equation (\ref{eq:musplit}) for $\mu_{\mc{R}_2,m_2,d_2/d_1}^{(L)}$ we see that
this yields an equation of constants of the form $c_1=ac_2$
with $c_i\in\mathbb{Q}^\times$, hence the conclusion.
\end{proof}
\begin{prop}\label{prop:z} If $\phi_T$ is a spectral transfer map from 
$(\mc{H}_1,\tau^{d_1})$ to $(\mc{H}_2,\tau^{d_2})$ such that  
$\phi_T(T_1)=L_n$ with $L\subset T_2$ a residual coset. Then
$\forall\,w_1\in W_{1,0}\ \exists\,w_2\in N_{W_{2,0}}(L)$ such that
$\phi_T\circ w_1=w_2\circ \phi_T$.
\end{prop}
\begin{proof}
In the non-semi-standard case there is nothing to prove (by (T4)), so let us
assume that we are in the semi-standard case.
We may assume that $T^L=T^P$ for some standard parabolic subset
$P\subset F_{2,0}$. Choose $\delta\in\Delta_P$
such that $S_{(P,\delta)}=N_{W_{2,0}}(L)\backslash L^{temp}$.
Recall from Theorem \ref{thm:plancherel} that the Plancherel measure
on $S_{(P,\delta)}$ is given by the invariant smooth density function
(with respect to the normalized $T^P_u$-invariant measures $d^L(t)$)
$m_L$ on $L^{temp}$. Let us denote by $m_\Xi$ the lift of $m_L$ to
$\Xi_{(P,\delta)}^{temp}$ via the isomorphism $j_L:\Xi_{(P,\delta)}
\to L$ given by $j_L((P,\delta,t))=rt$,
where we have chosen a generic residual point $r\in T_P$
such that $W_Pr$ is the generic central character of $\delta$.
Then the rational function $c(\xi)$ on $\Xi_{(P,\delta)}$ defined
by
\begin{equation}\label{dfn:c}
c(\xi):=\prod_{\alpha\in R_{0,+}\backslash R_{P,+}}c_\alpha(j_L(\xi))
\end{equation}
is independent of the choice of $r$ as above. We have
(see \cite[Theorem 4.3, Proposition 9.8]{DeOp1})
(compare to (\ref{eq:musplit}) and Theorem \ref{thm:ratfdeg}):
\begin{equation}\label{eq:transSigma}
m_\Xi(\xi)=v^{-2m_{W_0}(w^P)}|\mathcal{W}_P/\mathcal{K}_P|^{-1}\de(\delta)(c(\xi)c(w^P(\xi)))^{-1}
\end{equation}
where $w^P=w_0 w_P^{-1}\in W^P$ is the longest element.

Recall that
$c(\xi)^{-1}$ is smooth on $\Xi_{(P,\delta)}^{temp}$ and
$c(w^P(\xi))=\overline{c(\xi)}$ on $\Xi_{(P,\delta)}^{temp}$
(by \cite[Proposition 9.8]{DeOp1}).
Hence the zero set of
$m_\Xi$ on $\Xi_{(P,\delta)}^{temp}$ is the same as the
zero set of $c(\xi)^{-1}$ on $\Xi_{(P,\delta)}^{temp}$.
This zero set is a union of connected components of
certain hypersurfaces of the form $(P,\alpha)(\xi)=\mathrm{constant}$,
where $\alpha\in R_{2,0}\backslash R_{2,P}$. These hypersurfaces are called
mirrors on $\Xi_{(P,\delta)}^{temp}$. By a classical result of Harish-Chandra
(see \cite[Theorem 4.3]{DeOp2} in the present context), for any
$\xi\in\Xi_{(P,\delta)}^{temp}$, the set of restricted roots $(P,\alpha)$
such that $\xi$ belongs to a $(P,\alpha)$-mirror forms a root system
$\mathfrak{R}_\xi$, and all elements of $W(\mathfrak{R}_\xi)$ belong
to the isotropy group $\mc{W}_\xi$ of $\xi$ for the action of $\mc{W}$.

Let $L_1=T_1$ (which is a residual coset of $T_1$) and let
$(P_1,\delta_1)$ be the unique pair $P_1=F_{1,0}$, $\delta_1=1$ associated to $L_1$.
We will identify $\Xi_{(P_1,\delta_1)}$ and $T_1$ via $j_{L_1}$.
The assumption on $m_\pm(\alpha)$ implies that there exists a point
$o_1\in T_{1,u}=L_1^{temp}$
such that for all $\alpha\in R_{1,0}$ there exists a
$(P_1,\alpha)$-mirror containing $o_1$. In particular, $o_1$ is
$W_1$-invariant, and the action of $W_1$ on $T_1^{temp}$ is
generated by the mirror reflections in the mirrors through
$o_1$ (see Proposition \ref{prop:musym}).
Choose $\xi_o\in\Xi_{(P,\delta)}^{temp}$ such that $K_L^nj_L(\xi_o)=\phi_T(o_1)$.
Observe that (T3) and (\ref{dfn:c}) imply that for each mirror
$M_{(P_1,\alpha)}\subset L_1^{temp}=\Xi^{temp}_{(P_1,\delta_1)}$ there exists
a mirror $M_{(P,\beta)}\subset \Xi_{(P,\delta)}^{temp}$ through $\xi_o$
such that $\phi_T(M_{(P_1,\alpha)})=K_L^nj_L(M_{(P,\beta)})$.
Let $\mf{r}_{(P,\beta)}\in W(\mathfrak{R}_\xi)$ be the reflection in $M_{(P,\beta)}$.
Then (T1) and (T2) imply that we can choose an open $r_\alpha$-invariant 
neighborhood $V\subset T_1$ of $o_1$ such that $\phi_T$ restricts 
to an isomorphism on $V$, and a  $\mf{r}_{(P,\beta)}$-invariant
open neighborhood $U\ni\xi_o$ such that the covering $\pi:L\to L_n$
restricts to an isomorphism on $j_L(U)$ and such that $\pi(j_L(U))=\phi_T(V)$, 
and such that
$\mf{r}_{(P,\beta)}|_U=(\phi_T^{-1}\circ\pi\circ j_L)^{-1}\circ r_\alpha\circ(\phi_T^{-1}\circ\pi\circ j_L)|_U$.
By the previous paragraph $\mf{r}_{(P,\beta)}\in \mc{W}$. This means by definition
that there exists a $k\in K_P$ and a $n\in W_{2,0}$ with the following properties: $n(P)=P$,
$(\delta^n)_{k^{-1}}\simeq\delta$ and $\mf{r}_{(P,\delta)}(\xi)=(k\times n)(\xi)$ for all
$\xi=(P,\delta,t)\in \Xi_{(P,\delta)}$. In other words, we have for all $t\in T^P$ that
$j_L(\mf{r}_{(P,\beta)}(\xi))=rkn(t)$. Since $(\delta^n)_{k^{-1}}\simeq\delta$ we have
$wr=k^{-1}n(r)$ for some $w\in W(R_{2,P})$. Consequently
$w(j_L(\mf{r}_{(P,\beta)}(\xi)))=n(rt)=n(j_L(\xi))$, and we see
that the action of $\mf{r}_{(P,\beta)}$ on $\Xi_{(P,\delta)}$ is given by
the Weyl group element $w^{-1}n\in N_{W_{2,0}}(L)$ via the isomorphism $j_L$.
In view of the above we obtain that
$\phi_T\circ r_\alpha=w^{-1}n\circ\phi_T$ with $w^{-1}n\in N_{W_{2,0}}(L)$. Since 
we can do this construction for any $\alpha\in R_{1,0}$ the claim is proved.
\end{proof}
\begin{cor}\label{cor:z}
A spectral transfer map $\phi_T$ from 
$\mc{H}_1$ to $\mc{H}_2$ induces a morphism
\begin{align}
\phi_Z:W_{1,0}\backslash T_1(\mathbb{C})&\to W_{2,0}\backslash T_2(\mathbb{C})\\
\nonumber W_{1,0}(t)&\to W_{2,0}(\phi_T(t))
\end{align}
We have $\phi_Z(S_1)\subset S_2$, and
the restriction $\phi_Z^{temp}$ of $\phi_Z$ to $S_1$ is a
finite map sending components to components.
\end{cor}
\begin{proof}
Let $L\subset T_2$ be a residual coset such that $\phi_T(T_1)=L^n=L/K_L^n$.
The existence of $\phi_Z$ follows from Proposition \ref{prop:z} and 
the fact that $K_L^n\subset W_{2,0}(L):=N_{W_{2,0}}(L)/Z_{W_{2,0}}(L)$.
This map is finite by (T1).
By Proposition \ref{prop:S} and Proposition \ref{prop:restrans} we
see that $\phi_Z$ maps components of $S_1$ onto components of $S_2$.
\end{proof}
\subsubsection{Reduction to irreducible types in the general semisimple case}\label{par:products}
Suppose that $\phi_T$ is a spectral transfer map from $\mc{H}_1$ to $\mc{H}$,
with $\mc{H}=\mc{H}(\mc{R},m)$ and $\mc{H}_1=\mc{H}(\mc{R}_1,m_1)$ both 
semi-simple affine Hecke algebras.   
We will show in this paragraph that we can essentially reduce to a product of 
situations where $\mc{R}$ is irreducible.
Moreover we will see below (cf. Proposition \ref{prop:excel}(3)), 
that if $\mc{H}(\mc{R},m)$ is semi-standard and irreducible,  
then $\mc{R}_1$ is forced to be irreducible too. 
Hence if $\mc{R}$ is semi standard it will follow that modulo covering maps, a spectral transfer map
from $\mc{H}_1$ to $\mc{H}$ is 
always a fibered product (over $\textup{Spec}(\mathbf{L})$) of spectral transfer maps between the simple factors, 
provided we normalize the simple factors of $\mc{H}_1$ and of $\mc{H}$
in a coherent way. 

If $\mc{H}_1$ has positive rank, consider the covering maps (\ref{eq:ominv}). 
If we precompose $\phi_T$ with the covering 
map $\textup{Hom}(\mathbb{Z}\times P((R_1)_m),\mathbb{C}^\times)\to T_1$, it 
is obvious from the definition that we obtain a spectral transfer map from 
$\mc{H}(\mc{R}^m_1,m)$ to $\mc{H}$, with the same image.
Hence we may and will assume from now on that the normalized affine Hecke algebra $\mc{H}_1$ 
is a tensor product (over $\mathbf{L}$) 
of normalized affine Hecke algebras $\mc{H}_1^j$, each of irreducible type (hence of positive rank) or of rank $0$  
(in which case $\mc{H}_1^j\simeq\mathbf{L}$), where 
the trace of $\mc{H}$ is the tensor product of the traces of the $\mc{H}_1^j$. 
Let 
$T_1=T_1^1\times_{\textup{Spec}(\mathbf{L})}T_1^2\times_{\textup{Spec}(\mathbf{L})}\dots \times_{\textup{Spec}(\mathbf{L})} T_1^{l_1}$ be 
the corresponding fibered product decomposition of $T_1$. 
From now on we will suppress $\textup{Spec}(\mathbf{L})$, but all products are tacitly assumed to be 
fibered products over $\textup{Spec}(\mathbf{L})$. 
The $\mu$-function $\mu_1$ of $\mc{H}_1$ 
is the product of the $\mu$-functions $\mu_1^j$ of $\mc{H}_1^i$. 
We embed the factors $T_1^j$ in $T_1$ in the canonical way, and identify the image with $T_1^j$. 
Observe that the number of factors $\mc{H}_1^j$ of a certain irreducible type is determined by $\mc{H}_1$, 
but the number of factors $\mc{H}_1$ of rank $0$ is arbitrary. In the discussion below we will use this freedom 
to adapt the tensor decomposition of $\mc{H}_1$ to the map $\phi_T$.

If the rank of $\mc{H}$ is zero then the rank of $\mc{H}_1$ has to be zero too, and there is nothing to prove.
Hence assume that the rank of $\mc{H}$ is positive. Let  
$L\subset T$ be the residual coset such that $L_n=\textup{Im}(\phi_T)$, and write $L=rT^P$ with 
$T_P\subset T$. Using the action of $W_0$ we may and will assume that there exists a subset 
$P\subset F_0$ such that $T^P$ is 
the maximal subtorus on which the roots in $P$ all vanish. Let $R_P\subset R_m$ denote the corresponding 
``standard parabolic subsystem" of $R_m$.  Assume that $R_0$ is a disjoint union of irreducible parabolic 
subsystems 
$R_0=R_0^1\cup R_0^2\cdots\cup R_0^l$. 
Then, with $K:=\textup{Hom}(X/Q(R_0),\mathbb{C}^\times)$, 
we have $\overline{T}:=T/K=\overline{T^1}\times\dots\times\overline{T^l}$, with each factor a torus over $\mathbf{L}$
of positive rank.  
Moreover $L=(L^1\times L^2\times\cdots\times L^l)/K'$, where $L^i$ is a connected component 
in the inverse image in $T$ of a residual coset 
$\overline{L^i}=\overline{r^i}\overline{T^{P^i}}\subset \overline{T^i}$ with respect 
to the irreducible root datum 
$\mathcal{R}_0^i:=(Q(R^i_0),R^i_0,P((R^i_0)^\vee), R^i_0, F_0^i)$, and $K'$ the kernel of 
the product map (on the product (over $\mb{L}$) of the underlying subtori $T^{P^i}$). 
It follows easily from the definitions that  $L_n=\prod_i(L^i_n)/K'$,
$\overline{L}_n=\prod_i (\overline{L^i})_n$ and that $(\overline{L^i})_n$ is a quotient of $L^i_n$.  
Let $\pi^i$ denote the projection from $L_n$ to $(\overline{L^i})_n$.
Composing $\phi_T$ with the quotient map $L_n\to\overline{L}_n$ we obtain a
spectral transfer map $\overline{\psi_T}$ from 
$\mc{H}_1=\mc{H}(\mc{R}_1^m,m_1)$ to $\mc{H}(\mc{R}_0,m)$ (cf. (\ref{eq:ominv})). 
Proposition \ref{prop:z} implies that for each $j\in\{1,\cdots,l_1\}$ 
there exists a \emph{unique} $i\in\{1,\cdots,l\}$ such that $\pi^i\circ\overline{\psi}_T|_{T_1^j}$ 
is nonconstant. Clearly $i\in\{1,\dots,l\}$ does not  occur iff $L^i$ has dimension zero, i.e. 
is a residual point $r^i:\textup{Spec}(\mb{L})\to T^i$ over $\mb{L}$. 
For each such $i$, we formally introduce an additional rank zero tensor 
factor $\mc{H}^{j(i)}_1\approx\mb{L}$ of $\mc{H}_1$, normalized such that the corresponding map 
$\phi(i):=\overline{r^i}:\mc{H}^{j(i)}_1\leadsto \overline{\mc{H}^i}:=\mc{H}(\mathcal{R}_0^i,m^i)$ 
is a spectral transfer map of rank $0$ (cf. Proposition \ref{prop:rk0}). 
With these rank zero tensor factors added, we can partition the set $\{1,...,l_1\}$ into disjoint nonempty 
subsets $\Pi(i)$ such that  for each $i$, with $T_1(i):=\prod_{j\in\Pi(i)}T_1^j\subset T_1$,  
the map $\phi(i):=\pi^i\circ\overline{\psi}_T|_{T_1(i)}$ is a spectral transfer map from $\otimes_{j\in\Pi(i)}\mc{H}_1^j$ 
(tensor products over $\mathbf{L}$)
to $\overline{\mc{H}^i}:=\mc{H}(\mathcal{R}_0^i,m^i)$. 
%If $\mc{H}_1$ has rank $0$ then $\textup{Im}(\phi_T)=r\in T$ is a residual point, 
%and by definition we have $\mu^{(\{r\})}(r)=\tau_1(1)$ up to a constant. 
%The image $\overline{r}\in \overline{T}$ of $r$ is also a residual point, and 
%we equivalently have $\mu^{(\{\overline{r}\})}(\overline{r})=\tau_1(1)$. In the latter equation, 
%the left hand side is a product over the irreducible factors $\mu^{(\{\overline{r^j}\})}(\overline{r^j})$
%(with $j=1,\dots,l$). Hence we can formally write 
%$\mc{H}_1=\mc{H}_1^1\otimes_\mathbf{L}\dots\otimes_\mathbf{L}\mc{H}_1^l$, 
%with traces $\tau_1^i$ normalized appropriately such that the canonical 
%maps $\overline{\phi}_T^i:\textup{Spec}(\mathbf{L})\to \overline{T}^i$ associated to 
%$\overline{r}^i$ are spectral transfer maps. This achieves our goal in the rank $0$ case, 
%$\overline{\phi}_T$ being the product of the spectral transfer maps $\overline{\phi}_{T^i}$.
\subsubsection{Further properties of spectral transfer maps}\label{par:neccond}
The proof of Proposition \ref{prop:z} can be refined to yield important additional 
information about the image of a spectral transfer map. Suppose that $\phi_T$ is a spectral transfer 
map from $\mc{H}_1$ to $\mc{H}$,
with $\mc{H}=\mc{H}(\mc{R},m)$ and $\mc{H}_1=\mc{H}(\mc{R}_1,m_1)$ both 
semi-simple affine Hecke algebras.

By paragraph \ref{par:products} we may assume that $\mc{H}$ is of irreducible type.
The dual affine Weyl group $W^\vee=W(R_m)\rtimes Y$ (cf. (\ref{eq:dualWeyl})) naturally acts on 
$\mf{t}:=\mathbf{R}\otimes_\mathbf{Z}Y\approx V$ 
via translations over 
$Y$ (a lattice which contains the lattice $Q(R_m^\vee)$) and the Weyl group $W(R_m)$ (or equivalently  $W(R_0^\vee)$
or $W_0$). 
In particular,  the affine Weyl group $W(R_m^{(1)})$ is a subgroup of $W^\vee$ of finite index.  Let us consider 
the $W(R_0^\vee)$-equivariant covering $\mf{t}\to T_u$ whose group of deck transformations is $Y$. 
We choose a fundamental 
dual alcove $C^\vee\subset \mf{t}$. It is a fundamental domain for the action of the normal subgroup
$W(R_m^{(1)})$ of $W^\vee$. Recall that $\mc{R}^m$ is a based root datum, with base $F_{m,0}$. 
Let $\mc{F}^m$ be the corresponding base for the affine root system $R_m^{(1)}$. 
It consists of the base $F_{m,0}$ together with the affine root $a_0^\vee:=1-\psi$, 
where $\psi$ denotes the highest root of $R_m$.

\emph{We will assume until 
Corollary \ref{cor:corresp} that $P\subset F_{m,0}$ is a proper subset }.
Then $r\in T_P\cap L$ is a residual point in $T_P$ for the semisimple quotient affine Hecke 
algebra $\mc{H}_P$ of the ``Levi subalgebra" $\mc{H}^P\subset \mc{H}$) of positive rank.  Let $r=sc$ be the 
polar decomposition of $r\in T_P$. Then $s\in T_{P,u}$ and $c(\mb{v})\in T_{P,\mb{v}}$ for all $\mb{v}>0$, 
and $c$ is itself a residual point for the subsystem $R_{P,s}=\{\alpha\in R_P\mid \alpha(s)=1\}$
such that $c(1)=e$. In particular it follows that $R_{P,s}\subset R_P$ is a maximal rank root subsystem.
It follows easily that a lift $\mf{l}(1)\subset \mf{t}$ of $L(1)^{temp}=sT^P_u\subset T_u$ 
in $\mf{t}$ with respect to the covering $\mf{t}\to T_u$ is conjugate under $W(R_m^{(1)})$
to a unique affine subspace $\mf{t}^J$ which is generated by a facet $C^{\vee,J}$ of $C^\vee$. 
Here $J\subset \mc{F}^m$ is a subset whose complement contains at least two elements, 
and is uniquely determined by $L$. Let $R_J\subset R_m^{(1)}$ 
be the subset of dual affine roots which vanish on $\mf{t}^J$.  
We may and will assume from now that we have moved $L$ by an appropriate element of $W(R_m)$ so 
that there exists a lifting 
$\mf{l}(1)\subset \mf{t}$ of $L(1)^{temp}$ of the form $\mf{l}(1)=\mf{t}^J$ for a (proper) subset $J\subset \mc{F}^m$.
Notice that $J\subset F_{m,0}$ if and only if $s=1$. In this case we call $0\in \mf{t}^J$ the \emph{origin} of $\mf{t}^J$.
Observe that in this situation $J=P$ is a standard parabolic subset, and therefore $R_J=R_P$ is a 
standard parabolic subsystem.

Let us now assume that $s\not =1$. Then $a_0^\vee\in J$. Let $R_P\subset R_m$
be the parabolic subset of roots which are \emph{constant} on $\mf{t}^J$, and let $R_{P,+}=R_P\cap R_{m,+}$.
Obviously $J_0:=J\backslash\{a_0^\vee\}\subset F_{m,0}$. Let $\beta\in R_{P,+}$ be the unique positive root 
such that $J_0\cup\{\beta\}$ forms a basis of $R_{P,+}$. Observe that the gradient projection $D$ maps $R_J$ 
isomorphically to a root subsystem of maximal rank obtained from omitting the 
simple $R_{P,+}$-basis element $\beta$ from the basis of the affine extension $R_P^{(1)}$ with 
affine basis $J\cup\{\beta\}$. 

Consider the affine isomorphism 
between $L^{temp}_n=scT^P_u/K_L^n$ and $L_n(1)^{temp}:=sT^P_u/K_L^n$ given by multiplication 
with $c^{-1}$. Using this we can choose an  
affine linear isomorphism $D^a\phi_T$ from $\mf{t}_1$ to $\mf{t}^J$ which lifts the finite 
affine morphism  $\phi_T:T_{1,u}=L_1^{temp}\to L^{temp}_n$.   Consider inside the Lie algebra 
$\mf{t}_1$ of $L^{temp}_1=T_{1,u}$ the affine hyperplanes with respect the action of dual affine 
Weyl group $W^{\vee,a}_1:=W(R_{m,1}^\vee)\rtimes Q(R_{m,1}^\vee)$ on $\mf{t}_1$.
Analogous to what was said in 
in the proof of Proposition \ref{prop:z}, the images under $D^a\phi_T$ of the affine reflection hyperplanes 
in $\mf{t}_1$ are the intersections of $\mf{t}^J$ 
with the affine hyperplanes of $\mc{R}^m$ on $\mf{t}$ which are not in $\mathbb{Z}J$.
In the semi-standard case this is a direct consequence of the way we choose $L$ and the 
fact that $\phi_T$ is a spectral transfer map with image $L$. In the remaining case
this statement follows easily from (T4).
As a consequence we may and will choose $D^a\phi_T$ such that the fundamental 
alcove  $C^\vee_1$ of $\mf{t}_1$ with respect to $\mc{F}_1^m$  
is mapped to $C^{\vee,J}$ and this fixes $D^a\phi_T$ uniquely.
Let $D\phi_T:\mf{t}_1\to\mf{t}^J_0$ be the 
gradient of $D^a\phi_T$, with $\mf{t}^J_0$ is the linear subspace in $\mf{t}$ parallel to 
$\mf{t}^J$. 
\begin{prop}\label{prop:excel} Assume $\mc{H}$ is irreducible and semi-standard.
Consider a spectral transfer map 
$\phi_T$ from $\mc{H}_1$ to $\mc{H}$ with image $L_n=rT^P_n$, with $L$ a residual coset of positive dimension 
in the position as described in the text above. Let $r=sc$ and let $\mf{t}^J$ be a lift of $L(1)^{temp}$. 
Then $J\subset \mc{F}^m$ is a subset whose complement has at least two elements.
\begin{enumerate}
\item[(1)] $J\subset \mc{F}^m$ is \emph{excellent} 
in the sense of \cite[paragraph 2.28]{Lu4}. 
\item[(2)]  For each affine root $a^\vee\in \mc{R}_{m,1}$, the affine linear automorphism  
$r^*_{a^\vee}:=D\phi_T\circ r_{a^\vee} \circ D\phi_T^{-1}$ of $\mf{t}^J$  
belongs to $N_{W^{\vee}}(W_J)$. 
\item[(3)]  The elements $r^*_{a^\vee}$, where $a^\vee$ runs over 
a basis of simple reflections of $W^{\vee,a}_1$ with respect to $C^\vee_1$, generate 
a subgroup $W^*$ of $N_{W^{\vee}}(W_J)$ isomorphic to $W^{\vee,a}_1$ which acts  
as an affine reflection group on $\mf{t}^J$. Hence $D^a\phi_T$ defines an 
isomorphism of affine reflection groups. In particular $\mc{R}_1$ is irreducible. 
We have $N_{W^\vee}(W_J)=W^*.W_J$
\item[(4)]  The element $D^a\phi_T(0)\in C^{\vee,J}\subset \mf{t}^J$  is a special vertex 
of the alcove $C^{\vee,J}$ for the action of $W^*$.  Since $C^{\vee,J}$ is a face of $C^\vee$, we 
have $D^a\phi_T(0)=\omega_i\in C^\vee$, a vertex of $C^\vee$ corresponding to 
an affine root $a^\vee_{k_0}\in\mc{F}^m\backslash J$. Observe that $s=1$ 
if and only if $\omega_{k_0}$ is a (special) vertex in $C^\vee$ which is mapped to $1\in T$
by the covering map $\mf{t}\to T$. 
\end{enumerate}
\end{prop}
\begin{proof}
Observe that the translation lattice $Q(R_{m,1}^\vee)$ is mapped injectively to a sublattice of 
the lattice $Y_J=\mf{t}^J_0\cap Y$ in $\mf{t}^J_0$.
As in the proof of Proposition \ref{prop:z}, for any mirror $M_{(P,\beta)}$ of $\Xi_{(P,\beta)}^{temp}$ 
a reflection $\mf{r}_{P,\beta}$ exists in $\mc{W}_\Xi$ 
which is an involutive automorphism on the space of induction data leaving $M_{(P,\beta)}$
pointwise fixed, and mapping $(P,\beta)$ to $(P,-\beta)$. It also follows from the proof of Proposition 
\ref{prop:z} that $\mf{r}_{P,\beta}$ can be represented by an element $\sigma=nw\in W_0$ for some 
$n\in W_0$ with $w(P)=P$ and $w\in W_P$, such that $\sigma(L)=L$. In particular $\sigma(\mf{t}^J_0)=
\mf{t}^J_0$. 
Suppose that $\xi_1\in M_{(P,\beta)}$, implying that $\xi_1$ is fixed by $\mf{r}_{P,\beta}$.
Then $\sigma(\mf{t}^J)=\mf{t}^J+y$ for $y=\sigma(e)-e\in Y$. 
Let $r^*=t_{-y}\circ \sigma$, then  $r^*\in W^\vee$ 
is the unique involutive affine isomorphism of $\mf{t}^J$ whose image in $W(R_m^\vee)$ 
is equal to $\sigma$, 
normalizing $\mf{t}^J$, and fixing $e\in C^{\vee,J}$. Clearly $r^*$ acts on 
$\mf{t}^J$ as the affine refection in the affine hyperplane of $\mf{t}^J$ through $e$ corresponding 
to $M_{(P,\beta)}$. Now let $M_{(P,\beta)}$ be defined by a dual affine simple root 
$b^\vee\in \mc{F}^m\backslash J$. 
It follows in a standard way that $r^*$ must be equal to $w_{J\cup\{b^\vee\}}w_J$ 
(where $w_J$ and $w_{J\cup\{b^\vee\}}$ denote the longest elements in the finite 
Weyl groups $W_J$ and $W_{J\cup\{b^\vee\}}$ respectively) and that $r^*(J)=J$.
It follows that $J$ is excellent as claimed. This proves (1), and along the way we also proved 
(2) and (3) (the remark on $N_{W^\vee}(W_J)$ is \cite[Proposition 2.26]{Lu4}, and 
the irreducibility of $W^{\vee,a}_1$ follows from \cite[2.28(a)]{Lu4}), and (4) is 
obvious.
\end{proof}
\begin{cor}\label{cor:corresp}
The map $D^a\phi_T$ induces a bijective correspondence between 
$\{I\subset \mc{F}_1^m\}$ and $\{I'\subset \mc{F}^m\mid J\subset I'\}$ such that 
$D^a\phi_T(C^{\vee,I}_1)=C^{\vee,I'}$. The map 
$D\phi_T$ induces a 
bijective correspondence between the set of parabolic subsystems $R_Q\subset R_{1,m}$ 
and the set of parabolic subsystems $R_{Q'}\subset R_m$ which contain $R_P$, 
such that $D\phi_T(\mf{t}_1^Q)=\mf{t}^{Q'}$.  
\end{cor}
\begin{proof}
Both claims are simple consequences of  Proposition \ref{prop:excel}.
(first note that the result is trivial if $P=F_0$).
For the second claim, observe that by definition of $P$, we have $\mf{t}^P=\mf{t}^J_0$.
Since $\omega_i$ is special with respect to the action of $W^*$, every intersection 
affine root hyperplanes $V'$ of a affine roots of $R_m^{(1)}$ with $\mf{t}^J$ is parallel 
to a unique affine subspace of $\mf{t}^J$ through $\omega_i$ which is an intersection of 
affine hyperplanes. This corresponds via $D\phi_T^a$ to a linear subspace $V$ of $\mf{t}_1$ 
which is an intersection of linear root hyperplanes in $\mf{t}_1$. Thus $D^a\phi_T$ sets up a 
bijection between the collection of linear subspaces of $\mf{t}_1$ which are an intersection 
of roots of $R_{1,m}$ and the collection of classes of affine subspaces $V'$ of 
$\mf{t}^J$ 
which are intersections of affine root hyperplanes for $R_m^{(1)}$ and which are parallel 
to each other. To each such 
class of affine subspaces parallel to $V'$ we attach the set of roots $R(V')\subset R_m$ which are  
constant on $V'$. Then $R(V')=R_{Q'}$, a parabolic subsystem containing $P$, and it is 
clear that $D\phi_T$ yields the desired bijection.
\end{proof}
\begin{cor} \label{cor:origin}
In Corollary \ref{cor:corresp} we have $\phi_T(T_{1,Q})=K_L^n(T_{Q'}\cap L)^0/K_L^n$.
\end{cor}
\begin{proof}
It is enough to prove that $D\phi_T(\mf{t}_{1,Q})=\mf{t}_{Q'}\cap \mf{t}^P$. The 
dimensions of these subspaces are equal since $\mf{t}^Q_1\subset \mf{t}_1$ and 
$\mf{t}^{Q'}\subset \mf{t}^P$ have the same codimension, by definition of the correspondence. 
Hence it is enough to prove that $D\phi_T(\mf{t}_{1,Q})\subset \mf{t}_{Q'}\cap \mf{t}^P$. 
This follows by the remark that $D\phi_T$ must map $\alpha^\vee\in R_Q^\vee$ to a 
multiple of the projection of a coroot in $R_{Q'}$ onto $\mf{t}^P$ along $\mf{t}_P$.
Since $\mf{t}_P\subset \mf{t}_{Q'}$ this implies that the image is in $\mf{t}_{Q'}\cap \mf{t}^P$.
\end{proof}
\subsection{Composition of spectral transfer maps; spectral transfer morphisms}\label{subsub:compos}
Let $\phi_T$ be a spectral transfer map from $(\mc{H}_1,\tau^{d_1})$ 
to $(\mc{H}_2,\tau^{d_2})$ and let $\psi_T$ be a spectral transfer map from 
$(\mc{H}_2,\tau^{d_2})$ to $(\mc{H}_3,\tau^{d_3})$. We would like to define the 
the composition $\rho_T:=\psi_T\circ\phi_T$, but in order to do so we need to come to grips 
with the fact that the image of $\phi_T$ is the \emph{quotient} $L_n$ of a residual coset 
$L\subset T_2$. We are saved here by Proposition \ref{prop:z}. Indeed, we have 
$L_n:=L/K_L^n$ and $\dot{K}_L^n\subset N_{W_{2,0}}(L)$ by definition. By Proposition 
\ref{prop:z}, for all $k\in \dot{K}_L^n\subset N_{W_{2,0}}(L)$ there exists a 
$w\in N_{W_{3,0}}(M)$ 
such that $\psi_T\circ k=w\circ \psi_T$, where $M\subset T_3$ is a residual coset such that 
$\psi_T(T_2)=M_n$. By Proposition \ref{prop:restrans} there exists a residual coset $N\subset M$ 
such that $\psi_T(L)=\overline{N}:=N/{N_{K_M^n}(N)}$. Observe that $\dot{K}_M^n(N)\subset\dot{K}_N^n$ 
since the latter 
is the subgroup of elements of $W_{3,0}$ which stabilize $N$ and restrict to a translation on $N$
(i.e. a multiplication by some $k' \in T^N$).
Hence $N_n:=N/K_N^n$ is a quotient of $\overline{N}$. 
By the above and property (T2) of $\psi_T$ (which implies that $D\psi_T|_{\mf{t}^L}:\mf{t}^L\to\mf{t}^N$ 
is a linear isomorphism), we see that $w\in N_{W_{3,0}}(N)$ and that 
$w$ restricts to a translation on $N$. 
Hence $w\in K_N^n$.
In other words, $\psi_T|_L:L\to \overline{N}$ maps $K_L^n$ orbits on 
$L$ to $K_N^n/K_M^n(N)$ orbits on $\overline{N}$, and thus defines a map 
$\psi_{T,n}: L_n\to N_n$. Now we can finally define the composition of  $\phi_T$ and 
$\psi_T$ to be the map $\rho_T:=\psi_{T,n}\circ \phi_T$. 

It is clear that $\rho_T$ satisfies (T1) and that $\rho_T$ is an affine homomorphism 
from $T_1$ onto $N_n$ where $N\subset T_3$ is a residual coset. 
In order to verify (T2) we need to check in addition that $\rho_T(e)\in (T_N\cap N)/K_N^n$.  
Let $T_L\subset T_2$ be the subtorus whose Lie algebra 
is spanned by the coroots of the roots which restrict to constants on $L$ (and similarly 
we define $T_M$ and $T_N$).  Since $N\subset M$ we see that  
$T_M\subset T_N$. 
Observe that $R_{2,m,L}$ and $R_{3,m,N}$ correspond to each other under $\psi_T$
in the sense of Corollary \ref{cor:corresp}, and therefore we have 
$\psi_T(T_L)= K_M^n(T_N\cap M)^0/K_M^n\subset \overline{T_N}:=K_M^nT_N/K_M^n$ 
by Corollary \ref{cor:origin}. In particular, if $r_L\in K_L^n\phi_T(e)\subset L\cap T_L$ (using (T2) for $\phi_T$) then 
$\psi_T(r_L)\in\overline{T_N}$. Clearly $\psi_T(r_L)\in\overline{N}$ 
by definition of $N$, so that $\psi_T(r_L)\in \overline{N}\cap \overline{T_N}$. It follows 
that $\rho_T(e)=\overline{\psi_T}(r_LK_L^n/K_L^n)\in (N\cap T_N)/K_N^n$.
Hence $\rho_T$ also satisfies (T2).

Write $\mu_i=\mu_{\mc{R}_i,m_i,d_i}$ for the $\mu$-functions of the $(\mc{H}_i,\tau^{d_i})$.
Then $\phi^*_T(\mu_2^{(L_n)})=a\mu_1$
and $\psi^*_T(\mu_3^{(M_n)})=b\mu_2$. We claim that this last
identity implies that
$\overline{\psi_T}^*(\mu_3^{(N_n)})=b^\prime\mu_2^{(L_n)}$ for some nonzero complex
number $b^\prime$. 
For this it suffices to prove that $\psi_T^*(\mu_3^{(\overline{N})})=b^\prime\mu_2^{(L)}$, 
where $\mu_3^{(\overline{N})}$ is the restriction to $\overline{N}$ of the regularization of 
$\mu_3^{(M_n)}$ along $\overline{N}$. To prove this statement, note that the irreducible factors of
$\mu_3^{(M_n)}$ are of the form $(1-v^i\zeta\alpha/n)$
with $\alpha\in R_{3,m}$ not constant on $\mf{t}^M_3$, $i$ and $n$ certain integers 
(depending on $\alpha$), 
and $\zeta$ an $n$-th root of $1$.
Therefore we have (using (T1) for $\psi_T$)
$\psi^*_T(1-v^i\zeta\alpha/n)=(1-v^{i'}\zeta'\beta)$ where $\beta$ is a character of $T_2$, 
$\zeta'$ a complex root of $1$, and $i'$ is an integer.
By construction $(1-v^{i'}\zeta'\beta)$ is identically zero on $L$ iff
$(1-v^i\zeta\alpha/n)$ is identically zero on $\overline{N}$, and in this case all 
irreducible factors
of $(1-v^{i'}\zeta'\beta)$ restrict to constants on $L$.
This easily implies the claim.
We conclude that $\rho_T^*(\mu_3^{(N_n)})=ab^\prime\mu_1$,
proving (T3) for $\rho_T$.
Thus $\rho_T=\overline{\psi}_T\circ\phi_T$ is a spectral transfer map. 
\emph{From now on we will denote the composition $\rho_T$ simply by 
$\psi_T\circ \phi_T$ instead of $\overline{\psi}_T\circ\phi_T$}.

The most trivial examples of spectral transfer maps are Weyl group elements.
If $(\mc{H},\tau^d)$ is a normalized affine Hecke algebra with $\mc{H}=\mc{H}(\mc{R},m)$
then for any $w\in W_0$ we have an invertible spectral transfer map $w:T\to T$. 
Given a spectral transfer map $\phi_T$ from $(\mc{H}_1,\tau^{d_1})$ to $(\mc{H}_2,\tau^{d_2})$ 
with image $L_n$ where $L\subset rT^P\subset T$ (with $R_P\subset R_0$ a parabolic subsystem)  
it follows
from Proposition \ref{prop:z} that $W_{2,0}\circ\phi_T\circ W_{1,0}=W_{2,0}\circ \phi_T$.
Therefore we can define an equivalence relation as follows.
We call two spectral transfer maps $\phi_{1,T},\phi_{2,T}: T_1\to T_2$ from 
$(\mc{H}_1,\tau^{d_1})$ to $(\mc{H}_2,\tau^{d_2})$
\emph{equivalent} if $\phi_{2,T}\in W_{2,0}\circ\phi_{1,T}$. 
As a special case, we may post-compose 
$\phi_T$ with an affine map $M_k: L  \to L$ defined by $M_k(t)=kt$ where $k\in K_L^n$
within the same equivalence class. 

By the above remarks it is not difficult to check that 
this notion of equivalence of spectral transfer maps is compatible with the operation
of composition of spectral transfer maps (this follows from Proposition \ref{prop:z}).  
This shows that the following definitions makes sense:
\begin{defn}
The spectral transfer category $\mf{C}$ is the category formed by the normalized
affine Hecke algebras as objects, and equivalence classes of spectral transfer maps as morphisms.
A spectral transfer morphism (abbreviated to STM) from $(\mc{H}_1,\tau^{d_1})$ to $(\mc{H}_2,\tau^{d_2})$ 
is denoted by 
$\phi:(\mc{H}_1,\tau^{d_1})\leadsto(\mc{H}_2,\tau^{d_2})$ (or simply 
$\phi:\mc{H}_1\leadsto\mc{H}_2$
if there is no danger of confusion). 
If $\phi_T:T_1\to T_2$ is an affine morphism in the equivalence  
class of the STM $\phi$ we say that $\phi$ is represented by $\phi_T$. The composition 
$\psi\circ\phi$ of STMs is defined by composing representing 
STMs: $(\psi\circ\phi)_T=\psi_T\circ\phi_T$.
\end{defn}
It is obvious that two spectral 
transfer maps $\phi_{1,T},\phi_{2,T}: T_1\to T_2$ are equivalent iff $\phi_{1,Z}=\phi_{2,Z}$.
Hence an STM $\phi:(\mc{H}_1,\tau^{d_1})\leadsto(\mc{H}_2,\tau^{d_2})$ 
defines a finite morphism $\phi_Z$ from the spectrum of the center of $\mc{H}_1$ to the 
spectrum of the center of $\mc{H}_2$. 
Similarly we obtain a smooth finite map $\phi_Z^{temp}$. 
\begin{defn}
Let $\phi:(\mc{H}_1,\tau^{d_1})\leadsto(\mc{H}_2,\tau^{d_2})$ be an STM
and let $\Phi=W_{2,0}\circ\phi_T$ be the associated class of morphisms.
We attach various invariants to $\Phi$. The morphism $\phi_Z$ (see Corollary
\ref{cor:z}) clearly only depends on $\Phi$, which we express by writing $\Phi_Z$
and $\Phi^{temp}_Z$. By the \emph{image} $\textup{Im}(\Phi)$ of $\Phi$ we mean the
$W_{2,0}$-orbit of residual cosets $W_{2,0}\phi_T(T_1)\subset W_{2,0}\backslash T_2$.
Similarly, the \emph{tempered image} of $\Phi$ is the image $\phi_Z^{temp}(S_1)\subset S_2$.
We define two nonnegative integers associated with $\Phi$,
the \emph{rank} $\textup{rk}(\Phi)=\textup{dim}(T_1)-1$
and the \emph{co-rank} $\textup{cork}(\Phi)=\textup{dim}(T_2)-\textup{dim}(T_1)$.
\end{defn}
\section{The tempered correspondence of an STM}
We will show that a spectral transfer 
morphism induces a correspondence on the tempered spectra, respecting the 
connected components and Plancherel measures.
\begin{theorem}\label{thm:corr}
Let $\phi:(\mathcal{H}_1,\tau^{d_1})\leadsto (\mathcal{H}_2,\tau^{d_2})$ be an STM.
Consider the correspondence between $\mf{S}_1$ and $\mf{S}_2$ given by the fibered product
$\mf{S}_{12}$ of $\mf{S}_1$ and $\mf{S}_2$
with respect to the diagram:
\begin{equation}
\begin{CD}
\mf{S}_{12}@>{p_1}>>\mf{S}_1\\
@V{p_2}VV                  @VV{\phi_Z^{temp}\circ p_{Z,1}^{temp}}V\\
\mf{S}_2@>>{p_{Z,2}^{temp}}> S_2
\end{CD}
\end{equation}
If $C\subset \mf{S}_{12}$
is a component then $C_i:=p_i(C)\subset\mf{S}_i$ is a component
($i=1,2$). Moreover there exists a positive measure $\nu$ on $C$ and
$r_i\in\mathbb{Q}_+$ such that
\begin{equation}
(p_i)_*(\nu)=r_i\nu_{Pl,i}|_{C_i}
\end{equation}
for $i=1,2$.
\end{theorem}
\begin{proof}
Let $L\subset T_2$ be a residual coset which belongs to the image of $\Phi$.
Let $B_2\subset S_2$ be a component. Then $B_2\subset\phi_Z^{temp}(S_1)$ iff 
$B_2$ is associated to a residual coset $L_2\subset T_2$ (in the sense of Proposition \ref{prop:S}(i))
such that $L_2\subset W_0L$.
By Proposition \ref{prop:restrans} and Corollary \ref{cor:z} we see that $\phi_Z^{-1}(L_2)$
consists of a union of finitely many $W_{1,0}$-orbits of residual cosets in $T_1$.
Using Proposition \ref{prop:S} we see that $(\phi_Z^{temp})^{-1}(B_2)$ consists of a union
of finitely many components $B_1\subset S_1$, each of which has the property that
$\phi_Z^{temp}(B_1)=B_2$. It follows in particular that for any component $B_1\subset S_1$
its image $\phi_Z^{temp}(B_1)\subset S_2$ is a component of $S_2$. By Proposition
\ref{prop:S}(i) we see that the correspondence associated to $\mf{S}_{12}$ indeed
yields a finite correspondence between the components of $\mf{S}_1$ and $\mf{S}_2$.
This proves the first assertion of the Theorem.

Let $C_1\subset\mf{S}_1$ and $C_2\subset\mf{S}_2$ be corresponding components.
Put $B_i=p_{Z,i}^{temp}(C_i)$ for the corresponding components of $S_i$.
Let $L_i\subset T_i$ be corresponding residual cosets in the sense of Proposition
\ref{prop:S}. We may assume without loss of generality that $L_2\subset L$ and
that $L_1$ is a component of $\phi_T^{-1}(K^n_LL_2/K^n_L)$.
By Proposition \ref{prop:S}(iii) and \cite[Lemma 4.35]{Opd1} we see that
$B_i^\prime$ corresponds to $N_{W_{i,0}}(L_i)$-orbits of $R_{L_i}$-generic
points of $L_i^{temp}$. By Proposition \ref{prop:z} we see that
$B_1^\prime\subset B_2^{\prime\prime}:=(\phi_Z^{temp}|_{B_1})^{-1}(B_2^\prime)$.
Using \cite[Lemma 4.35]{Opd1} and (T1) we see that
$\phi_Z^{temp}|_{B_1}:B_1^{\prime\prime}\to B_2^\prime$ is a finite covering map.
By (T1) and Theorem \ref{thm:plancherel} the subset
$C_1^{\prime\prime}=(p_{Z,1}^{temp}|_{C_1})^{-1}(B_1^{\prime\prime})\subset C_1$
is dense and its complement has measure zero with respect to $\nu_{Pl,1}|_{C_1}$.
Consider the commuting diagram
\begin{equation}\label{eq:commreg}
\begin{CD}
C^{\prime\prime}@>{p_1|_{C^{\prime\prime}}}>>C_1^{\prime\prime}\\
@V{p_2|_{C^{\prime\prime}}}VV                  @VV{\phi_Z^{temp}\circ p_{Z,1}^{temp}|_{C_1^{\prime\prime}}}V\\
C_2^\prime@>>{p_{Z,2}^{temp}|_{C_2^\prime}}> S_2^\prime
\end{CD}
\end{equation}
By the above this is a diagram of finite covering maps of smooth manifolds.
We use the notation of (the proof of) Lemma \ref{lem:plan} and Theorem
\ref{thm:plancherel}.
Let $\nu$ be the measure on $C$ obtained by the extension by zero
of the pull back of the smooth volume form $\nu_{S_2}^\prime|_{B_2^\prime}$.
Since $p_2|_{C^{\prime\prime}}$ is a finite covering map, the push
forward of the measure $\nu|_{C^{\prime\prime}}$ to $C_2$
is a nonzero integer constant multiple of the pull back of $\nu_{S_2}^\prime$
to $C_2^\prime$. By Theorem \ref{thm:plancherel} the Plancherel measure
$\nu_{Pl,2}$ is a nonzero rational multiple of the extension by zero
of this measure. Hence $(p_2)_*(\nu)=r_2\nu_{Pl,2}|_{C_2}$ for a
certain $r_2\in\mathbb{Q}^\times$, as desired.

Finally we need to show that
\begin{equation}\label{eq:nonrat}
(p_1)_*(\nu)=r_1\nu_{Pl,1}|_{C_1}
\end{equation}
for some
nonzero rational constant $r_1\in\mathbb{Q}^\times$.
By Theorem \ref{thm:plancherel} it suffices to showing
\begin{equation}
(\phi_Z^{temp}|_{B_1^{\prime\prime}})^*(\nu_{S_2}^\prime)=r^\prime_1\nu_{S_1}^{\prime\prime}
\end{equation}
where $\nu^{\prime\prime}_{S_1}$ is the restriction of $\nu_{S_1}^\prime$
to $S_1^{\prime\prime}$, and $r^\prime_1\in\mathbb{Q}^\times$.
By the text above Theorem \ref{thm:plancherel} this boils down to proving
that
\begin{equation}\label{eq:help}
(\phi_T|_{L_1})^*(\mu_2^{(K_L^nL_2/K^n_L)})=r_1^{\prime\prime}\mu_1^{(L_1)}
\end{equation}
for some constant $r_1^{\prime\prime}\in\mathbb{Q}^\times$.
This is true by definition, and by Proposition \ref{prop:ratcon}.
\end{proof}
\begin{rem}
All results discussed so far are true, mutatis mutandis, in the slightly more general setting 
where we lift the condition on $\phi_T$ that it has to satisfy $\phi_T(e)\in (T_L\cap L)/K_L^n$ (if $L/K^n_L$ 
denotes the image of $\phi_T$), and keep all other requirements. We call such a map 
an \emph{essentially strict} STM if (in the notation of paragraph \ref{par:neccond}) 
$\mf{t}_P\cap \mf{t}^J$ is a special vertex of $C^{\vee,J}$ for $W^*$. We denote the category of 
essentially strict STMs by $\mf{C}_{es}$.
\end{rem}
\section{Examples of spectral transfer morphisms} In this section we discuss 
opposite extreme examples of STMs: Spectral covering morphisms (which are STMs 
of maximal rank), and cuspidal STMs (which are STMs of rank $0$). Understanding these
cases is important for the applications of STMs as discussed in \cite{Opd4}.
\subsection{Spectral morphisms of maximal rank}
\subsubsection{Semi-standard spectral automorphisms}\label{par:specaut}
A semi-standard spectral endomorphism $\Phi:(\mc{H},\tau^d)\leadsto(\mc{H},\tau^d)$
is represented by a homomorphism 
$\phi_T:T\to T$ of algebraic tori with finite kernel, such that $\phi^*(\mu)=a\mu$ 
for some nonzero constant $a$.
Thus $\phi_T$ lifts via the exponential mapping of $T$ to a linear isomorphism
$D^a\phi_T:V\to V$ with the property that $D^a\phi_T$ induces a bijection on the $\mu$-mirrors
such that the functions $m_+$ and $m_-$ associated with the mirrors are preserved
(see Proposition \ref{prop:musym} for the notion of $\mu$-mirror), and in addition
normalizes $Y$.
In the language of the proof of Proposition \ref{prop:musym}(ii),(iv) this means that
$D^a\phi_T\in\textup{Aut}_V(\mu,Y)_0:=\{\xi\in \textup{Aut}_V(\mu,Y) \mid \xi(0)=0\}$.
If we do not require $\phi_T$ to be strict then we similarly obtain
$D^a\phi_T\in\textup{Aut}_V(\mu,Y)$.
It follows that:
\begin{prop}\label{prop:aut}
Assume that $\mc{H}=\mc{H}(\mc{R},m)$ is semi-simple with $(\mc{R},m)$ semi-standard.
The spectral endomorphisms of $(\mc{H},\tau^d)$ are invertible, and there is
a canonical isomorphism of $\textup{Aut}_\mf{C}(\mc{H},\tau^d)$
and the group $\textup{Out}_T(\mu)_e:=\{\psi\in \textup{Out}_T(\mu)\mid
\psi(e)=e\}$. If $(\mc{H},\tau^d)$ is standard then 
this equals  $\Omega_0^Y(R_0^\vee)$ (see Proposition \ref{prop:musym}).
Similarly non-strict spectral transfer endomorphisms are in fact essentially strict 
spectral transfer automorphisms. 
These form a group $\textup{Aut}_{\mf{C}_{es}}(\mc{H},\tau^d)$ which 
is canonically isomorphic to  $\textup{Out}_T(\mu)$.  
In particular the constant $a$ of (T3) equals $a=1$ for all (not necessarily strict) 
spectral endomorphism in this situation.
\end{prop}
Similarly we have:
\begin{prop}
The group $\textup{Aut}_\mf{C}(\mc{H},\tau^d)$ of spectral automorphisms of 
$\mc{H}=\mc{H}(\mc{R},m)$ (with $(\mc{R},m)$ semi-standard)
is canonically anti-isomorphic to the group $\textup{Aut}_{s}(\mc{H})$ of 
strict algebra automorphisms, via the right action of this last group on $T$.
Similarly, the group $\textup{Aut}_{\mf{C}_{es}}(\mc{H},\tau^d)$ of essentially 
strict STMs is anti-isomorphic to $\textup{Aut}_{es}(\mc{H})$.
\end{prop}
\begin{proof}
Apply Proposition \ref{prop:musym}(v).
\end{proof}
These results show that the spectral correspondences of semi-standard spectral
automorphism can be refined to a Plancherel measure preserving automorphism in view
of the following result which follows trivially from the definitions.
\begin{prop}\label{prop:bij}
Admissible algebra automorphisms $\sigma\in\textup{Aut}_{adm}(\mc{H})$
preserve both the $*$-operator on $\mc{H}$ and the trace $\tau^d$ of $\mc{H}$.
In particular such an automorphism induces a Plancherel measure preserving homeomorphism
of $\mf{S}$.
\end{prop}
\begin{cor}
If $\phi\in\textup{Aut}_{\mf{C}_{es}}(\mc{H})$ corresponds to
$\sigma\in\textup{Aut}_{es}(\mc{H})$ (i.e. the action of $\sigma$ on $T$ is
equal to $\phi_T$) then
$\sigma:\mc{H}\leadsto\mc{H}$ induces a Plancherel measure preserving automorphism
on $\mf{S}$ refining the spectral correspondence of $\phi^\sigma$.
\end{cor}
\subsubsection{Spectral isomorphisms}\label{par:speciso}
Let $\mc{R}$ be a semisimple root datum.
The $\mu$ function has certain important symmetries with respect
to the cocharacter $m$. Consider the group $\textup{Iso}=\textup{Iso}(\mc{R})$ of
automorphism of $\mc{Q}_c=\mc{Q}_c(\mc{R})$ generated by the
following involutions $\eta^\mb{c}$ associated with the
$W$-conjugation classes $\mb{c}\subset S$ of affine simple reflections.
If $\mb{c}\cap S_0\not=\emptyset$ we define
(using notations as in Theorem \ref{thm:ber})
$\eta^\mb{c}(v(s))=v(s^\prime)^{-1}$
if $s\in\mb{c}$ or $s^\prime\in\mb{c}$, and $\eta^\mb{c}(v(s))=v(s)$ else.
For conjugacy classes $\mb{c}$ such that $S_0\cap\mb{c}=\emptyset$
we define $\eta^\mb{c}$ by $\eta^\mb{c}(v(s))=v(s)^{-1}$
if $s\in\mb{c}$ and $\eta^\mb{c}(v(s))=v(s)$ else.
We denote the action of $\eta\in\textup{Iso}$ on the
set of cocharacters of $\mc{Q}_c$ by $m\to\eta(m)$.
\begin{prop}
For all $W$-conjugacy classes $\mb{c}$ of affine simple reflections 
we can define an essentially strict spectral transfer isomorphism
$\phi^{\mb{c}}:(\mc{H}(\mc{R},m),\tau^d)\leadsto(\mc{H}(\mc{R},\eta^\mb{c}(m)),\tau^d)$
as follows.
If $\mb{c}\cap S_0\not=\emptyset$ we put $\phi_T^{\mb{c}}=\textup{id}_T$.
If $\mb{c}\cap S_0=\emptyset$ we put $\phi_T^{\mb{c}}(t)=s_\mb{c}t$, where
$s_\mb{c}\in T$ is the $W_0$-invariant element defined by
$\alpha(s_\mb{c})=-1$ if $\alpha\in R_0$ is such that $s_\alpha^\prime\in\mb{c}$,
and $\alpha(s_\mb{c})=1$ else.
\end{prop}
\begin{proof}
We need to verify (T3) (the other conditions being trivially satisfied)
and this is an easy computation.
\end{proof}
These spectral isomorphisms too are given by admissible algebra isomorphisms.
If $\mb{c}\cap S_0\not=\emptyset$ we define
$\sigma^\mb{c}\in\textup{Hom}_{adm}(\mc{H}(\mc{R},m),\mc{H}(\mc{R},\eta^\mb{c}(m)))$
by $\sigma^\mb{c}(N_s)=-N_{s^\prime}$ for $s\in\mb{c}$ or $s^\prime\in\mb{c}$
and $\sigma^\mb{c}(N_s)=N_s$ else.
For conjugacy classes $\mb{c}$ such that $\mb{c}\cap S_0=\emptyset$
we define
$\sigma^\mb{c}\in\textup{Hom}_{adm}(\mc{H}(\mc{R},m),\mc{H}(\mc{R},\eta^\mb{c}(m)))$
by $\sigma^\mb{c}(N_s)=-N_s$ for $s\in\mb{c}$ and $\sigma^\mb{c}(N_s)=N_s$ else.
We remark that if $\mb{c}$ is such that $\mb{c}\cap S_0\not=\emptyset$ and
$\mb{c}\not=\mb{c}^\prime$ then $\sigma^\mb{c}$ is not essentially strict and
does not send $\mc{A}$ to $\mc{A}$.
However, the induced morphism $\sigma^\mb{c}_\mc{Z}=\sigma^\mb{c}|_\mc{Z}$
on the center $\mc{Z}$ of $\mc{H}(\mc{R},m)$ is equal to $\phi^\mb{c}_Z$ in all cases.
Hence we obtain:
\begin{prop} The algebra isomorphism $\sigma^\mb{c}$ defines a Plancherel measure
preserving homeomorphim
\begin{equation}
\sigma^{\mb{c},temp}_Z:\mf{S}(\mc{R},\eta^\mb{c}(m))\to\mf{S}(\mc{R},m)
\end{equation}
refining the spectral correspondence of the spectral
transfer morphism $\phi^\mb{c}$.
\end{prop}
\begin{rem}\label{rem:D8}
In terms of $\Sigma_s$ the action $\textup{Iso}(\mc{R},m)$ is given
as follows. If $\mb{c}\cap S_0\not=\emptyset$ then $\eta^\mb{c}$ replaces
the parameter $m_+(\mb{c})$ by $-m_+(\mb{c})$ and leaves the rest of the spectral
diagram unchanged. If $\mb{c}\cap S_0=\emptyset$ then $\eta^\mb{c}$ interchanges
the role of $m_+(\mb{c})$ and $m_-(\mb{c})$, and leaves the rest of the spectral
diagram unchanged. In particular, if $\mb{c}\not=\mb{c}^\prime$ then the two
involutions $\eta^\mb{c}$ and $\eta^{\mb{c^\prime}}$ acting on the corresponding
component of type $\textup{C}_n^{(1)}$ of the spectral diagram generate a dihedral group
of order $8$.
\end{rem}
\subsubsection{Spectral covering morphisms}\label{subsub:cov}
We call an STM $\phi$ a spectral covering morphism if
$\phi_T$ is surjective, or otherwise said, if $\textup{cork}(\phi)=0$. 
\begin{prop}
Assume that $(\mc{R}_i,m_i)$ ($i=1,2$) both are semi-standard and let
$\mc{H}_i=\mc{H}(\mc{R}_i,m_i)$. Let $\phi:(\mc{H}_1,\tau^{d_1})\leadsto(\mc{H}_2,\tau^{d_2})$
be a (not necessarily strict) spectral covering morphism. 
After replacing $(\mc{H}_1,\tau^{d_1})$ by an isomorphic
object in $\mf{C}_{es}$ if necessary, and up to the action of
$\textup{Out}_T(\mu_1)$ (see Proposition \ref{prop:aut}), 
$\phi$ is represented by 
a finite morphism $\phi_T:T_1\to T_2$ of tori associated to
a sublattice $X_2\subset X_1$ of maximal rank such that
$\mc{R}_1=(X_1,R_{1,0},Y_1,R_{1.0}^\vee)$
and $\mc{R}_2=(X_2,R_{2,0},Y_2,R_{2,0}^\vee)$ where $R_{1,m_1}=R_{2,m_2}$ and
$m_{1,R_1}^\vee=m_{2,R_2}^\vee$. In this situation the labelled Dynkin diagram
of $\Sigma_s(\mc{R}_1,m_1)$ can be identified with the underlying labelled Dynkin diagram
of $\Sigma_s(\mc{R}_2,m_2)$ while $\Omega_{Y_1}^\vee\subset\Omega_{Y_2}^\vee$.
In addition, $\phi$ is  essentially strict.
\end{prop}
\begin{proof}
Let $\mc{H}_i=\mc{H}(\mc{R}_i,m_i)$ and let $\mc{R}_i=(X_i,R_{i,0},Y_i,R_{i,0}^\vee)$.
Clearly $X_2\subset X_1$ via the morphism of algebraic tori $T_1\to T_2$ underlying
the affine morphism $\phi_T$. By (T3) we have an isomorphism
of reflection groups $W(R_{1,0})\approx W(R_{2,0})$ via the map $(D\phi_T)_p$
where $p\in T_1$ is such that $\phi_T(p)=e$
(using the argument given in the proof of Proposition \ref{prop:z}).
Hence $p\in T_1$ is a $W(R_{1,0})$-invariant point, and there exists an 
essentially strict spectral transfer automorphism of $\mc{H}_1$ mapping $e\in T_1$ to $p$.
Therefore we may assume without loss of generality that $\phi_T$ is
a morphism of algebraic tori. From (T3) it also follows that
Proposition \ref{prop:musym} applies both to $\mc{H}_1$ and
$\mc{H}_2$.
We conclude that $d(\phi_T)_e$ induces a monomorphism
$W(R_{m_1}^{(1)})\rtimes\Omega_{Y_1}^\vee\to W(R_{m_2}^{(1)})\rtimes\Omega_{Y_2}^\vee$.
We have $R_{2,0}\subset X_2\subset X_1$ and there exists a
$W(R_{1,0})$-invariant function $n^\phi: R_{1,0}\to\{\frac{1}{2},1,2\}$ such that
$R_{2,0}=\{n^\phi(\alpha)\alpha\mid\alpha\in R_{1,0}\}$. If $n^\phi(\alpha)=\frac{1}{2}$
then $\alpha$ belongs to an irreducible component of $R_{1,0}$ of type
$\textup{C}_n$ (with $n\geq 1$) whose weight lattice is contained in $X_2$. Hence
by Proposition \ref{prop:Rmcan} 
there is a corresponding direct summand $\mc{C}^1$ of $\mc{R}_1$ with
root system $R_0^1$ of type $\textup{C}_n $ and $X^1=P(R_0^1)$, and a direct summand 
$\mc{C}^2$ of $\mc{R}_2$ with root system $R_0^2$ of type $\textup{B}_n$
and lattice $X^2=X^1$. By (T3) we see that $m_{\mc{C}^2,-}=m_{\mc{C}^2,+}$,
and the tensor factor $\mc{H}(\mc{C}^1,m_{\mc{C}^1})$ of $\mc{H}_1$ is isomorphic
(as algebras)
to the tensor factor $\mc{H}(\mc{C}^2,m_{\mc{C}^2})$ of $\mc{H}_2$ in a
sense compatible with the STM.
Observe in particular that $R_m^1=R_m^2$ in this situation.
If on the other hand
$n^\phi(\alpha)=2$ for some $\alpha\in R_{1,0}$ then $\alpha$
belongs to an irreducible component $R^1_0$ of $R_{1,0}$ of type $\textup{B}_n$.
Using (T3) one checks that the only possibility is in this
case that $X^2=X^1=X_1\cap\mathbb{R}R^1_0=Q(R^1_0)$, and again we
have corresponding isomorphic tensor factors in $\mc{H}_1$ and
$\mc{H}_2$ with $R_m^1=R_m^2$.

Hence we may assume that $R_{1,0}=R_{2,0}$. By (T3) it is clear that
we can replace $(\mc{H}_1,\tau^{d_1})$ by an isomorphic object of $\mf{C}$
such that $m_{1,+}= m_{2,+}$ and $m_{1,-}= m_{2,-}$, and the remaining assertions
easily follow.
\end{proof}
\begin{cor} If $\phi:(\mc{H}_1,\tau^{d_1})\leadsto(\mc{H}_2,\tau^{d_2})$ is a not necessarily 
strict spectral covering morphism
with both $(\mc{R}_i,m_i)$ ($i=1,2$) are semi-standard, then $\phi$ is, after replacing
$(\mc{H}_1,\tau^{d_1})$ by an isomorphic object of $\mf{C}$ if necessary,
given by a (strict) algebra embedding
$\sigma:\mc{H}_2\hookrightarrow\mc{H}_1$ of $\mc{H}_2$ corresponding to an 
isogenous sub-root datum of $\mc{R}_1$ and $d_1d_2^{-1}\in\mathbb{Q}^\times$.
\end{cor}
\begin{proof}
This follows from the two results above in combination with (\ref{eq:ominv}).
\end{proof}
\begin{cor}
Given a spectral covering morphism $\phi:(\mc{H}_1,\tau^{d_1})\leadsto(\mc{H}_2,\tau^{d_2})$
where both $(\mc{R}_i,m_i)$ ($i=1,2$) are semi-standard,
there exists an admissible algebra isomorphism $\sigma:\mc{H}_2\leadsto\mc{H}\subset\mc{H}_1$
of $\mc{H}_2$ to a subalgebra of $\mc{H}_1$ whose embedding is a strict algebra morphism.
The branching of tempered representations of
$\mc{H}_1$ via $\sigma$ is a refinement of the spectral
correspondence of $\phi$.
\end{cor}
\subsubsection{Spectral covering morphisms to non-semi-standard affine Hecke algebras}
\label{par:nonstand}
In the previous paragraph we looked at covering morphisms
$\phi:(\mc{H}_1,\tau^{d_1})\leadsto(\mc{H}_2,\tau^{d_2})$ with
$(\mc{H}_i,\tau^{d_i})$ both semi-standard. We found that essentially such covering STMs 
correspond to admissible algebra embeddings $\mc{H}_2\subset\mc{H}_1$. 
And by Proposition \ref{prop:aut}, 
the endomorphisms in $\mf{C}_{es}$ of a semisimple, semi-standard normalized affine 
Hecke algebra are automorphisms.

In general things are a bit more complicated. For instance, if $\mc{H}$ 
is the group algebra of a lattice $X$ (normalized by $d=1$ say), 
then there obviously exist spectral transfer endomorphisms $\mc{H}\leadsto\mc{H}$ 
which are not invertible. And also, if $(X,R_0,Y,R_0^\vee)$ is a root datum with $R_0$ nonempty, 
and $T=\textup{Spec}(\mathbf{L}[X])$,  then the identity map 
$T\to T$ defines an STM from $\mc{H}=\mathbf{L}[X]$ to $\mc{H}':=\mathbf{L}[X\rtimes W_0]$ which is 
a spectral covering map. Observe that the target affine Hecke algebra $\mc{H}'$ is not 
semi-standard here, and that the embedding $\mc{H}\subset\mc{H}'$ goes ``the wrong way" 
compared to the case of semi-standard spectral covering maps.
 
The more serious examples of non-semistandard spectral covering morphisms discussed below 
are important for later applications in \cite{Opd4}. Consider 
\begin{equation}
\phi:\mc{H}(\mc{R}^{D }_{sc},m^D)\leadsto\mc{H}(\mc{R}^{C }_{sc},m^C)
\end{equation}
where $\mc{R}^D_{sc}$ denotes the root datum of the simply connected cover of 
$\textup{SO}_{2n}$, i.e. the root datum with $R_0=R_0^D$ of type $\textup{D}_n$
($n\geq 3$), and $X$ the root lattice of type $\textup{D}_n$, and $\mc{R}^C_{sc}$ 
is the root datum with root system $R_0^C\supset R_0^D$ of type $\textup{C}_n$ 
containing $R_0^D$
as its set of short roots, with the same $X$, now viewed as
the root lattice of type $\textup{C}_n$. Here we have defined the co-character 
$m^D$ by $m^D_S(s)=1$ and $m^C$ by $m^C_S(s)=1$ if $s$ is a simple reflection
of a simple root in $F_0^D\cap F_0^C$, and $m^C_S(s)=0$ else.
We normalized the traces of these algebras by $1$.

Then $\mc{H}(\mc{R}^{C }_{sc},m^C)$ is
represented by the non-standard spectral diagram of Figure \ref{nonstand}.
\input{nonstand.TpX}
We define $\phi$ by $\phi_T=\textup{id}_T$, where $T$ is the torus
with character lattice $X$. This is obviously an STM.
Notice however that although $\phi_T$ is an isomorphism, $\phi$ itself is not.

On the other hand there exists an admissible algebra embedding
\begin{equation}
\sigma:\mc{H}(\mc{R}^{D }_{sc},m^D)\to\mc{H}(\mc{R}^{C }_{sc},m^C)
\end{equation}
mapping $\mc{H}(\mc{R}^{\textup{D} }_{sc},m^D)$ onto the subalgebra of
$\mc{H}(\mc{R}^{C }_{sc},m^C)$ which is invariant for the algebra
involution $\sigma^{\mb{c}}$ where $\mb{c}$ is the class of reflections
of the short roots of $\textup{B}_n=\textup{C}_n^\vee$.
In fact, one can define $\sigma$ on the
generators as follows. Assume that
\begin{equation}
F^D_0=(\alpha_1=e_1-e_2,\dots,\alpha_{n-1}=e_{n-1}-e_n,\alpha_n=e_{n-1}+e_n)
\end{equation}
and
\begin{equation}
F^C_0=(\alpha_1^\prime=e_1-e_2,\dots,\alpha_{n-1}^\prime
=e_{n-1}-e_n,\alpha_n^\prime=2e_n)
\end{equation}
Then we may define
$\sigma(N_i)=N_i^\prime$ for $i=0,\dots,n-1$ and $\sigma(N_n)=N_n^\prime N_{n-1}^\prime N_n^\prime$.
Using that $(N^\prime_n)^2=1$ one checks easily that this defines an admissible algebra homomorphism
whose image is the invariant algebra of the algebra involution
$\sigma^\mb{c}$ of $\mc{H}(\mc{R}^{C }_{sc},m^C)$(given by
$\sigma^\mb{c}(N_i^\prime)=N_i^\prime$ for $i=0,\dots,n-1$ and $\sigma^\mb{c}(N_n^\prime)=-N_n^\prime$).

Then $\sigma$ is compatible with $\phi$ in the sense that $\sigma$ induces the identity
on the lattice $X$. Hence, as before, the branching correspondence of this embedding
does refine the spectral correspondence of $\phi$.
Observe however that this admissible algebra embedding goes in the
``wrong way'' compared to the situation of spectral coverings of (semi)-standard
Hecke algebras.

We can combine the algebra embedding $\sigma$ with the strict
algebra embedding of $\mc{H}(\mc{R}^{C }_{sc},m^C)$ as a subalgebra of index two
in the algebra $\mc{H}(\textup{C}_n^{(1)},m^0)$, 
corresponding to the embedding of the root lattice the type $\textup{D}_n$  root system 
as a sub lattice of index two in $\mathbb{Z}^n$, and where $m^0$ has $m_-^0=m_+^0=0$ 
(and on the long roots $\alpha$ of $\textup{B}_n$ we have $m^0(\alpha)=1$). 
This amounts to removing the arrow in the spectral diagram
of Figure \ref{nonstand}, and corresponds to a $2:1$ covering STM 
$\mc{H}(\textup{C}_n^{(1)},m^0)\leadsto \mc{H}(\mc{R}^{C }_{sc},m^C)$. 
Altogether we have an admissible algebra embedding
of $\mc{H}(\mc{R}^D_{sc},m^D)$ into $\mc{H}(\textup{C}_n^{(1)},m^0)$ as subalgebra
of index $4$. Notice that the underlying morphism of tori does not represent 
a spectral covering morphism. 

On the other hand, $\mc{H}(\mc{R}^D_{sc},m^D)$ is also embedded as 
subalgebra of index $2$ into the algebra $\mc{H}(\mc{R}^D_{\mathbb{Z}^n},m^D)$, 
by the same embedding of the root lattice of type $\textup{D}_n$ as index two sub lattice of 
$\mathbb{Z}^n$, corresponding to the standard covering STM 
$\mc{H}(\mc{R}^D_{\mathbb{Z}^n},m^D)\leadsto\mc{H}(\mc{R}^D_{sc},m^D)$.
We may complete the diagram to obtain a commuting square of STMs 
by adding the semistandard STM $\mc{H}(\mc{R}^D_{\mathbb{Z}^n},m^D)\leadsto \mc{H}(\textup{C}_n^{(1)},m^0)$ 
represented by $\textup{id}_T$. In turn $\mc{H}(\mc{R}^D_{\mathbb{Z}^n},m^D)$ is contained in 
the maximally extended affine Hecke algebra $\mc{H}(\mc{R}^D_{ad},m^D)$ of type $\textup{D}_n$
as an index two admissible subalgebra,  corresponding to a standard STM 
$\mc{H}(\mc{R}^D_{ad},m^D)\leadsto\mc{H}(\mc{R}^D_{\mathbb{Z}^n},m^D)$. 
Finally we add into the picture the extended affine Hecke algebra $\mc{H}(\mc{R}^B_{ad},m^B)$ 
with $R_0$ of type $\textup{B}_n$, and $X$ equal to the weight lattice $R_0$ (which equals the 
weight lattice of type $\textup{D}_n$), and where $m^{B}(s)=0$ if $s$ is the simple reflection corresponding to 
the short simple root of $R_0$.
We have admissible algebra embeddings of index two of the form 
$\mc{H}(\textup{C}_n^{(1)},m^0)\subset \mc{H}(\mc{R}^B_{ad},m^B)$ and 
$\mc{H}(\mc{R}^D_{ad},m^D)\subset \mc{H}(\mc{R}^B_{ad},m^B)$ corresponding to STMs 
$ \mc{H}(\mc{R}^B_{ad},m^B)\leadsto\mc{H}(\textup{C}_n^{(1)},m^0)$ and 
$\mc{H}(\mc{R}^D_{ad},m^D)\leadsto \mc{H}(\mc{R}^B_{ad},m^B)$, 
which completes the diagram as follows:
%\begin{equation}
%\begin{CD}
%\mc{H}(\mc{R}^D_{ad},m^D)@>{}>>\mc{H}(\mc{R}^D_{\mathbb{Z}^n},m^D)@>{}>>\mc{H}(\mc{R}^D_{sc},m^D)\\
%@V{}VV                                                  @V{}VV                                                @VV{}V\\
%\mc{H}(\mc{R}^B_{ad},m^B)@>>{}>\mc{H}(\textup{C}_n^{(1)},m^0)@>>{}>\mc{H}(\mc{R}^C_{sc},m^C)
%\end{CD}
%\end{equation}
\begin{equation}\label{eq:cd}
\begin{tikzcd}
\mc{H}(\mc{R}^D_{ad},m^D)\arrow[rightsquigarrow]{r}{\supset}\arrow[rightsquigarrow]{d}{\cap}&\mc{H}(\mc{R}^D_{\mathbb{Z}^n},m^D)
\arrow[rightsquigarrow]{r}{\supset}\arrow[rightsquigarrow]{d}{\cap}&\mc{H}(\mc{R}^D_{sc},m^D)
\arrow[rightsquigarrow]{d}{\cap}\\
\mc{H}(\mc{R}^B_{ad},m^B)\arrow[rightsquigarrow]{r}{\supset}&\mc{H}(\textup{C}_n^{(1)},m^0)
\arrow[rightsquigarrow]{r}{\supset}
&\mc{H}(\mc{R}^C_{sc},m^C)\\
\end{tikzcd}
\end{equation}
\subsection{Spectral transfer morphisms of rank $0$}
So far we have only considered STMs $\phi$
of co-rank $\textup{cork}(\phi)=0$. In the examples of this kind we
have seen only spectral correspondences arising from the branching
of tempered representations of $\mc{H}$ to admissibly
embedded subalgebras of $\mc{H}^\prime\subset \mc{H}$.
At the opposite other end we have
STMs $\phi$ with $\textup{rk}(\phi)=0$, and
these have a very different behaviour.
Let $\mc{H}^0=\mb{L}$ denote the rank $0$ affine Hecke algebra.
\begin{prop}\label{prop:rk0}
Let $(\mc{H},\tau^d)$ be an arbitrary normalized affine Hecke algebra with
$\mc{H}=\mc{H}(\mc{R},m)$. Let $r\in\textup{Res}(\mc{R},m)$ be
a generic residual point for $(\mc{R},m)$. Let 
$\lambda\in\mathbb{Q}^\times$, and 
define $d^0\in\mb{K}^\times$ by
\begin{equation}\label{eq:cuspdeg}
d^0(\mb{v})=\lambda\mu^{(\{r\})}(\mb{v},r(\mb{v}))
\end{equation}
Then
$d^0\in\mb{M}$, and $\phi_T=r$ defines an STM
$\phi:(\mc{H}^0,\tau^{d^0})\leadsto(\mc{H},\tau^d)$ with $\textup{rk}(\phi)=0$.
Conversely, all STMs $\phi$ to $\mc{H}$ with
$\textup{rk}(\phi)=0$ are of this form.
\end{prop}
\begin{proof}
The fact that $d^0\in\mb{M}$ is clear from Theorem \ref{thm:ratfdeg}.
The fact that $\phi_T=r$ defines an STM is
obvious from the definitions.
\end{proof}
Hence for any $d^0\in\mb{M}$ we have (with $\textup{Res}=\textup{Res}(\mc{R},m)$):
\begin{equation}\label{eq:rk0}
\textup{Hom}_\mf{C}((\mc{H}^0,\tau^{d^0}),(\mc{H},\tau^d))
=\{W_0r\in W_0\backslash\textup{Res}\mid
\exists\lambda\in\mathbb{Q}^\times:d^0=\lambda\mu^{(\{r\})}\}
\end{equation}
and given $\Phi\in\textup{Hom}_\mf{C}((\mc{H}^0,\tau^{d^0}),(\mc{H},\tau^d))$
defined by $r\in\textup{Res}(\mc{R},m)$, the spectral correspondence
$\mf{S}_{12}(\Phi)$ of $\Phi$ is given by (with $\Delta=\Delta(\mc{R},m)$):
\begin{equation}\label{eq:corrk0}
\mf{S}_{12}(\Phi)=\{(\delta^0,\delta)\in\Delta^0\times\Delta\mid gcc(\delta)=W_0r\}
\end{equation}
where $\delta^0\in\Delta^0$ denotes the $\mathbb{R}_{>1}$-family
of characters of $\mc{H}^0=\mb{L}$ defined by $\delta^0_\mb{v}(1)=1$.
In particular, Theorem \ref{thm:corr} reduces in this
special case to Theorem \ref{thm:ratfdeg}(ii).
\begin{rem}
If the rank of $\mc{H}$ is positive and
$\Phi\in\textup{Hom}_\mf{C}((\mc{H}^0,\tau^{d^0}),(\mc{H},\tau^d))$ then
the spectral correspondence $\mf{S}_{12}(\Phi)$ is obviously
not associated to a branching law of a morphism of $\mb{L}$-algebras.
\end{rem}
\subsection{General spectral transfer morphisms}
\subsubsection{Rigidity in $\mf{C}$ and $\mf{C}_{es}$}
The following result states a fundamental rigidity property of general semi-standard 
morphisms of semi-simple objects of category $\mf{C}$, and is a strengthening of Proposition 
\ref{prop:aut}. The proof is similar, and is left to the reader.
\begin{prop}\label{prop:rigid}
A spectral transfer morphism $\Phi: (\mc{H}_1,\tau^{d_1})\leadsto(\mc{H}_2,\tau^{d_2})$ 
with $(\mc{H}_1,\tau^{d_1})$ semi-simple and semi-standard is essentially determined by its image: 
If $\Phi_1,\Phi_2\in\textup{Hom}_\mf{C}((\mc{H}_1,\tau^{d_1}),(\mc{H}_2,\tau^{d_2}))$ and
$\textup{Im}(\Phi_1)=\textup{Im}(\Phi_2)$ then there exists a
$\gamma\in\textup{Aut}_\mf{C}(\mc{H}_1,\tau^{d_1})$ such that $\Phi_2=\Phi_1\circ\gamma$.
Similarly for essentially strict STMs.
\end{prop}
\subsubsection{Examples of general STMs of intermediate rank}
We refer the reader to \cite{Opd4}. There we describe the structure of the spectral 
transfer category of unipotent affine Hecke algebras in detail, where 
STMs of intermediate rank are abundant. 
\section{A partial ordering of normalized affine Hecke algebras}\label{subsub:isogpo}
In this final section we define a partial ordering on the set of certain equivalence classes 
(called spectral isogeny classes) of 
normalized affine Hecke algebras.
These notions will play an important role in the study of STMs of unipotent 
affine Hecke algebras \cite{{Opd4}}.
\begin{defn}\label{defn:isog}
We say that $(\mc{H}_1,d_1)$ and $(\mc{H}_2,d_2)$ are \emph{spectrally isogenous} 
if there exist essentially strict STMs $\phi:(\mc{H}_1,d_1)\leadsto(\mc{H}_2,d_2)$ and 
$\psi:(\mc{H}_2,d_2)\leadsto(\mc{H}_1,d_1)$.
\end{defn}
Clearly this is an equivalence relation. 
Let us denote by $[(\mc{H},d)]$ the spectral isogeny class of a normalized affine Hecke algebra. 
There exists a canonical partial ordering on the set of spectral isogeny classes:
\begin{defn}
We say that $[(\mc{H}_2,\tau_2)]$ is \emph{lower} than $[(\mc{H}_1,\tau_1)]$, notation $[(\mc{H}_2,,\tau_2)]
\gtrsim[(\mc{H}_1,\tau_1)]$
(or equivalently, $[(\mc{H}_1,,\tau_1)]$ is \emph{higher} than $[(\mc{H}_2,\tau_2)]$, denoted by 
$[(\mc{H}_1,\tau_1)]\lesssim[(\mc{H}_2,\tau_2)]$)  
if there exists a spectral transfer morphism $\phi:(\mc{H}_1,\tau_1)\leadsto (\mc{H}_2,\tau_2)$. 
This defines a partial ordering on the set of spectral isogeny classes of normalized affine Hecke algebras.
\end{defn}
In most cases, the underlying affine Hecke algebras of the normalized affine Hecke algebras in a 
spectral isogeny class are isomorphic as affine Hecke algebras. 
Exceptions to this are of a trivial nature, caused by the fact that we do not assume that 
normalized affine Hecke algebras under consideration are semisimple and semi standard. 
The spectral isogeny classes of Lusztig's unipotent affine Hecke algebras as they arise   
in the theory of unramified almost simple algebraic groups defined over a local field 
are always simply equal to  their isomorphism classes
(see Proposition \ref{prop:rigid1} below). We will just write $\mc{H}_1\lesssim\mc{H}_2$
instead of $[(\mc{H}_1,,\tau_1)]\lesssim[(\mc{H}_2,\tau_2)]$.
\begin{prop}\label{prop:rigid1}
Let $(\mc{H}_1,d_1)$ and $(\mc{H}_2,d_2)$ be \emph{spectrally isogenous} normalized affine 
Hecke algebras with essentially strict STMs $\phi_1:(\mc{H}_1,d_1)\leadsto(\mc{H}_2,d_2)$ and 
$\phi_2:(\mc{H}_2,d_2)\leadsto(\mc{H}_1,d_1)$. Then $\phi_1$ and $\phi_2$ are spectral 
covering maps, and if $\mc{H}_1$ or $\mc{H}_2$ is semisimple and semi-standard then 
$\phi_i$ ($i=1,2$) both are essentially strict spectral isomorphisms. Hence the spectral isogeny 
classes of semisimple, semi-standard normalized affine Hecke algebras are simply their 
isomorphisms classes in $\mf{C}_{es}$. The same holds true for  
$\mc{H}(\mc{R}^B_{ad},m^B)$, $\mc{H}(\textup{C}_n^{(1)},m^0)$, and  
$\mc{H}(\mc{R}^{C}_{sc},m^C)$. 
\end{prop}
\begin{proof}
This is an easy consequence of the fact that for semisimple, semi-standard 
affine Hecke algebras, spectral transfer endomorphisms are 
in fact automorphisms, according to Proposition \ref{prop:aut}. For the 
listed non-semi-standard cases the spectral endomorphisms are also 
automorphisms, as follow easily from (\ref{eq:cd}). 
\end{proof}
{\bf Nota bene}: The notion of spectral isogeny of normalized affine Hecke algebras is obviously very 
different from the notion of isogeny of the underlying root data, as we saw in Proposition \ref{prop:rigid1}.
The main reason for introducing this notion is the importance for unipotent affine Hecke algebras 
(in \cite{Opd4})  
of the partial ordering on the spectral isogeny 
classes.  
Also observe from (\ref{eq:cd}) that this partial ordering 
is not related to embeddings of algebras.

\end{document}